\documentclass[12pt,reqno]{amsart}
\expandafter\ifx\csname ProvidesPackage\endcsname\relax\toks0=\expandafter{%
\fi\ProvidesPackage{diagrams}[2024/11/20 v3.97 Paul Taylor's commutative
diagrams]
\toks0=\bgroup}
\ifx\diagram\isundefined\else\expandafter\endinput
\fi

\edef\cdrestoreat{
\noexpand\catcode`\noexpand\@=\the\catcode`\@
\noexpand\catcode`\noexpand\#=\the\catcode`\#
\noexpand\catcode`\noexpand\$=\the\catcode`\$
\noexpand\catcode`\noexpand\<=\the\catcode`\<
\noexpand\catcode`\noexpand\>=\the\catcode`\>
\noexpand\catcode`\noexpand\:=\the\catcode`\:
\noexpand\catcode`\noexpand\;=\the\catcode`\;
\noexpand\catcode`\noexpand\!=\the\catcode`\!
\noexpand\catcode`\noexpand\?=\the\catcode`\?
\noexpand\catcode`\noexpand\+=\the\catcode'53
}\catcode`\@=11 \catcode`\#=6 \catcode`\<=12 \catcode`\>=12 \catcode'53=12
\catcode`\:=12 \catcode`\;=12 \catcode`\!=12 \catcode`\?=12

\ifx\diagram@help@messages\CD@qK\let\diagram@help@messages y\fi

\def\cdps@Rokicki#1{\special{ps:#1}}\let\cdps@dvips\cdps@Rokicki\let
\cdps@RadicalEye\cdps@Rokicki\let\CD@HB\cdps@Rokicki\let\CD@IK\cdps@Rokicki
\let\CD@HB\cdps@Rokicki
\def\cdps@Bechtolsheim#1{\special{dvitps: Literal "#1"}}%
\let\cdps@dvitps\cdps@Bechtolsheim\let\cdps@IntegratedComputerSystems
\cdps@Bechtolsheim
\def\cdps@Clark#1{\special{dvitops: inline #1}}
\let\cdps@dvitops\cdps@Clark
\let\cdps@OzTeX\empty\let\cdps@oztex\empty\let\cdps@Trevorrow\empty
\def\cdps@Coombes#1{\special{ps-string #1}}

\count@=\year\multiply\count@12 \advance\count@\month
\ifnum\count@>24382 
\ifnum\count@>24385 
\errhelp{You may press RETURN and carry on for the time being.}\errmessage{! remain
compatible with your files. Please get it NOW.}\fi\fi

\def\CD@DE{\global\let}\def\CD@RH{\outer\def}

{\escapechar\m@ne\xdef\CD@o{\string\{}\xdef\CD@yC{\string\}}
\catcode`\&=4 \CD@DE\CD@Q=&\xdef\CD@S{\string\&}
\catcode`\$=3 \CD@DE\CD@k=$\CD@DE\CD@ND=$
\xdef\CD@nC{\string\$}\gdef\CD@LG{$$}
\catcode`\_=8 \CD@DE\CD@lJ=_
\obeylines\catcode`\^=7 \CD@DE\@super=^
\ifnum\newlinechar=10 \gdef\CD@uG{^^J}
\else\ifnum\newlinechar=13 \gdef\CD@uG{^^M}
\else\ifnum\newlinechar=-1 \gdef\CD@uG{^^J}
\else\CD@DE\CD@uG\space\expandafter\message{! input error: \noexpand
\newlinechar\space is ASCII \the\newlinechar, not LF=10 or CR=13.}
\fi\fi\fi}

\mathchardef\lessthan='30474 \mathchardef\greaterthan='30476

\ifx\tenln\CD@qK
\font\tenln=line10\relax
\fi\ifx\tenlnw\CD@qK\ifx\tenln\nullfont\let\tenlnw\nullfont\else
\font\tenlnw=linew10\relax
\fi\fi

\ifx\inputlineno\CD@qK\csname newcount\endcsname\inputlineno\inputlineno\m@ne
\fi

\def\cd@shouldnt#1{\CD@KB{* THIS (#1) SHOULD NEVER HAPPEN! *}}

\def\get@round@pair#1(#2,#3){#1{#2}{#3}}
\def\get@square@arg#1[#2]{#1{#2}}
\def\CD@AE#1{\CD@PK\let\CD@DH\CD@@E\CD@@E#1,],}
\def\CD@m{[}\def\CD@RD{]}\def\commdiag#1{{\let\enddiagram\relax\diagram[]#1%
\enddiagram}}

\def\CD@BF{{\ifx\CD@EH[\aftergroup\get@square@arg\aftergroup\CD@YH\else
\aftergroup\CD@JH\fi}}
\def\CD@CF#1#2{\def\CD@YH{#1}\def\CD@JH{#2}\futurelet\CD@EH\CD@BF}

\def\CD@KK{|}

\def\CD@PB{
\tokcase\CD@DD:\CD@y\break@args;\catcase\@super:\upper@label;\catcase\CD@lJ:%
\lower@label;\tokcase{~}:\middle@label;
\tokcase<:\CD@iF;
\tokcase>:\CD@iI;
\tokcase(:\CD@BC;
\tokcase[:\optional@;
\tokcase.:\CD@JJ;
\catcase\space:\eat@space;\catcase\bgroup:\positional@;\default:\CD@@A
\break@args;\endswitch}

\def\switch@arg{
\catcase\@super:\upper@label;\catcase\CD@lJ:\lower@label;\tokcase[:\optional@
;
\tokcase.:\CD@JJ;
\catcase\space:\eat@space;\catcase\bgroup:\positional@;\tokcase{~}:%
\middle@label;
\default:\CD@y\break@args;\endswitch}

\let\CD@tJ\relax\ifx\protect\CD@qK\let\protect\relax\fi\ifx\AtEndDocument
\CD@qK\def\CD@PG{\CD@gB}\def\CD@GF#1#2{}\else\def\CD@PG#1{\edef\CD@CH{#1}%
\expandafter\CD@oC\CD@CH\CD@OD}\def\CD@oC#1\CD@OD{\AtEndDocument{\typeout{%
\CD@tA: #1}}}\def\CD@GF#1#2{\gdef#1{#2}\AtEndDocument{#1}}\fi\def\CD@ZA#1#2{%
\def#1{\CD@PG{#2\CD@mD\CD@W}\CD@DE#1\relax}}\def\CD@uF#1\repeat{\def\CD@p{#1}%
\CD@OF}\def\CD@OF{\CD@p\relax\expandafter\CD@OF\fi}\def\CD@sF#1\repeat{\def
\CD@q{#1}\CD@PF}\def\CD@PF{\CD@q\relax\expandafter\CD@PF\fi}\def\CD@tF#1%
\repeat{\def\CD@r{#1}\CD@QF}\def\CD@QF{\CD@r\relax\expandafter\CD@QF\fi}\def
\CD@tG#1#2#3{\def#2{\let#1\iftrue}\def#3{\let#1\iffalse}#3}\if y%
\diagram@help@messages\def\CD@rG#1#2{\csname newtoks\endcsname#1#1=%
\expandafter{\csname#2\endcsname}}\else\csname newtoks\endcsname\no@cd@help
\no@cd@help{See the manual}\def\CD@rG#1#2{\let#1\no@cd@help}\fi\chardef\CD@lF
=1 \chardef\CD@lI=2 \chardef\CD@MH=5 \chardef\CD@tH=6 \chardef\CD@sH=7
\chardef\CD@PC=9 \dimendef\CD@hI=2 \dimendef\CD@hF=3 \dimendef\CD@mF=4
\dimendef\CD@mI=5 \dimendef\CD@wJ=6 \dimendef\CD@tI=8 \dimendef\CD@sI=9
\skipdef\CD@uB=1 \skipdef\CD@NF=2 \skipdef\CD@tB=3 \skipdef\CD@ZE=4 \skipdef
\countdef\CD@JC=9 \countdef\CD@gD=8 \countdef\CD@A=7 \def\sdef#1#2{\def#1{#2}%
}\def\CD@L#1{\expandafter\aftergroup\csname#1\endcsname}\def\CD@RC#1{%
\expandafter\def\csname#1\endcsname}\def\CD@sD#1{\expandafter\gdef\csname#1%
\endcsname}\def\CD@vC#1{\expandafter\edef\csname#1\endcsname}\def\CD@nF#1#2{%
\expandafter\let\csname#1\expandafter\endcsname\csname#2\endcsname}\def\CD@EE
#1#2{\expandafter\CD@DE\csname#1\expandafter\endcsname\csname#2\endcsname}%
\def\CD@AK#1{\csname#1\endcsname}\def\CD@XJ#1{\expandafter\show\csname#1%
\endcsname}\def\CD@ZJ#1{\expandafter\showthe\csname#1\endcsname}\def\CD@WJ#1{%
\expandafter\showbox\csname#1\endcsname}\def\CD@tA{Commutative Diagram}\edef
\CD@kH{\string\par}\edef\CD@dC{\string\diagram}\edef\CD@HD{\string\enddiagram
}\edef\CD@EC{\string\\}\def\CD@eF{LaTeX}\ifx\@ignoretrue\CD@qK\expandafter
\CD@tG\csname if@ignore\endcsname\ignore@true\ignore@false\def\@ignoretrue{%
\global\ignore@true}\def\@ignorefalse{\global\ignore@false}\fi

\def\CD@g{{\ifnum0=`}\fi}\def\CD@wC{\ifnum0=`{\fi}}\def\catcase#1:{\ifcat
\noexpand\CD@EH#1\CD@tJ\expandafter\CD@kC\else\expandafter\CD@dJ\fi}\def
\tokcase#1:{\ifx\CD@EH#1\CD@tJ\expandafter\CD@kC\else\expandafter\CD@dJ\fi}%
\def\CD@kC#1;#2\endswitch{#1}\def\CD@dJ#1;{}\let\endswitch\relax\def\default:%
#1;#2\endswitch{#1}\ifx\at@\CD@qK\def\at@{@}\fi\edef\CD@P{\CD@o pt\CD@yC}%
\CD@RC{\CD@P>}#1>#2>{\CD@z\rTo\sp{#1}\sb{#2}\CD@z}\CD@RC{\CD@P<}#1<#2<{\CD@z
\lTo\sp{#1}\sb{#2}\CD@z}\CD@RC{\CD@P)}#1)#2){\CD@z\rTo\sp{#1}\sb{#2}\CD@z}%
\CD@RC{\CD@P(}#1(#2({\CD@z\lTo\sp{#1}\sb{#2}\CD@z}
\def\CD@O{\def\endCD{\enddiagram}\CD@RC{\CD@P A}##1A##2A{\uTo<{##1}>{##2}%
\CD@z\CD@z}\CD@RC{\CD@P V}##1V##2V{\dTo<{##1}>{##2}\CD@z\CD@z}\CD@RC{\CD@P=}{%
\CD@z\hEq\CD@z}\CD@RC{\CD@P\CD@KK}{\vEq\CD@z\CD@z}\CD@RC{\CD@P\string\vert}{%
\vEq\CD@z\CD@z}\CD@RC{\CD@P.}{\CD@z\CD@z}\let\CD@z\CD@Q}\def\CD@IE{\let\tmp
\CD@JE\ifcat A\noexpand\CD@CH\else\ifcat=\noexpand\CD@CH\else\ifcat\relax
\noexpand\CD@CH\else\let\tmp\at@\fi\fi\fi\tmp}\def\CD@JE#1{\CD@nF{tmp}{\CD@P
\string#1}\ifx\tmp\relax\def\tmp{\at@#1}\fi\tmp}\def\CD@z{}\begingroup
\aftergroup\def\aftergroup\CD@T\aftergroup{\aftergroup\def\catcode`\@\active
\aftergroup @\endgroup{\futurelet\CD@CH\CD@IE}}\def\CD@uK#1{\CD@nF{x}{tex_#1:%
D}\ifx\x\relax\else\CD@nF{#1}{x}\fi} \def\CD@vK{\CD@uK{par}\CD@uK{everypar}%
\CD@uK{noindent}\CD@uK{parskip}\CD@uK{unskip}\CD@uK{hskip}\CD@uK{indent}}

\newcount\CD@uA\newcount\CD@vA\newcount\CD@wA\newcount\CD@xA\newdimen\CD@OA
\newdimen\CD@PA\CD@tG\CD@gE\CD@@A\CD@y\CD@tG\CD@hE\CD@EA\CD@BA\newdimen\CD@RA
\newdimen\CD@SA\newcount\CD@yA\newcount\CD@zA\newdimen\CD@QA\newbox\CD@DA
\CD@tG\CD@lE\CD@dA\CD@bA\newcount\CD@LH\newcount\CD@TC\def\CD@V#1#2{\ifdim#1<%
#2\relax#1=#2\relax\fi}\def\CD@X#1#2{\ifdim#1>#2\relax#1=#2\relax\fi}%
\newdimen\CD@XH\CD@XH=1sp \newdimen\CD@zC\CD@zC\z@\def\CD@cJ{\ifdim\CD@zC=1em%
\else\CD@nJ\fi}\def\CD@nJ{\CD@zC1em\def\CD@NC{\fontdimen8\textfont3 }\CD@@J
\CD@NJ\setbox0=\vbox{\CD@t\noindent\CD@k\null\penalty-9993\null\CD@ND\null
\endgraf\setbox0=\lastbox\unskip\unpenalty\setbox1=\lastbox\global\setbox
\CD@IG=\hbox{\unhbox0\unskip\unskip\unpenalty\setbox0=\lastbox}\global\setbox
\CD@KG=\hbox{\unhbox1\unskip\unpenalty\setbox1=\lastbox}}}\newdimen\CD@@I
\CD@@I=1true in \divide\CD@@I300 \def\CD@zH#1{\multiply#1\tw@\advance#1\ifnum
#1<\z@-\else+\fi\CD@@I\divide#1\tw@\divide#1\CD@@I\multiply#1\CD@@I}\def
\MapBreadth{\afterassignment\CD@gI\CD@LF}\newdimen\CD@LF\newdimen\CD@oI\def
\CD@gI{\CD@oI\CD@LF\CD@V\CD@@I{4\CD@XH}\CD@X\CD@@I\p@\CD@zH\CD@oI\ifdim\CD@LF
>\z@\CD@V\CD@oI\CD@@I\fi\CD@cJ}\def\CD@RJ#1{\CD@zD\count@\CD@@I#1\ifnum
\count@>\z@\divide\CD@@I\count@\fi\CD@gI\CD@NJ}\def\CD@NJ{\dimen@\CD@QC
\count@\dimen@\divide\count@5\divide\count@\CD@@I\edef\CD@OC{\the\count@}}%
\def\CD@AJ{\CD@QJ\z@}\def\CD@QJ#1{\CD@tI\axisheight\advance\CD@tI#1\relax
\advance\CD@tI-.5\CD@oI\CD@zH\CD@tI\CD@sI-\CD@tI\advance\CD@tI\CD@LF}%
\newdimen\CD@DC\CD@DC\z@\newdimen\CD@eJ\CD@eJ\z@\def\CD@CJ#1{\CD@sI#1\relax
\CD@tI\CD@sI\advance\CD@tI\CD@LF\relax}\def\horizhtdp{height\CD@tI depth%
\CD@sI}\def\axisheight{\fontdimen22\the\textfont\tw@}\def\script@axisheight{%
\fontdimen22\the\scriptfont\tw@}\def\ss@axisheight{\fontdimen22\the
\scriptscriptfont\tw@}\def\CD@NC{0.4pt}\def\CD@VK{\fontdimen3\textfont\z@}%
\def\CD@UK{\fontdimen3\textfont\z@}\newdimen\PileSpacing\newdimen\CD@nA\CD@nA
\z@\def\CD@RG{\ifincommdiag1.3em\else2em\fi}\newdimen\CD@YB\def\CellSize{%
\afterassignment\CD@kB\DiagramCellHeight}\newdimen\DiagramCellHeight
\DiagramCellHeight-\maxdimen\newdimen\DiagramCellWidth\DiagramCellWidth-%
\maxdimen\def\CD@kB{\DiagramCellWidth\DiagramCellHeight}\def\CD@QC{3em}%
\newdimen\MapShortFall\def\MapsAbut{\MapShortFall\z@\objectheight\z@
\objectwidth\z@}\newdimen\CD@iA\CD@iA\z@\CD@tG\CD@vE\CD@aB\CD@ZB\expandafter
\ifx\expandafter\iftrue\csname ifUglyObsoleteDiagrams\endcsname\CD@ZB\else
\CD@aB\fi\CD@nF{ifUglyObsoleteDiagrams}{relax}\newif\ifUglyObsoleteDiagrams
\def\CD@nK{\CD@aB\UglyObsoleteDiagramsfalse}\def\CD@oK{\CD@ZB
\UglyObsoleteDiagramstrue}\CD@vE\CD@nK\else\CD@oK\fi\CD@tG\CD@hK\CD@dK\CD@cK
\CD@cK\def\CD@sK{\ifx\pdfoutput\CD@qK\else\ifx\pdfoutput\relax\else\ifnum
\pdfoutput>\z@\CD@pK\fi\fi\fi} \def\CD@pK{\global\CD@dK\global\CD@aB\global
\UglyObsoleteDiagramsfalse\global\let\CD@n\empty\global\let\CD@oK\relax
\global\let\CD@pK\relax\global\let\CD@sK\relax}\def\CD@tK#1{}\ifx\pdfliteral\CD@qK\else\ifx\pdfliteral\relax\else\let\CD@tK
\pdfliteral\fi\fi\ifx\XeTeXrevision\CD@qK\else\ifx\XeTeXrevision\relax\else
\ifdim\XeTeXrevision pt<.996pt \expandafter\message{! XeTeX version
\XeTeXrevision\space does not support PDF literals, so diagonals will not work%
!}\else\expandafter\message{RUNNING UNDER XETEX \XeTeXrevision}\CD@pK\fi\fi
\fi\CD@sK\def\newarrowhead{\CD@mG h\CD@BG\CD@GG>}\def\newarrowtail{\CD@mG t%
\CD@BG\CD@GG>}\def\newarrowmiddle{\CD@mG m\CD@BG\hbox@maths\empty}\def
\newarrowfiller{\CD@mG f\CD@bE\CD@MK-}\def\CD@mG#1#2#3#4#5#6#7#8#9{\CD@RC{r#1%
:#5}{#2{#6}}\CD@RC{l#1:#5}{#2{#7}}\CD@RC{d#1:#5}{#3{#8}}\CD@RC{u#1:#5}{#3{#9}%
}\CD@vC{-#1:#5}{\expandafter\noexpand\csname-#1:#4\endcsname\noexpand\CD@MC}%
\CD@vC{+#1:#5}{\expandafter\noexpand\csname+#1:#4\endcsname\noexpand\CD@MC}}%
\CD@ZA\CD@MC{\CD@eF\space diagonals are used unless PostScript is set}\def
\defaultarrowhead#1{\edef\CD@sJ{#1}\CD@@J}\def\CD@@J{\CD@IJ\CD@sJ<>ht\CD@IJ
\CD@sJ<>th}\def\CD@IJ#1#2#3#4#5{\CD@HJ{r#4}{#3}{l#5}{#2}{r#4:#1}\CD@HJ{r#5}{#%
2}{l#4}{#3}{l#4:#1}\CD@HJ{d#4}{#3}{u#5}{#2}{d#4:#1}\CD@HJ{d#5}{#2}{u#4}{#3}{u%
#4:#1}}\def\CD@HJ#1#2#3#4#5{\begingroup\aftergroup\CD@GJ\CD@L{#1+:#2}\CD@L{#1%
:#2}\CD@L{#3:#4}\CD@L{#5}\endgroup}\def\CD@GJ#1#2#3#4{\csname newbox%
\endcsname#1\def#2{\copy#1}\def#3{\copy#1}\setbox#1=\box\voidb@x}\def\CD@sJ{}%
\CD@@J\def\CD@GJ#1#2#3#4{\setbox#1=#4}\ifx\tenln\nullfont\def\CD@sJ{vee}\else
\let\CD@sJ\CD@eF\fi\def\CD@xF#1#2#3{\begingroup\aftergroup\CD@wF\CD@L{#1#2:#3%
#3}\CD@L{#1#2:#3}\aftergroup\CD@yF\CD@L{#1#2:#3-#3}\CD@L{#1#2:#3}\endgroup}%
\def\CD@wF#1#2{\def#1{\hbox{\rlap{#2}\kern.4\CD@zC#2}}}\def\CD@yF#1#2{\def#1{%
\hbox{\rlap{#2}\kern.4\CD@zC#2\kern-.4\CD@zC}}}\CD@xF lh>\CD@xF rt>\CD@xF rh<%
\CD@xF rt<\def\CD@yF#1#2{\def#1{\hbox{\kern-.4\CD@zC\rlap{#2}\kern.4\CD@zC#2}%
}}\CD@xF rh>\CD@xF lh<\CD@xF lt>\CD@xF lt<\def\CD@wF#1#2{\def#1{\vbox{\vbox to%
\z@{#2\vss}\nointerlineskip\kern.4\CD@zC#2}}}\def\CD@yF#1#2{\def#1{\vbox{%
\vbox to\z@{#2\vss}\nointerlineskip\kern.4\CD@zC#2\kern-.4\CD@zC}}}\CD@xF uh>%
\CD@xF dt>\CD@xF dh<\CD@xF dt<\def\CD@yF#1#2{\def#1{\vbox{\kern-.4\CD@zC\vbox
to\z@{#2\vss}\nointerlineskip\kern.4\CD@zC#2}}}\CD@xF dh>\CD@xF ut>\CD@xF uh<%
\CD@xF ut<\def\CD@BG#1{\hbox{\mathsurround\z@\offinterlineskip\CD@k\mkern-1.5%
mu{#1}\mkern-1.5mu\CD@ND}}\def\hbox@maths#1{\hbox{\CD@k#1\CD@ND}}\def\CD@GG#1%
{\hbox to\CD@LF{\setbox0=\hbox{\offinterlineskip\mathsurround\z@\CD@k{#1}%
\CD@ND}\dimen0.5\wd0\advance\dimen0-.5\CD@oI\CD@zH{\dimen0}\kern-\dimen0%
\unhbox0\hss}}\def\CD@sB#1{\hbox to2\CD@LF{\hss\offinterlineskip\mathsurround
\z@\CD@k{#1}\CD@ND\hss}}\def\CD@vF#1{\hbox{\mathsurround\z@\CD@k{#1}\CD@ND}}%
\def\CD@bE#1{\hbox{\kern-.15\CD@zC\CD@k{#1}\CD@ND\kern-.15\CD@zC}}\def\CD@MK#%
1{\vbox{\offinterlineskip\kern-.2ex\CD@GG{#1}\kern-.2ex}}\def\@fillh{%
\xleaders\vrule\horizhtdp}\def\@fillv{\xleaders\hrule width\CD@LF}\CD@nF{rf:-%
}{@fillh}\CD@nF{lf:-}{@fillh}\CD@nF{df:-}{@fillv}\CD@nF{uf:-}{@fillv}\CD@nF{%
rh:}{null}\CD@nF{rm:}{null}\CD@nF{rt:}{null}\CD@nF{lh:}{null}\CD@nF{lm:}{null%
}\CD@nF{lt:}{null}\CD@nF{dh:}{null}\CD@nF{dm:}{null}\CD@nF{dt:}{null}\CD@nF{%
uh:}{null}\CD@nF{um:}{null}\CD@nF{ut:}{null}\CD@nF{+h:}{null}\CD@nF{+m:}{null%
}\CD@nF{+t:}{null}\CD@nF{-h:}{null}\CD@nF{-m:}{null}\CD@nF{-t:}{null}\def
\CD@@D{\hbox{\vrule height 1pt depth-1pt width 1pt}}\CD@RC{rf:}{\CD@@D}\CD@nF
{lf:}{rf:}\CD@nF{+f:}{rf:}\CD@RC{df:}{\CD@@D}\CD@nF{uf:}{df:}\CD@nF{-f:}{df:}%
\def\CD@BD{\CD@U\null\CD@@D\null\CD@@D\null}\edef\CD@lG{\string\newarrow}\def
\newarrow#1#2#3#4#5#6{\begingroup\edef\@name{#1}\edef\CD@oJ{#2}\edef\CD@iD{#3%
}\edef\CD@QG{#4}\edef\CD@jD{#5}\edef\CD@LE{#6}\let\CD@HE\CD@sG\let\CD@FK
\CD@BH\let\@x\CD@AH\ifx\CD@oJ\CD@iD\let\CD@oJ\empty\fi\ifx\CD@LE\CD@jD\let
\CD@LE\empty\fi\def\CD@LI{r}\def\CD@SF{l}\def\CD@IC{d}\def\CD@yJ{u}\def\CD@gH
{+}\def\@m{-}\ifx\CD@iD\CD@jD\ifx\CD@QG\CD@iD\let\CD@QG\empty\fi\ifx\CD@LE
\empty\ifx\CD@iD\CD@aE\let\@x\CD@yG\else\let\@x\CD@zG\fi\fi\else\edef\CD@a{%
\CD@iD\CD@oJ}\ifx\CD@a\empty\ifx\CD@QG\CD@jD\let\CD@QG\empty\fi\fi\fi\ifmmode
\aftergroup\CD@kG\else\CD@@A\CD@oB rh{head\space\space}\CD@LE\CD@oB rf{filler%
}\CD@iD\CD@oB rm{middle}\CD@QG\ifx\CD@jD\CD@iD\else\CD@oB rf{filler}\CD@jD\fi
\CD@oB rt{tail\space\space}\CD@oJ\CD@gE\CD@HE\CD@FK\@x\CD@nG l-2+2{lu}{nw}%
\NorthWest\CD@nG r+2+2{ru}{ne}\NorthEast\CD@nG l-2-2{ld}{sw}\SouthWest\CD@nG r%
+2-2{rd}{se}\SouthEast\else\aftergroup\CD@b\CD@L{r\@name}\fi\fi\endgroup}\def
\CD@sG{\CD@vG\CD@LI\CD@SF rl\Horizontal@Map}\def\CD@BH{\CD@vG\CD@IC\CD@yJ du%
\Vertical@Map}\def\CD@AH{\CD@vG\CD@gH\@m+-\Vector@Map}\def\CD@yG{\CD@vG\CD@gH
\@m+-\Slant@Map}\def\CD@zG{\CD@vG\CD@gH\@m+-\Slope@Map}\catcode`\/=\active
\def\CD@vG#1#2#3#4#5{\CD@jG#1#3#5t:\CD@oJ/f:\CD@iD/m:\CD@QG/f:\CD@jD/h:\CD@LE
//\CD@jG#2#4#5h:\CD@LE/f:\CD@jD/m:\CD@QG/f:\CD@iD/t:\CD@oJ//}\def\CD@jG#1#2#3%
#4//{\edef\CD@fG{#2}\aftergroup\sdef\CD@L{#1\@name}\aftergroup{\aftergroup#3%
\CD@M#4//\aftergroup}}\def\CD@M#1/{\edef\CD@EH{#1}\ifx\CD@EH\empty\else\CD@L{%
\CD@fG#1}\expandafter\CD@M\fi}\catcode`\/=12 \def\CD@nG#1#2#3#4#5#6#7#8{%
\aftergroup\sdef\CD@L{#6\@name}\aftergroup{\CD@L{#2\@name}\if#2#4\aftergroup
\CD@CI\else\aftergroup\CD@BI\fi\CD@L{#1\@name}%
\aftergroup(\aftergroup#3\aftergroup,\aftergroup#5\aftergroup)\aftergroup}}%
\def\CD@oB#1#2#3#4{\expandafter\ifx\csname#1#2:#4\endcsname\relax\CD@y\CD@gB{%
arrow#3 "#4" undefined}\fi}\CD@rG\CD@VE{All five components must be defined
before an arrow.}\CD@rG\CD@SE{\CD@lG, unlike \string\HorizontalMap, is a
declaration.}\def\CD@b#1{\CD@YA{Arrows \string#1 etc could not be defined}%
\CD@VE}\def\CD@kG{\CD@YA{misplaced \CD@lG}\CD@SE}\def\newdiagramgrid#1#2#3{%
\CD@RC{cdgh@#1}{#2,],}
\CD@RC{cdgv@#1}{#3,],}}
\CD@tG\ifincommdiag\incommdiagtrue\incommdiagfalse\CD@tG\CD@@F\CD@IF\CD@HF
\newcount\CD@VA\CD@VA=0 \def\CD@yH{\CD@VA6 }\def\CD@OB{\CD@VA1 \global\CD@yA1
\CD@DE\CD@YF\empty}\def\CD@YF{}\def\CD@nB#1{\relax\CD@MD\edef\CD@vJ{#1}%
\begingroup\CD@rE\else\ifcase\CD@VA\ifmmode\else\CD@YG\CD@E0\fi\or\CD@cE5\or
\CD@YG\CD@F5\or\CD@YG\CD@B5\or\CD@YG\CD@B5\or\CD@YG\CD@C5\or\CD@cE7\or\CD@YG
\CD@D7\fi\fi\endgroup\xdef\CD@YF{#1}}\def\CD@pB#1#2#3#4#5{\relax\CD@MD\xdef
\CD@vJ{#4}\begingroup\ifnum\CD@VA<#1 \expandafter\CD@cE\ifcase\CD@VA0\or#2\or
#3\else#2\fi\else\ifnum\CD@VA<6 \CD@tJ\CD@YG\CD@B#2\else\CD@YG\CD@G#2\fi\fi
\endgroup\CD@DE\CD@YF\CD@vJ\ifincommdiag\let\CD@ZD#5\else\let\CD@ZD\CD@LK\fi}%
\def\CD@yI{\global\CD@yA=\ifnum\CD@VA<5 1\else2\fi\relax}\def\CD@OI{\CD@VA
\CD@yA}\def\CD@cE#1{\aftergroup\CD@VA\aftergroup#1\aftergroup\relax}\def
\CD@HH{\def\CD@nB##1{\relax}\let\CD@pB\CD@FH\let\CD@yH\relax\let\CD@OB\relax
\let\CD@yI\relax\let\CD@OI\relax}\def\CD@FH#1#2#3#4#5{\ifincommdiag\let\CD@ZD
#5\else\xdef\CD@vJ{#4}\let\CD@ZD\CD@LK\fi}\def\CD@YG#1{\aftergroup#1%
\aftergroup\relax\CD@cE}\def\CD@B{\CD@YE\CD@S\CD@ME\CD@Q}\def\CD@G{\CD@YE{%
\CD@yC\CD@S}\CD@XE\CD@QD\CD@Q}\def\CD@F{\CD@YE{*\CD@S}\CD@RE\clubsuit\CD@Q}%
\def\CD@C{\CD@YE{\CD@S*\CD@S}\CD@RE\CD@Q\clubsuit\CD@Q}\def\CD@D{\CD@YE\CD@EC
\CD@TE\\}\def\CD@E{\CD@YE\CD@nC\CD@QE\CD@k}\def\CD@LK{\CD@YA{\CD@vJ\space
ignored \CD@dH}\CD@WE}\def\CD@FE{}\def\CD@d{\CD@YA{maps must never be enclosed
in braces}\CD@OE}\def\CD@dH{outside diagram}\def\CD@FC{\string\HonV, \string
\VonH\space and \string\HmeetV}\CD@rG\CD@ME{The way that horizontal and
vertical arrows are terminated implicitly means\CD@uG that they cannot be
mixed with each other or with \CD@FC.}\CD@rG\CD@XE{\string\pile\space is for
parallel horizontal arrows; verticals can just be put together in\CD@uG a cell%
. \CD@FC\space are not meaningful in a \string\pile.}\CD@rG\CD@RE{The
horizontal maps must point to an object, not each other (I've put in\CD@uG one
which you're unlikely to want). Use \string\pile\space if you want them
parallel.}\CD@rG\CD@TE{Parallel horizontal arrows must be in separate layers
of a \string\pile.}\CD@rG\CD@QE{Horizontal arrows may be used \CD@dH s, but
must still be in maths.}\CD@rG\CD@WE{Vertical arrows, \CD@FC\space\CD@dH s don%
't know where\CD@uG where to terminate.}\CD@rG\CD@OE{This prevents them from
stretching correctly.}\def\CD@YE#1{\CD@YA{"#1" inserted \ifx\CD@YF\empty
before \CD@vJ\else between \CD@YF\ifx\CD@YF\CD@vJ s\else\space and \CD@vJ\fi
\fi}}\count@=\year\multiply\count@12 \advance\count@\month\ifnum\count@>24391
\expandafter\endinput\fi\def
\Horizontal@Map{\CD@nB{horizontal map}\CD@LC\CD@TJ\CD@qD}\def\CD@TJ{\CD@GB-%
9999 \let\CD@ZD\CD@XD\ifincommdiag\else\CD@cJ\ifinpile\else\skip2\z@ plus 1.5%
\CD@VK minus .5\CD@UK\skip4\skip2 \fi\fi\let\CD@kD\@fillh\CD@nF{fill@dot}{rf:%
.}}\def\Vector@Map{\CD@HK4}\def\Slant@Map{\CD@HK{\CD@EF255\else6\fi}}\def
\Slope@Map{\CD@HK\CD@OC}\def\CD@HK#1#2#3#4#5#6{\CD@LC\def\CD@WK{2}\def\CD@aK{%
2}\def\CD@ZK{1}\def\CD@bK{1}\let\Horizontal@Map\CD@nI\def\CD@OG{#1}\def\CD@NI
{\CD@U#2#3#4#5#6}}\def\CD@nI{\CD@TJ\CD@JB\let\CD@ZD\CD@TD\CD@qD}\CD@tG\CD@pE
\CD@rA\CD@qA\CD@rA\def\cds@missives{\CD@rA}\def\CD@TD{\CD@vE\let\CD@OG\CD@OC
\CD@x\CD@zE\CD@WF\fi\setbox0\hbox{\incommdiagfalse\CD@HI}\CD@pE\CD@aD\else
\global\CD@YC\CD@bD\fi\ifvoid6 \ifvoid7 \CD@eE\fi\fi\CD@zE\else\CD@BD\global
\CD@YC\let\CD@CG\CD@IH\CD@YD\fi\else\CD@NI\CD@MI\global\CD@YC\CD@YD\fi}\def
\CD@LC{\begingroup\dimen1=\MapShortFall\dimen2=\CD@RG\dimen5=\MapShortFall
\setbox3=\box\voidb@x\setbox6=\box\voidb@x\setbox7=\box\voidb@x\CD@pD
\mathsurround\z@\skip2\z@ plus1fill minus 1000pt\skip4\skip2 \CD@TB}\CD@tG
\CD@tE\CD@UB\CD@TB\def\CD@U#1#2#3#4#5{\let\CD@oJ#1\let\CD@iD#2\let\CD@QG#3%
\let\CD@jD#4\let\CD@LE#5\CD@TB\ifx\CD@iD\CD@jD\CD@UB\fi}\def\CD@qD#1#2#3#4#5{%
\CD@U#1#2#3#4#5\CD@tD}\def\Vertical@Map{\CD@pB433{vertical map}\CD@cD\CD@LC
\CD@GB-9995 \let\CD@kD\@fillv\CD@nF{fill@dot}{df:.}\CD@qD}\def\break@args{%
\def\CD@tD{\CD@ZD}\CD@ZD\endgroup\aftergroup\CD@FE}\def\CD@MJ{\setbox1=\CD@oJ
\setbox5=\CD@LE\ifvoid3 \ifx\CD@QG\null\else\setbox3=\CD@QG\fi\fi\CD@@G2%
\CD@iD\CD@@G4\CD@jD}\def\CD@pF#1{\ifvoid1\else\CD@oF1#1\fi\ifvoid2\else\CD@oF
2#1\fi\ifvoid3\else\CD@oF3#1\fi\ifvoid4\else\CD@oF4#1\fi\ifvoid5\else\CD@oF5#%
1\fi} \def\CD@oF#1#2{\setbox#1\vbox{\offinterlineskip\box#1\dimen@\prevdepth
\advance\dimen@-#2\relax\setbox0\null\dp0\dimen@\ht0-\dimen@\box0}}\def\CD@@G
#1#2{\ifx#2\CD@kD\setbox#1=\box\voidb@x\else\setbox#1=#2\def#2{\xleaders\box#%
1}\fi}\CD@ZA\CD@BK{\string\HorizontalMap, \string\VerticalMap\space and
\string\DiagonalMap\CD@uG are obsolete - use \CD@lG\space to pre-define maps}%
\def\HorizontalMap#1#2#3#4#5{\CD@BK\CD@nB{old horizontal map}\CD@LC\CD@TJ\def
\CD@oJ{\CD@UH{#1}}\CD@SH\CD@iD{#2}\def\CD@QG{\CD@UH{#3}}\CD@SH\CD@jD{#4}\def
\CD@LE{\CD@UH{#5}}\CD@tD}\def\VerticalMap#1#2#3#4#5{\CD@BK\CD@pB433{vertical
map}\CD@cD\CD@LC\CD@GB-9995 \let\CD@kD\@fillv\def\CD@oJ{\CD@GG{#1}}\CD@VH
\CD@iD{#2}\def\CD@QG{\CD@GG{#3}}\CD@VH\CD@jD{#4}\def\CD@LE{\CD@GG{#5}}\CD@tD}%
\def\DiagonalMap#1#2#3#4#5{\CD@BK\CD@LC\def\CD@OG{4}\let\CD@kD\CD@qK\let
\CD@ZD\CD@YD\def\CD@WK{2}\def\CD@aK{2}\def\CD@ZK{1}\def\CD@bK{1}\def\CD@QG{%
\CD@vF{#3}}\ifPositiveGradient\let\mv\raise\def\CD@oJ{\CD@vF{#5}}\def\CD@iD{%
\CD@vF{#4}}\def\CD@jD{\CD@vF{#2}}\def\CD@LE{\CD@vF{#1}}\else\let\mv\lower\def
\CD@oJ{\CD@vF{#1}}\def\CD@iD{\CD@vF{#2}}\def\CD@jD{\CD@vF{#4}}\def\CD@LE{%
\CD@vF{#5}}\fi\CD@tD}\def\CD@aE{-}\def\CD@AD{\empty}\def\CD@SH{\CD@EG\CD@bE
\CD@aE\@fillh}\def\CD@VH{\CD@EG\CD@MK\CD@KK\@fillv}\def\CD@EG#1#2#3#4#5{\def
\CD@CH{#5}\ifx\CD@CH#2\let#4#3\else\let#4\null\ifx\CD@CH\empty\else\ifx\CD@CH
\CD@AD\else\let#4\CD@CH\fi\fi\fi}\def\CD@UH#1{\hbox{\mathsurround\z@
\offinterlineskip\def\CD@CH{#1}\ifx\CD@CH\empty\else\ifx\CD@CH\CD@AD\else
\CD@k\mkern-1.5mu{\CD@CH}\mkern-1.5mu\CD@ND\fi\fi}}\def\CD@yD#1#2{\setbox#1=%
\hbox\bgroup\setbox0=\hbox{\CD@k\labelstyle()\CD@ND}
\setbox1=\null\ht1\ht0\dp1\dp0\box1 \kern.1\CD@zC\CD@k\bgroup\labelstyle
\aftergroup\CD@LD\CD@xD}\def\CD@LD{\CD@ND\kern.1\CD@zC\egroup\CD@tD}\def
\CD@xD{\futurelet\CD@EH\CD@mJ}\def\CD@mJ{
\catcase\bgroup:\CD@v;\catcase\egroup:\missing@label;\catcase\space:\CD@TF;%
\tokcase[:\CD@XF;
\default:\CD@zJ;\endswitch}\def\CD@v{\let\CD@MD\CD@c\let\CD@CH}\def\CD@zJ#1{%
\let\CD@UF\egroup{\let\actually@braces@missing@around@macro@in@label\CD@ZH
\let\CD@MD\CD@xC\let\CD@UF\CD@VF#1%
\actually@braces@missing@around@macro@in@label}\CD@UF}\def
\actually@braces@missing@around@macro@in@label{\let\CD@CH=}\def\missing@label
{\egroup\CD@YA{missing label}\CD@PE}\def\CD@xC{\egroup\missing@label}\outer
\def\CD@ZH{}\def\CD@UF{}\def\CD@VF{\CD@wC\CD@UF}\def\CD@MD{}\def\CD@XF{\let
\CD@N\CD@xD\get@square@arg\CD@AE}\CD@rG\CD@PE{The text which has just been
read is not allowed within map labels.}\def\CD@c{\egroup\CD@YA{missing \CD@yC
\space inserted after label}\CD@PE}\def\upper@label{\CD@oD\CD@yD6}\def
\lower@label{\def\positional@{\CD@@A\break@args}\CD@yD7}\def\middle@label{%
\CD@yD3}\CD@tG\CD@yE\CD@pD\CD@oD\def\CD@iF{\ifPositiveGradient\CD@tJ
\expandafter\upper@label\else\expandafter\lower@label\fi}\def\CD@iI{%
\ifPositiveGradient\CD@tJ\expandafter\lower@label\else\expandafter
\upper@label\fi}\def\positional@{\CD@gB{labels as positional arguments are
obsolete}\CD@yE\CD@tJ\expandafter\upper@label\else\expandafter\lower@label\fi
-}\def\CD@tD{\futurelet\CD@EH\switch@arg}\def\eat@space{\afterassignment
\CD@tD\let\CD@EH= }\def\CD@TF{\afterassignment\CD@xD\let\CD@EH= }\def\CD@BC{%
\get@round@pair\CD@uD}\def\CD@uD#1#2{\def\CD@WK{#1}\def\CD@aK{#2}\CD@tD}\def
\optional@{\let\CD@N\CD@tD\get@square@arg\CD@AE}\def\CD@JJ.{\CD@sC\CD@tD}\def
\CD@sC{\let\CD@iD\fill@dot\let\CD@jD\fill@dot\def\CD@MI{\let\CD@iD\dfdot\let
\CD@jD\dfdot}}\def\CD@MI{}\def\CD@@E#1,{\CD@nH#1,\begingroup\ifx\@name\CD@RD
\CD@FF\aftergroup\CD@e\fi\aftergroup\CD@jC\else\expandafter\def\expandafter
\CD@RF\expandafter{\csname\@name\endcsname}\expandafter\CD@vD\CD@RF\CD@KD\ifx
\CD@RF\empty\aftergroup\CD@pC\expandafter\aftergroup\csname\CD@FB\@name
\endcsname\expandafter\aftergroup\csname\CD@FB @\@name\endcsname\else\gdef
\CD@GE{#1}\CD@gB{\string\relax\space inserted before `[\CD@GE'}\message{(I was
trying to read this as a \CD@tA\ option.)}\aftergroup\CD@H\fi\fi\endgroup}%
\def\CD@vD#1#2\CD@KD{\def\CD@RF{#2}}\def\CD@jC{\let\CD@CH\CD@N\let\CD@N\relax
\CD@CH}\def\CD@H#1],{
\CD@jC\relax\def\CD@RF{#1}\ifx\CD@RF\empty\def\CD@RF{[\CD@GE]}%
\else\def\CD@RF{[\CD@GE,#1]}
\fi\CD@RF}\def\CD@pC#1#2{\ifx#2\CD@qK\ifx#1\CD@qK\CD@gB{option `\@name'
undefined}\else#1\fi\else\CD@FF\expandafter#2\CD@GK\CD@PK\else\CD@QK\fi\fi
\CD@DH}\CD@tG\CD@FF\CD@QK\CD@PK\def\CD@nH#1,{\CD@FF\ifx\CD@GK\CD@qK\CD@e\else
\expandafter\CD@oH\CD@GK,#1,(,),(,)[]%
\fi\fi\CD@FF\else\CD@mH#1==,\fi}\def\CD@e{\CD@gB{option `\@name' needs (x,y)
value}\CD@PK\let\@name\empty}\def\CD@mH#1=#2=#3,{\def\@name{#1}\def\CD@GK{#2}%
\def\CD@RF{#3}\ifx\CD@RF\empty\let\CD@GK\CD@qK\fi}%
\def\CD@oH#1(#2,#3)#4,(#5,#6)#7[]{\def\CD@GK{{#2}{#3}}\def\CD@RF{#1#4#5#6}%
\ifx\CD@RF\empty\def\CD@RF{#7}\ifx\CD@RF\empty\CD@e\fi\else\CD@e\fi}\def
\CD@FB{cds@}\let\CD@N\relax\def\CD@zD#1{\ifx\CD@GK\CD@qK\CD@gB{option `\@name
' needs a value}\else#1\CD@GK\relax\fi}\def\CD@BE#1#2{\ifx\CD@GK\CD@qK#1#2%
\relax\else#1\CD@GK\relax\fi}\def\cds@@showpair#1#2{\message{x=#1,y=#2}}\def
\cds@@diagonalbase#1#2{\edef\CD@ZK{#1}\edef\CD@bK{#2}}\def\CD@DI#1{\def\CD@CH
{#1}\CD@nF{@x}{cdps@#1}\ifx\CD@CH\empty\CD@f\CD@CH{cannot be used}\else\ifx
\CD@CH\relax\CD@f\CD@CH{unknown}\else\let\CD@IK\@x\fi\fi}\def\CD@f#1#2{\CD@gB
{PostScript translator `#1' #2}}\def\CD@PH{}\def\CD@PJ{\CD@fA\edef\CD@PH{%
\noexpand\CD@KB{\@name\space ignored within maths}}}\def\diagramstyle{\CD@cJ
\let\CD@N\relax\CD@CF\CD@AE\CD@AE}\CD@tG\CD@sE
\CD@SB\CD@RB\CD@tG\CD@qE\CD@EB\CD@DB\CD@tG\CD@oE\CD@pA\CD@oA\CD@tG\CD@iE
\CD@HA\CD@GA\CD@HA\CD@tG\CD@jE\CD@JA\CD@IA\CD@tG\CD@kE\CD@LA\CD@KA\CD@tG
\CD@EF\CD@DK\CD@CK\CD@tG\CD@rE\CD@JB\CD@IB\CD@tG\CD@mE\CD@gA\CD@fA\CD@tG
\CD@nE\CD@kA\CD@jA\CD@tG\CD@AF\CD@iG\CD@hG\CD@RC{cds@ }{}\CD@RC{cds@}{}\CD@RC
{cds@1em}{\CellSize1\CD@zC}\CD@RC{cds@1.5em}{\CellSize1.5\CD@zC}\CD@RC{cds@2%
em}{\CellSize2\CD@zC}\CD@RC{cds@2.5em}{\CellSize2.5\CD@zC}\CD@RC{cds@3em}{%
\CellSize3\CD@zC}\CD@RC{cds@3.5em}{\CellSize3.5\CD@zC}\CD@RC{cds@4em}{%
\CellSize4\CD@zC}\CD@RC{cds@4.5em}{\CellSize4.5\CD@zC}\CD@RC{cds@5em}{%
\CellSize5\CD@zC}\CD@RC{cds@6em}{\CellSize6\CD@zC}\CD@RC{cds@7em}{\CellSize7%
\CD@zC}\CD@RC{cds@8em}{\CellSize8\CD@zC}\def\cds@abut{\MapsAbut\dimen1\z@
\dimen5\z@}\def\cds@alignlabels{\CD@IA\CD@KA}\def\cds@amstex{\ifincommdiag
\CD@O\else\def\CD{\diagram[amstex]}
\fi\CD@T\catcode`\@\active}\def\cds@b{\let\CD@dB\CD@bB}\def\cds@balance{\let
\CD@hA\CD@AA}\let\cds@bottom\cds@b\def\cds@center{\cds@vcentre\cds@nobalance}%
\let\cds@centre\cds@center\def\cds@centerdisplay{\CD@HA\CD@PJ\cds@balance}%
\let\cds@centredisplay\cds@centerdisplay\def\cds@crab{\CD@BE\CD@DC{.5%
\PileSpacing}}\CD@RC{cds@crab-}{\CD@DC-.5\PileSpacing}\CD@RC{cds@crab+}{%
\CD@DC.5\PileSpacing}\CD@RC{cds@crab++}{\CD@DC1.5\PileSpacing}\CD@RC{cds@crab%
--}{\CD@DC-1.5\PileSpacing}\def\cds@defaultsize{\CD@BE{\let\CD@QC}{3em}\CD@NJ
}\def\cds@displayoneliner{\CD@DB}\let\cds@dotted\CD@sC\def\cds@dpi{\CD@RJ{1%
truein}}\def\cds@dpm{\CD@RJ{100truecm}}\let\CD@XA\CD@qK\def\cds@eqno{\let
\CD@XA\CD@GK\let\CD@EJ\empty}\def\cds@fixed{\CD@qA}\CD@tG\CD@fE\CD@J\CD@I\def
\cds@flushleft{\CD@I\CD@GA\CD@PJ\cds@nobalance\CD@BE\CD@nA\CD@nA}\def\cds@gap
{\CD@AJ\setbox3=\null\ht3=\CD@tI\dp3=\CD@sI\CD@BE{\wd3=}\MapShortFall} \def
\cds@grid{\ifx\CD@GK\CD@qK\let\h@grid\relax\let\v@grid\relax\else\CD@nF{%
h@grid}{cdgh@\CD@GK}\CD@nF{v@grid}{cdgv@\CD@GK}\ifx\h@grid\relax\CD@gB{%
unknown grid `\CD@GK'}\else\CD@WB\fi\fi}\let\h@grid\relax\let\v@grid\relax
\def\cds@gridx{\ifx\CD@GK\CD@qK\else\cds@grid\fi\let\CD@CH\h@grid\let\h@grid
\v@grid\let\v@grid\CD@CH}\def\cds@h{\CD@zD\DiagramCellHeight}\def\cds@hcenter
{\let\CD@hA\CD@aA}\let\cds@hcentre\cds@hcenter\def\cds@heads{\CD@BE{\let
\CD@sJ}\CD@sJ\CD@@J\CD@vE\else\ifx\CD@sJ\CD@eF\else\CD@MC\fi\fi}\let
\cds@height\cds@h\let\cds@hmiddle\cds@balance\def\cds@htriangleheight{\CD@BE
\DiagramCellHeight\DiagramCellHeight\DiagramCellWidth1.73205%
\DiagramCellHeight}\def\cds@htrianglewidth{\CD@BE\DiagramCellWidth
\DiagramCellWidth\DiagramCellHeight.57735\DiagramCellWidth}\CD@tG\CD@zE\CD@eE
\CD@dE\CD@eE\def\cds@hug{\CD@eE} \def\cds@inline{\CD@gA\let\CD@PH\empty}\def
\cds@inlineoneliner{\CD@EB}\CD@RC{cds@l>}{\CD@zD{\let\CD@RG}\dimen2=\CD@RG}%
\def\cds@labelstyle{\CD@zD{\let\labelstyle}}\def\cds@landscape{\CD@kA}\def
\cds@large{\CellSize5\CD@zC}\let\CD@EJ\empty\def\CD@FJ{\refstepcounter{%
equation}\def\CD@XA{\hbox{\@eqnnum}}}\def\cds@LaTeXeqno{\let\CD@EJ\CD@FJ}\def
\cds@lefteqno{\CD@pA}\def\cds@leftflush{\cds@flushleft\CD@J}\def
\cds@leftshortfall{\CD@zD{\dimen1 }}\def\cds@lowershortfall{%
\ifPositiveGradient\cds@leftshortfall\else\cds@rightshortfall\fi}\def
\cds@loose{\CD@VB}\def\cds@midhshaft{\CD@JA}\def\cds@midshaft{\CD@JA}\def
\cds@midvshaft{\CD@LA}\def\cds@moreoptions{\CD@@A}\let\cds@nobalance
\cds@hcenter\def\cds@nohcheck{\CD@HH}\def\cds@nohug{\CD@dE} \def
\cds@nooptions{\def\CD@aC{\CD@WD}}\let\cds@noorigin\cds@nobalance\def
\cds@nopixel{\CD@@I4\CD@XH\CD@cJ}\def\cds@UO{\CD@oK\global\let\CD@n\empty}%
\def\cds@UglyObsolete{\cds@UO\let\cds@PS\empty}\def\CD@rK#1{\CD@gB{option `#1%
' renamed as `UglyObsolete'}}\def\cds@noPostScript{\CD@rK{noPostScript}}\def
\cds@noPS{\CD@rK{noPostScript}}\def\cds@notextflow{\CD@RB}\def\cds@noTPIC{%
\CD@CK}\def\cds@objectstyle{\CD@zD{\let\objectstyle}}\def\cds@origin{\let
\CD@hA\CD@iB}\def\cds@p{\CD@zD\PileSpacing}\let\cds@pilespacing\cds@p\def
\cds@pixelsize{\CD@zD\CD@@I\CD@gI}\def\cds@portrait{\CD@jA}\def
\cds@PostScript{\CD@nK\global\let\CD@n\empty\CD@BE\CD@DI\empty}\def\cds@PS{%
\CD@nK\global\let\CD@n\empty}\CD@GF\CD@n{\typeout{\CD@tA: try the PostScript
option for better results}}\def\cds@repositionpullbacks{\let\make@pbk\CD@fH
\let\CD@qH\CD@pH}\def\cds@righteqno{\CD@oA}\def\cds@rightshortfall{\CD@zD{%
\dimen5 }}\def\cds@ruleaxis{\CD@zD{\let\axisheight}}\def\cds@cmex{\let\CD@GG
\CD@sB\let\CD@QJ\CD@CJ}\def\cds@s{\cds@height\DiagramCellWidth
\DiagramCellHeight}\def\cds@scriptlabels{\let\labelstyle\scriptstyle}\def
\cds@shortfall{\CD@zD\MapShortFall\dimen1\MapShortFall\dimen5\MapShortFall}%
\def\cds@showfirstpass{\CD@BE{\let\CD@nD}\z@}\def\cds@silent{\def\CD@KB##1{}%
\def\CD@gB##1{}}\let\cds@size\cds@s\def\cds@small{\CellSize2\CD@zC}\def
\cds@snake{\CD@BE\CD@eJ\z@}\def\cds@t{\let\CD@dB\CD@fB}\def\cds@textflow{%
\CD@SB\CD@PJ}\def\cds@thick{\let\CD@rF\tenlnw\CD@LF\CD@NC\CD@BE\MapBreadth{2%
\CD@LF}\CD@@J}\def\cds@thin{\let\CD@rF\tenln\CD@BE\MapBreadth{\CD@NC}\CD@@J}%
\def\cds@tight{\CD@WB}\let\cds@top\cds@t\def\cds@TPIC{\CD@DK}\def
\cds@uppershortfall{\ifPositiveGradient\cds@rightshortfall\else
\cds@leftshortfall\fi}\def\cds@vcenter{\let\CD@dB\CD@cB}\let\cds@vcentre
\cds@vcenter\def\cds@vtriangleheight{\CD@BE\DiagramCellHeight
\DiagramCellHeight\DiagramCellWidth.577035\DiagramCellHeight}\def
\cds@vtrianglewidth{\CD@BE\DiagramCellWidth\DiagramCellWidth
\DiagramCellHeight1.73205\DiagramCellWidth}\def\cds@vmiddle{\let\CD@dB\CD@eB}%
\def\cds@w{\CD@zD\DiagramCellWidth}\let\cds@width\cds@w\def\diagram{\relax
\protect\CD@bC}\def\enddiagram{\protect\CD@SG}\def\CD@bC{\CD@g\CD@uI
\incommdiagtrue\edef\CD@wI{\the\CD@NB}\global\CD@NB\z@\boxmaxdepth\maxdimen
\everycr{}\CD@sK\everymath{}\everyhbox{}\ifx\pdfsyncstop\CD@qK\else
\pdfsyncstop\fi\CD@aC}\def\CD@aC{\CD@y\let\CD@N\CD@ZC\CD@CF\CD@AE\CD@WD}\def
\CD@ZC{\CD@gE\expandafter\CD@aC\else\expandafter\CD@WD\fi}\def\CD@WD{\let
\CD@EH\relax\CD@nE\CD@vE\else\CD@hK\else\CD@KB{landscape ignored without
PostScript}\CD@jA\fi\fi\fi\CD@EJ\setbox2=\vbox\bgroup\CD@JF\CD@VD}\def\CD@cH{%
\CD@nE\CD@fB\else\CD@dB\fi\CD@hA\nointerlineskip\setbox0=\null\ht0-\CD@pI\dp0%
\CD@pI\wd0\CD@kI\box0 \global\CD@QA\CD@kF\global\CD@yA\CD@XB\ifx\CD@NK\CD@qK
\global\CD@RA\CD@kF\else\global\CD@RA\CD@NK\fi\egroup\CD@zF\CD@nE\setbox2=%
\hbox to\dp2{\vrule height\wd2 depth\CD@QA width\z@\global\CD@QA\ht2\ht2\z@
\dp2\z@\wd2\z@\CD@hK\CD@tK{q 0 1 -1 0 0 0 cm}\else\global\CD@iG\CD@IK{0 1
bturn}\fi\box2\CD@gK\hss}\CD@DB\fi\ifnum\CD@yA=1 \else\CD@DB\fi\global
\@ignorefalse\CD@mE\leavevmode\fi\ifvmode\CD@TA\else\ifmmode\CD@PH\CD@GI\else
\CD@qE\CD@gA\fi\ifinner\CD@gA\fi\CD@mE\CD@GI\else\CD@sE\CD@QB\else\CD@TA\fi
\fi\fi\fi\CD@dD}\def\CD@dD{\global\CD@NB\CD@wI\relax\CD@xE\global\CD@ID\else
\aftergroup\CD@mC\fi\if@ignore\aftergroup\ignorespaces\fi\CD@wC\ignorespaces}%
\def\CD@fB{\advance\CD@pI\dimen1\relax}\def\CD@eB{\advance\CD@pI.5\dimen1%
\relax}\def\CD@bB{}\def\CD@cB{\CD@fB\advance\CD@pI\CD@YB\divide\CD@pI2
\advance\CD@pI-\axisheight\relax}\def\CD@aA{}\def\CD@iB{\CD@kF\z@}\def\CD@AA{%
\ifdim\dimen2>\CD@kF\CD@kF\dimen2 \else\dimen2\CD@kF\CD@kI\dimen0 \advance
\CD@kI\dimen2 \fi}\def\CD@QB{\skip0\z@\relax\loop\skip1\lastskip\ifdim\skip1>%
\z@\unskip\advance\skip0\skip1 \repeat\vadjust{\prevdepth\dp\strutbox\penalty
\predisplaypenalty\vskip\abovedisplayskip\CD@UA\penalty\postdisplaypenalty
\vskip\belowdisplayskip}\ifdim\skip0=\z@\else\hskip\skip0 \global\@ignoretrue
\fi}\def\CD@TA{\CD@LG\kern-\displayindent\CD@UA\CD@LG\global\@ignoretrue}\def
\CD@UA{\hbox to\hsize{\CD@fE\ifdim\CD@RA=\z@\else\advance\CD@QA-\CD@RA\setbox
2=\hbox{\kern\CD@RA\box2}\fi\fi\setbox1=\hbox{\ifx\CD@XA\CD@qK\else\CD@k
\CD@XA\CD@ND\fi}\CD@oE\CD@iE\else\advance\CD@QA\wd1 \fi\wd1\z@\box1 \fi\dimen
0\wd2 \advance\dimen0\wd1 \advance\dimen0-\hsize\ifdim\dimen0>-\CD@nA\CD@HA
\fi\advance\dimen0\CD@QA\ifdim\dimen0>\z@\CD@KB{wider than the page by \the
\dimen0 }\CD@HA\fi\CD@iE\hss\else\CD@V\CD@QA\CD@nA\fi\CD@GI\hss\kern-\wd1\box
1 }}\def\CD@GI{\CD@AF\CD@@F\else\CD@SC\global\CD@hG\fi\fi\kern\CD@QA\box2 }%
\CD@tG\CD@wE\CD@YC\CD@XC\def\CD@JF{\CD@cJ\ifdim\DiagramCellHeight=-\maxdimen
\DiagramCellHeight\CD@QC\fi\ifdim\DiagramCellWidth=-\maxdimen
\DiagramCellWidth\CD@QC\fi\global\CD@XC\CD@IF\let\CD@FE\empty\let\CD@z\CD@Q
\let\overprint\CD@eH\let\CD@s\CD@rJ\let\enddiagram\CD@ED\let\\\CD@cC\let\par
\CD@jH\let\CD@MD\empty\let\switch@arg\CD@PB\let\shift\CD@iA\baselineskip
\DiagramCellHeight\lineskip\z@\lineskiplimit\z@\mathsurround\z@\tabskip\z@
\CD@OB}\def\CD@VD{\penalty-123 \begingroup\CD@jA\aftergroup\CD@K\halign
\bgroup\global\advance\CD@NB1 \vadjust{\penalty1}\global\CD@FA\z@\CD@OB\CD@j#%
#\CD@DD\CD@Q\CD@Q\CD@OI\CD@j##\CD@DD\cr}\def\CD@ED{\CD@MD\CD@GD\crcr\egroup
\global\CD@JD\endgroup}\def\CD@j{\global\advance\CD@FA1 \futurelet\CD@EH\CD@i
}\def\CD@i{\ifx\CD@EH\CD@DD\CD@tJ\hskip1sp plus 1fil \relax\let\CD@DD\relax
\CD@vI\else\hfil\CD@k\objectstyle\let\CD@FE\CD@d\fi}\def\CD@DD{\CD@MD\relax
\CD@yI\CD@vI\global\CD@QA\CD@iA\penalty-9993 \CD@ND\hfil\null\kern-2\CD@QA
\null}\def\CD@cC{\cr}\def\across#1{\span\omit\mscount=#1 \global\advance
\CD@FA\mscount\global\advance\CD@FA\m@ne\CD@sF\ifnum\mscount>2 \CD@fJ\repeat
\ignorespaces}\def\CD@fJ{\relax\span\omit\advance\mscount\m@ne}\def\CD@qJ{%
\ifincommdiag\ifx\CD@iD\@fillh\ifx\CD@jD\@fillh\ifdim\dimen3>\z@\else\ifdim
\dimen2>93\CD@@I\ifdim\dimen2>18\p@\ifdim\CD@LF>\z@\count@\CD@bJ\advance
\count@\m@ne\ifnum\count@<\z@\count@20\let\CD@aJ\CD@uJ\fi\xdef\CD@bJ{\the
\count@}\fi\fi\fi\fi\fi\fi\fi}\def\CD@cG#1{\vrule\horizhtdp width#1\dimen@
\kern2\dimen@}\def\CD@uJ{\rlap{\dimen@\CD@@I\CD@V\dimen@{.182\p@}\CD@zH
\dimen@\advance\CD@tI\dimen@\CD@cG0\CD@cG0\CD@cG2\CD@cG6\CD@cG6\CD@cG2\CD@cG0%
\CD@cG0\CD@cG2\CD@cG6\CD@cG0\CD@cG0\CD@cG2\CD@cG2\CD@cG6\CD@cG0\CD@cG0\CD@cG2%
\CD@cG6\CD@cG2\CD@cG2\CD@cG0\CD@cG0}}\def\CD@bJ{10}\def\CD@aJ{}\def\CD@XD{%
\CD@gE\CD@TB\fi\CD@x\CD@WF\CD@HI}\def\CD@x{\CD@QJ\CD@DC\CD@MJ\ifdim\CD@DC=\z@
\else\CD@pF\CD@DC\fi\ifvoid3 \setbox3=\null\ht3\CD@tI\dp3\CD@sI\else\CD@V{\ht
3}\CD@tI\CD@V{\dp3}\CD@sI\fi\dimen3=.5\wd3 \ifdim\dimen3=\z@\CD@tE\else\dimen
3-\CD@XH\fi\else\CD@TB\fi\CD@V{\dimen2}{\wd7}\CD@V{\dimen2}{\wd6}\CD@qJ
\advance\dimen2-2\dimen3 \dimen4.5\dimen2 \dimen2\dimen4 \advance\dimen2%
\CD@eJ\advance\dimen4-\CD@eJ\advance\dimen2-\wd1 \advance\dimen4-\wd5 \ifvoid
2 \else\CD@V{\ht3}{\ht2}\CD@V{\dp3}{\dp2}\CD@V{\dimen2}{\wd2}\fi\ifvoid4 \else
\CD@V{\ht3}{\ht4}\CD@V{\dp3}{\dp4}\CD@V{\dimen4}{\wd4}\fi\advance\skip2\dimen
2 \advance\skip4\dimen4 \CD@tE\advance\skip2\skip4 \dimen0\dimen5 \advance
\dimen0\wd5 \skip3-\skip4 \advance\skip3-\dimen0 \let\CD@jD\empty\else\skip3%
\z@\relax\dimen0\z@\fi}\def\CD@WF{\offinterlineskip\lineskip.2\CD@zC\ifvoid6
\else\setbox3=\vbox{\hbox to2\dimen3{\hss\box6\hss}\box3}\fi\ifvoid7 \else
\setbox3=\vtop{\box3 \hbox to2\dimen3{\hss\box7\hss}}\fi}\def\CD@HI{\kern
\dimen1 \box1 \CD@aJ\CD@iD\hskip\skip2 \kern\dimen0 \ifincommdiag\CD@jE
\penalty1\fi\kern\dimen3 \penalty\CD@GB\hskip\skip3 \null\kern-\dimen3 \else
\hskip\skip3 \fi\box3 \CD@jD\hskip\skip4 \box5 \kern\dimen5}\def\CD@MF{\ifnum
\CD@LH>\CD@TC\CD@V{\dimen1}\objectheight\CD@V{\dimen5}\objectheight\else\CD@V
{\dimen1}\objectwidth\CD@V{\dimen5}\objectwidth\fi}\def\CD@Y{\begingroup
\ifdim\dimen7=\z@\kern\dimen8 \else\ifdim\dimen6=\z@\kern\dimen9 \else\dimen5%
\dimen6 \dimen6\dimen9 \CD@KJ\dimen4\dimen2 \CD@dG{\dimen4}\dimen6\dimen5
\dimen7\dimen8 \CD@KJ\CD@iC{\dimen2}\ifdim\dimen2<\dimen4 \kern\dimen2 \else
\kern\dimen4 \fi\fi\fi\endgroup}\def\CD@jJ{\CD@JI\setbox\z@\hbox{\lower
\axisheight\hbox to\dimen2{\CD@DF\ifPositiveGradient\dimen8\ht\CD@MH\dimen9%
\CD@mI\else\dimen8\dp3 \dimen9\dimen1 \fi\else\dimen8 \ifPositiveGradient
\objectheight\else\z@\fi\dimen9\objectwidth\fi\advance\dimen8
\ifPositiveGradient-\fi\axisheight\CD@Y\unhbox\z@\CD@DF\ifPositiveGradient
\dimen8\dp3 \dimen9\dimen0 \else\dimen8\ht\CD@MH\dimen9\CD@mF\fi\else\dimen8
\ifPositiveGradient\z@\else\objectheight\fi\dimen9\objectwidth\fi\advance
\dimen8 \ifPositiveGradient\else-\fi\axisheight\CD@Y}}}\def\CD@bD{\dimen6
\CD@aK\DiagramCellHeight\dimen7 \CD@WK\DiagramCellWidth\CD@jJ
\ifPositiveGradient\advance\dimen7-\CD@ZK\DiagramCellWidth\else\dimen7 \CD@ZK
\DiagramCellWidth\dimen6\z@\fi\advance\dimen6-\CD@bK\DiagramCellHeight\CD@mK
\setbox0=\rlap{\kern-\dimen7 \lower\dimen6\box\z@}\ht0\z@\dp0\z@\raise
\axisheight\box0 }\def\CD@mK{\setbox0\hbox{\ht\z@\z@\dp\z@\z@\wd\z@\z@\CD@hK
\expandafter\CD@tK{q \CD@eK\space\CD@lK\space\CD@kK\space\CD@eK\space0 0 cm}%
\else\global\CD@iG\CD@eD{\the\CD@TC\space\ifPositiveGradient\else-\fi\the
\CD@LH\space bturn}\fi\box\z@\CD@gK}}\def\CD@vB{\advance\CD@hF-\CD@mI\CD@wJ
\CD@hF\advance\CD@wJ\CD@hI\ifvoid\CD@sH\ifdim\CD@wJ<.1em\ifnum\CD@gD=\@m\else
\CD@aG h\CD@wJ<.1em:objects overprint:\CD@FA\CD@gD\fi\fi\else\ifhbox\CD@sH
\CD@SK\else\CD@TK\fi\advance\CD@wJ\CD@mI\CD@bH{-\CD@mI}{\box\CD@sH}{\CD@wJ}%
\z@\fi\CD@hF-\CD@mF\CD@gD\CD@FA\CD@hI\z@}\def\CD@SK{\setbox\CD@sH=\hbox{%
\unhbox\CD@sH\unskip\unpenalty}\setbox\CD@tH=\hbox{\unhbox\CD@tH\unskip
\unpenalty}\setbox\CD@sH=\hbox to\CD@wJ{\CD@OA\wd\CD@sH\unhbox\CD@sH\CD@PA
\lastkern\unkern\ifdim\CD@PA=\z@\CD@UB\advance\CD@OA-\wd\CD@tH\else\CD@TB\fi
\ifnum\lastpenalty=\z@\else\CD@JA\unpenalty\fi\kern\CD@PA\ifdim\CD@hF<\CD@OA
\CD@JA\fi\ifdim\CD@hI<\wd\CD@tH\CD@JA\fi\CD@jE\CD@hI\CD@wJ\advance\CD@hI-%
\CD@OA\advance\CD@hI\wd\CD@tH\ifdim\CD@hI<2\wd\CD@tH\CD@aG h\CD@hI<2\wd\CD@tH
:arrow too short:\CD@FA\CD@gD\fi\divide\CD@hI\tw@\CD@hF\CD@wJ\advance\CD@hF-%
\CD@hI\fi\CD@tE\kern-\CD@hI\fi\hbox to\CD@hI{\unhbox\CD@tH}\CD@HG}}\CD@tG
\ifinpile\inpiletrue\inpilefalse\inpilefalse\def\pile{\protect\CD@UJ\protect
\CD@uH}\def\CD@uH#1{\CD@l#1\CD@QD}\def\CD@UJ{\CD@nB{pile}\setbox0=\vtop
\bgroup\aftergroup\CD@lD\inpiletrue\let\CD@FE\empty\let\pile\CD@KF\let\CD@QD
\CD@PD\let\CD@GD\CD@FD\CD@yH\baselineskip.5\PileSpacing\lineskip.1\CD@zC
\relax\lineskiplimit\lineskip\mathsurround\z@\tabskip\z@\let\\\CD@wH}\def
\CD@l{\CD@DE\CD@YF\empty\halign\bgroup\hfil\CD@k\let\CD@FE\CD@d\let\\\CD@vH##%
\CD@MD\CD@ND\hfil\CD@Q\CD@R##\cr}\CD@rG\CD@NE{pile only allows one column.}%
\CD@rG\CD@UE{you left it out!}\def\CD@R{\CD@QD\CD@Q\relax\CD@YA{missing \CD@yC
\space inserted after \string\pile}\CD@NE}\def\CD@PD{\CD@MD\crcr\egroup
\egroup}\def\CD@GD{\CD@MD}\def\CD@FD{\CD@MD\relax\CD@QD\CD@YA{missing \CD@yC
\space inserted between \string\pile\space and \CD@HD}\CD@UE}\def\CD@QD{%
\CD@MD}\def\CD@lD{\vbox{\dimen1\dp0 \unvbox0 \setbox0=\lastbox\advance\dimen1%
\dp0 \nointerlineskip\box0 \nointerlineskip\setbox0=\null\dp0.5\dimen1\ht0-%
\dp0 \box0}\ifincommdiag\CD@tJ\penalty-9998 \fi\xdef\CD@YF{pile}}\def\CD@vH{%
\cr}\def\CD@wH{\noalign{\skip@\prevdepth\advance\skip@-\baselineskip
\prevdepth\skip@}}\def\CD@KF#1{#1}\def\CD@TK{\setbox\CD@sH=\vbox{\unvbox
\CD@sH\setbox1=\lastbox\setbox0=\box\voidb@x\CD@tF\setbox\CD@sH=\lastbox
\ifhbox\CD@sH\CD@rC\repeat\unvbox0 \global\CD@QA\CD@ZE}\CD@ZE\CD@QA}\def
\CD@rC{\CD@jE\setbox\CD@sH=\hbox{\unhbox\CD@sH\unskip\setbox\CD@sH=\lastbox
\unskip\unhbox\CD@sH}\ifdim\CD@wJ<\wd\CD@sH\CD@aG h\CD@wJ<\wd\CD@sH:arrow in
pile too short:\CD@FA\CD@gD\else\setbox\CD@sH=\hbox to\CD@wJ{\unhbox\CD@sH}%
\fi\else\CD@gJ\fi\setbox0=\vbox{\box\CD@sH\nointerlineskip\ifvoid0 \CD@tJ\box
1 \else\vskip\skip0 \unvbox0 \fi}\skip0=\lastskip\unskip}\def\CD@gJ{\penalty7
\noindent\unhbox\CD@sH\unskip\setbox\CD@sH=\lastbox\unskip\unhbox\CD@sH
\endgraf\setbox\CD@tH=\lastbox\unskip\setbox\CD@tH=\hbox{\CD@JG\unhbox\CD@tH
\unskip\unskip\unpenalty}\ifcase\prevgraf\cd@shouldnt P\or\ifdim\CD@wJ<\wd
\CD@tH\CD@aG h\CD@wJ<\wd\CD@sH:object in pile too wide:\CD@FA\CD@gD\setbox
\CD@sH=\hbox to\CD@wJ{\hss\unhbox\CD@tH\hss}\else\setbox\CD@sH=\hbox to\CD@wJ
{\hss\kern\CD@hF\unhbox\CD@tH\kern\CD@hI\hss}\fi\or\setbox\CD@sH=\lastbox
\unskip\CD@SK\else\cd@shouldnt Q\fi\unskip\unpenalty}\def\CD@cD{\CD@MJ\ifvoid
3 \setbox3=\null\ht3\axisheight\dp3-\ht3 \dimen3.5\CD@LF\else\dimen4\dp3
\dimen3.5\wd3 \setbox3=\CD@GG{\box3}\dp3\dimen4 \ifdim\ht3=-\dp3 \else\CD@TB
\fi\fi\dimen0\dimen3 \advance\dimen0-.5\CD@LF\setbox0\null\ht0\ht3\dp0\dp3\wd
0\wd3 \ifvoid6\else\setbox6\hbox{\unhbox6\kern\dimen0\kern2pt}\dimen0\wd6 \fi
\ifvoid7\else\setbox7\hbox{\kern2pt\kern\dimen3\unhbox7}\dimen3\wd7 \fi
\setbox3\hbox{\ifvoid6\else\kern-\dimen0\unhbox6\fi\unhbox3 \ifvoid7\else
\unhbox7\kern-\dimen3\fi}\ht3\ht0\dp3\dp0\wd3\wd0 \CD@tE\dimen4=\ht\CD@MH
\advance\dimen4\dp5 \advance\dimen4\dimen1 \let\CD@jD\empty\else\dimen4\ht3
\fi\setbox0\null\ht0\dimen4 \offinterlineskip\setbox8=\vbox spread2ex{\kern
\dimen5 \box1 \CD@iD\vfill\CD@tE\else\kern\CD@eJ\fi\box0}\ht8=\z@\setbox9=%
\vtop spread2ex{\kern-\ht3 \kern-\CD@eJ\box3 \CD@jD\vfill\box5 \kern\dimen1}%
\dp9=\z@\hskip\dimen0plus.0001fil \box9 \kern-\CD@LF\box8 \CD@kE\penalty2 \fi
\CD@tE\penalty1 \fi\kern\PileSpacing\kern-\PileSpacing\kern-.5\CD@LF\penalty
\CD@GB\null\kern\dimen3}\def\CD@cI{\ifhbox\CD@VA\CD@KB{clashing verticals}\ht
\CD@MH.5\dp\CD@VA\dp\CD@MH-\ht5 \CD@yB\ht\CD@MH\z@\dp\CD@MH\z@\fi\dimen1\dp
\CD@VA\CD@xA\prevgraf\unvbox\CD@VA\CD@wA\lastpenalty\unpenalty\setbox\CD@VA=%
\null\setbox\CD@lI=\hbox{\CD@JG\unhbox\CD@lI\unskip\unpenalty\dimen0\lastkern
\unkern\unkern\unkern\kern\dimen0 \CD@HG}\setbox\CD@lF=\hbox{\unhbox\CD@lF
\dimen0\lastkern\unkern\unkern\global\CD@QA\lastkern\unkern\kern\dimen0 }%
\CD@tF\ifnum\CD@xA>4 \CD@zI\repeat\unskip\unskip\advance\CD@mF.5\wd\CD@VA
\advance\CD@mF\wd\CD@lF\advance\CD@mI.5\wd\CD@VA\advance\CD@mI\wd\CD@lI\ifnum
\CD@FA=\CD@lA\CD@OA.5\wd\CD@VA\edef\CD@NK{\the\CD@OA}\fi\setbox\CD@VA=\hbox{%
\kern-\CD@mF\box\CD@lF\unhbox\CD@VA\box\CD@lI\kern-\CD@mI\penalty\CD@wA
\penalty\CD@NB}\ht\CD@VA\dimen1 \dp\CD@VA\z@\wd\CD@VA\CD@tB\CD@vB}\def\CD@zI{%
\ifdim\wd\CD@lF<\CD@QA\setbox\CD@lF=\hbox to\CD@QA{\CD@JG\unhbox\CD@lF}\fi
\advance\CD@xA\m@ne\setbox\CD@VA=\hbox{\box\CD@lF\unhbox\CD@VA}\unskip\setbox
\CD@lF=\lastbox\setbox\CD@lF=\hbox{\unhbox\CD@lF\unskip\unpenalty\dimen0%
\lastkern\unkern\unkern\global\CD@QA\lastkern\unkern\kern\dimen0 }}\def\CD@yB
{\dimen1\dp\CD@VA\ifhbox\CD@VA\CD@xB\else\CD@zB\fi\setbox\CD@VA=\vbox{%
\penalty\CD@NB}\dp\CD@VA-\dp\CD@MH\wd\CD@VA\CD@tB}\def\CD@zB{\unvbox\CD@VA
\CD@wA\lastpenalty\unpenalty\ifdim\dimen1<\ht\CD@MH\CD@aG v\dimen1<\ht\CD@MH:%
rows overprint:\CD@NB\CD@wA\fi}\def\CD@xB{\dimen0=\ht\CD@VA\setbox\CD@VA=%
\hbox\bgroup\advance\dimen1-\ht\CD@MH\unhbox\CD@VA\CD@xA\lastpenalty
\unpenalty\CD@wA\lastpenalty\unpenalty\global\CD@RA-\lastkern\unkern\setbox0=%
\lastbox\CD@tF\setbox\CD@VA=\hbox{\box0\unhbox\CD@VA}\setbox0=\lastbox\ifhbox
0 \CD@kJ\repeat\global\CD@SA-\lastkern\unkern\global\CD@QA\CD@JK\unhbox\CD@VA
\egroup\CD@JK\CD@QA\CD@bH{\CD@SA}{\box\CD@VA}{\CD@RA}{\dimen1}}\def\CD@kJ{%
\setbox0=\hbox to\wd0\bgroup\unhbox0 \unskip\unpenalty\dimen7\lastkern\unkern
\ifnum\lastpenalty=1 \unpenalty\CD@UB\else\CD@TB\fi\ifnum\lastpenalty=2
\unpenalty\dimen2.5\dimen0\advance\dimen2-.5\dimen1\advance\dimen2-%
\axisheight\else\dimen2\z@\fi\setbox0=\lastbox\dimen6\lastkern\unkern\setbox1%
=\lastbox\setbox0=\vbox{\unvbox0 \CD@tE\kern-\dimen1 \else\ifdim\dimen2=\z@
\else\kern\dimen2 \fi\fi}\ifdim\dimen0<\ht0 \CD@aG v\dimen0<\ht0:upper part of
vertical too short:{\CD@tE\CD@NB\else\CD@wA\fi}\CD@xA\else\setbox0=\vbox to%
\dimen0{\unvbox0}\fi\setbox1=\vtop{\unvbox1}\ifdim\dimen1<\dp1 \CD@aG v\dimen
1<\dp1:lower part of vertical too short:\CD@NB\CD@wA\else\setbox1=\vtop to%
\dimen1{\ifdim\dimen2=\z@\else\kern-\dimen2 \fi\unvbox1 }\fi\box1 \kern\dimen
6 \box0 \kern\dimen7 \CD@HG\global\CD@QA\CD@JK\egroup\CD@JK\CD@QA\relax}%
\countdef\CD@u=14 \newcount\CD@CA\newcount\CD@XB\newcount\CD@NB\let\CD@LB
\insc@unt\newcount\CD@FA\newcount\CD@lA\let\CD@mA\CD@XB\newcount\CD@MB\CD@tG
\CD@DF\CD@bI\CD@aI\CD@aI\def\CD@nD{-1}\def\CD@K{\CD@t\ifnum\CD@nD<\z@\else
\begingroup\scrollmode\showboxdepth\CD@nD\showboxbreadth\maxdimen\showlists
\endgroup\fi\CD@bI\CD@zF\CD@CA=\CD@u\advance\CD@CA1 \CD@XB=\CD@CA\ifnum\CD@NB
=1 \CD@JA\fi\advance\CD@XB\CD@NB\dimen1\z@\skip0\z@\count@=\insc@unt\advance
\count@\CD@u\divide\count@2 \ifnum\CD@XB>\count@\CD@KB{The diagram has too
many rows! It can't be reformatted.}\else\CD@NG\CD@WI\fi\CD@cH}\def\CD@NG{%
\CD@NB\CD@CA\CD@uF\ifnum\CD@NB<\CD@XB\setbox\CD@NB\box\voidb@x\advance\CD@NB1%
\relax\repeat\CD@NB\CD@CA\skip\z@\z@\CD@uF\CD@GB\lastpenalty\unpenalty\ifnum
\CD@GB>\z@\CD@KE\repeat\ifnum\CD@GB=-123 \CD@tJ\unpenalty\else\cd@shouldnt D%
\fi\ifx\v@grid\relax\else\CD@NB\CD@XB\advance\CD@NB\m@ne\expandafter\CD@VJ
\v@grid\fi\CD@MB\CD@mA\CD@tB\z@\CD@XG\ifx\h@grid\relax\else\expandafter\CD@LJ
\h@grid\fi\count@\CD@XB\advance\count@\m@ne\CD@YB\ht\count@}\def\CD@KE{%
\ifcase\CD@GB\or\CD@MG\else\CD@uA-\lastpenalty\unpenalty\CD@vA\lastpenalty
\unpenalty\setbox0=\lastbox\CD@WG\fi\CD@wD}\def\CD@wD{\skip1\lastskip\unskip
\advance\skip0\skip1 \ifdim\skip1=\z@\else\expandafter\CD@wD\fi}\def\CD@MG{%
\setbox0=\lastbox\CD@pI\dp0 \advance\CD@pI\skip\z@\skip\z@\z@\advance\CD@NF
\CD@pI\CD@uE\ifnum\CD@NB>\CD@CA\CD@NF\DiagramCellHeight\CD@pI\CD@NF\advance
\CD@pI-\CD@qI\fi\fi\CD@qI\ht0 \CD@NF\CD@qI\setbox\CD@NB\hbox{\unhbox\CD@NB
\unhbox0}\dp\CD@NB\CD@pI\ht\CD@NB\CD@qI\advance\CD@NB1 }\def\CD@WG{\ifnum
\CD@uA<\z@\advance\CD@uA\CD@XB\ifnum\CD@uA<\CD@CA\CD@UG\else\CD@OA\dp\CD@uA
\CD@PA\ht\CD@uA\setbox\CD@uA\hbox{\box\z@\penalty\CD@vA\penalty\CD@GB\unhbox
\CD@uA}\dp\CD@uA\CD@OA\ht\CD@uA\CD@PA\fi\else\CD@UG\fi}\def\CD@UG{\CD@KB{%
diagonal goes outside diagram (lost)}}\def\CD@fI{\advance\CD@uA\CD@XB\ifnum
\CD@uA<\CD@CA\CD@UG\else\ifnum\CD@uA=\CD@NB\CD@VG\else\ifnum\CD@uA>\CD@NB
\cd@shouldnt M\else\CD@OA\dp\CD@uA\CD@PA\ht\CD@uA\setbox\CD@uA\hbox{\box\z@
\penalty\CD@vA\penalty\CD@GB\unhbox\CD@uA}\dp\CD@uA\CD@OA\ht\CD@uA\CD@PA\fi
\fi\fi}\def\CD@WI{\CD@AJ\setbox\CD@PC=\hbox{\CD@k A\@super f\CD@lJ f\CD@ND}%
\CD@ZE\z@\CD@JK\z@\CD@kI\z@\CD@kF\z@\CD@NB=\CD@XB\CD@NF\z@\CD@uB\z@\CD@uF
\ifnum\CD@NB>\CD@CA\advance\CD@NB\m@ne\CD@qI\ht\CD@NB\CD@pI\dp\CD@NB\advance
\CD@NF\CD@qI\CD@rI\advance\CD@uB\CD@NF\CD@KC\CD@ZI\CD@w\ht\CD@NB\CD@qI\dp
\CD@NB\CD@pI\nointerlineskip\box\CD@NB\CD@NF\CD@pI\setbox\CD@NB\null\ht\CD@NB
\CD@uB\repeat\CD@wB\nointerlineskip\box\CD@NB\CD@gG\CD@ZE\DiagramCellWidth{%
width}\CD@gG\CD@JK\DiagramCellHeight{height}\CD@VA\CD@LB\advance\CD@VA-\CD@lA
\advance\CD@VA\m@ne\advance\CD@VA\CD@mA\dimen0\wd\CD@VA\CD@tI\axisheight
\dimen1\CD@uB\advance\dimen1-\CD@YB\dimen2\CD@kI\advance\dimen2-\dimen0
\advance\CD@XB-\CD@CA\advance\CD@LB-\CD@lA}\count@\year\multiply\count@12
\advance\count@\month\ifnum\count@>24398 \loop\iftrue\repeat\fi\def\CD@wB{\CD@qI-\CD@NF\CD@pI\CD@NF\setbox\CD@MH=\null\dp
\CD@MH\CD@NF\ht\CD@MH-\CD@NF\CD@mF\z@\CD@mI\z@\CD@lA\CD@LB\advance\CD@lA-%
\CD@MB\advance\CD@lA\CD@mA\CD@FA\CD@LB\CD@VA\CD@MB\CD@sF\ifnum\CD@FA>\CD@lA
\advance\CD@FA\m@ne\advance\CD@VA\m@ne\CD@tB\wd\CD@VA\setbox\CD@FA=\box
\voidb@x\CD@yB\repeat\CD@w\ht\CD@NB\CD@qI\dp\CD@NB\CD@pI}\def\CD@gG#1#2#3{%
\ifdim#1>.01\CD@zC\CD@PA#2\relax\advance\CD@PA#1\relax\advance\CD@PA.99\CD@zC
\count@\CD@PA\divide\count@\CD@zC\CD@KB{increase cell #3 to \the\count@ em}%
\fi}\def\CD@rI{\CD@FA=\CD@LB\penalty4 \noindent\unhbox\CD@NB\CD@sF\unskip
\setbox0=\lastbox\ifhbox0 \advance\CD@FA\m@ne\setbox\CD@FA\hbox to\wd0{\null
\penalty-9990\null\unhbox0}\repeat\CD@lA\CD@FA\advance\CD@FA\CD@MB\advance
\CD@FA-\CD@mA\ifnum\CD@FA<\CD@LB\count@\CD@FA\advance\count@\m@ne\dimen0=\wd
\count@\count@\CD@MB\advance\count@\m@ne\CD@tB\wd\count@\CD@sF\ifnum\CD@FA<%
\CD@LB\CD@DJ\CD@XG\dimen0\wd\CD@FA\advance\CD@FA1 \repeat\fi\CD@sF\CD@GB
\lastpenalty\unpenalty\ifnum\CD@GB>\z@\CD@vA\lastpenalty\unpenalty\CD@VG
\repeat\endgraf\unskip\ifnum\lastpenalty=4 \unpenalty\else\cd@shouldnt S\fi}%
\def\CD@VG{\advance\CD@vA\CD@lA\advance\CD@vA\m@ne\setbox0=\lastbox\ifnum
\CD@vA<\CD@LB\setbox\CD@vA\hbox{\box0\penalty\CD@GB\unhbox\CD@vA}\else\CD@UG
\fi}\def\CD@bG{}\CD@tG\CD@uE\CD@WB\CD@VB\def\CD@DJ{\advance\dimen0\wd\CD@FA
\divide\dimen0\tw@\CD@uE\dimen0\DiagramCellWidth\else\CD@V{\dimen0}%
\DiagramCellWidth\CD@pJ\fi\advance\CD@tB\dimen0 }\def\CD@XG{\setbox\CD@MB=%
\vbox{}\dp\CD@MB=\CD@uB\wd\CD@MB\CD@tB\advance\CD@MB1 }\def\CD@LJ#1,{\def
\CD@GK{#1}\ifx\CD@GK\CD@RD\else\advance\CD@tB\CD@GK\DiagramCellWidth\CD@XG
\expandafter\CD@LJ\fi}\def\CD@VJ#1,{\def\CD@GK{#1}\ifx\CD@GK\CD@RD\else\ifnum
\CD@NB>\CD@CA\CD@NF\CD@GK\DiagramCellHeight\advance\CD@NF-\dp\CD@NB\advance
\CD@NB\m@ne\ht\CD@NB\CD@NF\fi\expandafter\CD@VJ\fi}\def\CD@pJ{\CD@wE\CD@OA
\dimen0 \advance\CD@OA-\DiagramCellWidth\ifdim\CD@OA>2\MapShortFall\CD@KB{%
badly drawn diagonals (see manual)}\let\CD@pJ\empty\fi\else\let\CD@pJ\empty
\fi}\def\CD@KC{\CD@VA\CD@mA\CD@sF\ifnum\CD@VA<\CD@MB\dimen0\dp\CD@VA\advance
\dimen0\CD@NF\dp\CD@VA\dimen0 \advance\CD@VA1 \repeat}\def\CD@bH#1#2#3#4{%
\ifnum\CD@FA<\CD@LB\CD@OA=#1\relax\setbox\CD@FA=\hbox{\setbox0=#2\dimen7=#4%
\relax\dimen8=#3\relax\ifhbox\CD@FA\unhbox\CD@FA\advance\CD@OA-\lastkern
\unkern\fi\ifdim\CD@OA=\z@\else\kern-\CD@OA\fi\raise\dimen7\box0 \kern-\dimen
8 }\ifnum\CD@FA=\CD@lA\CD@V\CD@kF\CD@OA\fi\else\cd@shouldnt O\fi}\def\CD@w{%
\setbox\CD@NB=\hbox{\CD@FA\CD@lA\CD@VA\CD@mA\CD@PA\z@\relax\CD@sF\ifnum\CD@FA
<\CD@LB\CD@tB\wd\CD@VA\relax\CD@eI\advance\CD@FA1 \advance\CD@VA1 \repeat}%
\CD@V\CD@kI{\wd\CD@NB}\wd\CD@NB\z@}\def\CD@eI{\ifhbox\CD@FA\CD@OA\CD@tB\relax
\advance\CD@OA-\CD@PA\relax\ifdim\CD@OA=\z@\else\kern\CD@OA\fi\CD@PA\CD@tB
\advance\CD@PA\wd\CD@FA\relax\unhbox\CD@FA\advance\CD@PA-\lastkern\unkern\fi}%
\def\CD@ZI{\setbox\CD@sH=\box\voidb@x\CD@VA=\CD@MB\CD@FA\CD@LB\CD@VA\CD@mA
\advance\CD@VA\CD@FA\advance\CD@VA-\CD@lA\advance\CD@VA\m@ne\CD@tB\wd\CD@VA
\count@\CD@LB\advance\count@\m@ne\CD@hF.5\wd\count@\advance\CD@hF\CD@tB\CD@A
\m@ne\CD@gD\@m\CD@sF\ifnum\CD@FA>\CD@lA\advance\CD@FA\m@ne\advance\CD@hF-%
\CD@tB\CD@PI\wd\CD@VA\CD@tB\advance\CD@hF\CD@tB\advance\CD@VA\m@ne\CD@tB\wd
\CD@VA\repeat\CD@mF\CD@kF\CD@mI-\CD@mF\CD@vB}\newcount\CD@GB\def\CD@s{}\def
\CD@t{\CD@vK\mathsurround\z@\hsize\z@\rightskip\z@ plus1fil minus\maxdimen
\parfillskip\z@\linepenalty9000 \looseness0 \hfuzz\maxdimen\hbadness10000
\clubpenalty0 \widowpenalty0 \displaywidowpenalty0 \interlinepenalty0
\predisplaypenalty0 \postdisplaypenalty0 \interdisplaylinepenalty0
\interfootnotelinepenalty0 \floatingpenalty0 \brokenpenalty0 \everypar{}%
\leftskip\z@\parskip\z@\parindent\z@\pretolerance10000 \tolerance10000
\hyphenpenalty10000 \exhyphenpenalty10000 \binoppenalty10000 \relpenalty10000
\adjdemerits0 \doublehyphendemerits0 \finalhyphendemerits0 \baselineskip\z@
\CD@IA\def\parskip{\cd@shouldnt{PS}}\prevdepth\z@}\newbox\CD@KG\newbox\CD@IG
\def\CD@JG{\unhcopy\CD@KG}\def\CD@HG{\unhcopy\CD@IG}\def\CD@iJ{\hbox{}%
\penalty1\nointerlineskip}\def\CD@PI{\penalty5 \noindent\setbox\CD@MH=\null
\CD@mF\z@\CD@mI\z@\ifnum\CD@FA<\CD@LB\ht\CD@MH\ht\CD@FA\dp\CD@MH\dp\CD@FA
\unhbox\CD@FA\skip0=\lastskip\unskip\else\CD@OK\skip0=\z@\fi\endgraf\ifcase
\prevgraf\cd@shouldnt Y \or\cd@shouldnt Z \or\CD@RI\or\CD@XI\else\CD@QI\fi
\unskip\setbox0=\lastbox\unskip\unskip\unpenalty\noindent\unhbox0\setbox0%
\lastbox\unpenalty\unskip\unskip\unpenalty\setbox0\lastbox\CD@tF\CD@GB
\lastpenalty\unpenalty\ifnum\CD@GB>\z@\setbox\z@\lastbox\CD@lB\repeat\endgraf
\unskip\unskip\unpenalty}\def\CD@YJ{\CD@uA\CD@XB\advance\CD@uA-\CD@NB\CD@vA
\CD@FA\advance\CD@vA-\CD@lA\advance\CD@vA1 \expandafter\message{prevgraf=\the
\prevgraf at (\the\CD@uA,\the\CD@vA)}}\def\CD@XI{\CD@CE\setbox\CD@lI=\lastbox
\setbox\CD@lI=\hbox{\unhbox\CD@lI\unskip\unpenalty}\unskip\ifdim\ht\CD@lI>\ht
\CD@PC\setbox\CD@MH=\copy\CD@lI\else\ifdim\dp\CD@lI>\dp\CD@PC\setbox\CD@MH=%
\copy\CD@lI\else\CD@FG\CD@lI\fi\fi\advance\CD@mF.5\wd\CD@lI\advance\CD@mI.5%
\wd\CD@lI\setbox\CD@lI=\hbox{\unhbox\CD@lI\CD@HG}\CD@bH\CD@mF{\box\CD@lI}%
\CD@mI\z@\CD@yB\CD@vB}\def\CD@CE{\ifnum\CD@A>0 \advance\dimen0-\CD@tB\CD@iA-.%
5\dimen0 \CD@A-\CD@A\else\CD@A0 \CD@iA\z@\fi\setbox\CD@MH=\lastbox\setbox
\CD@MH=\hbox{\unhbox\CD@MH\unskip\unskip\unpenalty\setbox0=\lastbox\global
\CD@QA\lastkern\unkern}\advance\CD@iA-.5\CD@QA\unskip\setbox\CD@MH=\null
\CD@mI\CD@iA\CD@mF-\CD@iA}\def\CD@Z{\ht\CD@MH\CD@tI\dp\CD@MH\CD@sI}\def\CD@FG
#1{\setbox\CD@MH=\hbox{\CD@V{\ht\CD@MH}{\ht#1}\CD@V{\dp\CD@MH}{\dp#1}\CD@V{%
\wd\CD@MH}{\wd#1}\vrule height\ht\CD@MH depth\dp\CD@MH width\wd\CD@MH}}\def
\CD@QI{\CD@CE\CD@Z\setbox\CD@lI=\lastbox\unskip\setbox\CD@lF=\lastbox\unskip
\setbox\CD@lF=\hbox{\unhbox\CD@lF\unskip\global\CD@yA\lastpenalty\unpenalty}%
\advance\CD@yA9999 \ifcase\CD@yA\CD@VI\or\CD@YI\or\CD@TI\or\CD@dI\or\CD@cI\or
\CD@SI\else\cd@shouldnt9\fi}\def\CD@VI{\CD@FG\CD@lI\CD@UI\setbox\CD@sH=\box
\CD@lF\setbox\CD@tH=\box\CD@lI}\def\CD@YI{\CD@FG\CD@lF\setbox\CD@lI\hbox{%
\penalty8 \unhbox\CD@lI\unskip\unpenalty\ifnum\lastpenalty=8 \else\CD@xH\fi}%
\CD@UI\setbox\CD@lF=\hbox{\unhbox\CD@lF\unskip\unpenalty\global\setbox\CD@DA=%
\lastbox}\ifdim\wd\CD@lF=\z@\else\CD@xH\fi\setbox\CD@sH=\box\CD@DA}\def\CD@xH
{\CD@KB{extra material in \string\pile\space cell (lost)}}\def\CD@UI{\CD@yB
\ifvoid\CD@sH\else\CD@KB{Clashing horizontal arrows}\CD@mI.5\CD@hF\CD@mF-%
\CD@mI\CD@vB\CD@mI\z@\CD@mF\z@\fi\CD@hI\CD@hF\advance\CD@hI-\CD@mI\CD@hF-%
\CD@mF\CD@JC\CD@FA}\def\CD@RI{\setbox0\lastbox\unskip\CD@iA\z@\CD@Z\ifdim
\skip0>\z@\CD@tJ\CD@A0 \else\ifnum\CD@A<1 \CD@A0 \dimen0\CD@tB\fi\advance
\CD@A1 \fi}\def\VonH{\CD@MA46\VonH{.5\CD@LF}}\def\HonV{\CD@MA57\HonV{.5\CD@LF
}}\def\HmeetV{\CD@MA44\HmeetV{-\MapShortFall}}\def\CD@MA#1#2#3#4{\CD@pB34#1{%
\string#3}\CD@SD\CD@GB-999#2 \dimen0=#4\CD@tI\dimen0\advance\CD@tI\axisheight
\CD@sI\dimen0\advance\CD@sI-\axisheight\CD@CF\CD@HC\CD@ZD}\def\CD@HC#1{%
\setbox0=\hbox{\CD@k#1\CD@ND}\dimen0.5\wd0 \CD@tI\ht0 \CD@sI\dp0 \CD@ZD}\def
\CD@SD{\setbox0=\null\ht0=\CD@tI\dp0=\CD@sI\wd0=\dimen0 \copy0\penalty\CD@GB
\box0 }\def\CD@TI{\CD@GC\CD@yB}\def\CD@dI{\CD@GC\CD@vB}\def\CD@SI{\CD@GC
\CD@yB\CD@vB}\def\CD@GC{\setbox\CD@lI=\hbox{\unhbox\CD@lI}\setbox\CD@lF=\hbox
{\unhbox\CD@lF\global\setbox\CD@DA=\lastbox}\ht\CD@MH\ht\CD@DA\dp\CD@MH\dp
\CD@DA\advance\CD@mF\wd\CD@DA\advance\CD@mI\wd\CD@lI}\CD@tG
\ifPositiveGradient\CD@CI\CD@BI\CD@CI\CD@tG\ifClimbing\CD@rB\CD@qB\CD@rB
\newcount\DiagonalChoice\DiagonalChoice\m@ne\ifx\tenln\nullfont\CD@tJ\def
\CD@qF{\CD@KH\ifPositiveGradient/\else\CD@k\backslash\CD@ND\fi}\else\def
\CD@qF{\CD@rF\char\count@}\fi\let\CD@rF\tenln\def\Use@line@char#1{\hbox{#1%
\CD@rF\ifPositiveGradient\else\advance\count@64 \fi\char\count@}}\def\CD@cF{%
\Use@line@char{\count@\CD@TC\multiply\count@8\advance\count@-9\advance\count@
\CD@LH}}\def\CD@ZF{\Use@line@char{\ifcase\DiagonalChoice\CD@gF\or\CD@fF\or
\CD@fF\else\CD@gF\fi}}\def\CD@gF{\ifnum\CD@TC=\z@\count@'33 \else\count@
\CD@TC\multiply\count@\sixt@@n\advance\count@-9\advance\count@\CD@LH\advance
\count@\CD@LH\fi}\def\CD@fF{\count@'\ifcase\CD@LH55\or\ifcase\CD@TC66\or22\or
52\or61\or72\fi\or\ifcase\CD@TC66\or25\or22\or63\or52\fi\or\ifcase\CD@TC66\or
16\or36\or22\or76\fi\or\ifcase\CD@TC66\or27\or25\or67\or22\fi\fi\relax}\def
\CD@uC#1{\hbox{#1\setbox0=\Use@line@char{#1}\ifPositiveGradient\else\raise.3%
\ht0\fi\copy0 \kern-.7\wd0 \ifPositiveGradient\raise.3\ht0\fi\box0}}\def
\CD@jF#1{\hbox{\setbox0=#1\kern-.75\wd0 \vbox to.25\ht0{\ifPositiveGradient
\else\vss\fi\box0 \ifPositiveGradient\vss\fi}}}\def\CD@jI#1{\hbox{\setbox0=#1%
\dimen0=\wd0 \vbox to.25\ht0{\ifPositiveGradient\vss\fi\box0
\ifPositiveGradient\else\vss\fi}\kern-.75\dimen0 }}\CD@RC{+h:>}{%
\Use@line@char\CD@fF}\CD@RC{-h:>}{\Use@line@char\CD@gF}\CD@nF{+t:<}{-h:>}%
\CD@nF{-t:<}{+h:>}\CD@RC{+t:>}{\CD@jF{\Use@line@char\CD@fF}}\CD@RC{-t:>}{%
\CD@jI{\Use@line@char\CD@gF}}\CD@nF{+h:<}{-t:>}\CD@nF{-h:<}{+t:>}\CD@RC{+h:>>%
}{\CD@uC\CD@fF}\CD@RC{-h:>>}{\CD@uC\CD@gF}\CD@nF{+t:<<}{-h:>>}\CD@nF{-t:<<}{+%
h:>>}\CD@nF{+h:>->}{+h:>>}\CD@nF{-h:>->}{-h:>>}\CD@nF{+t:<-<}{-h:>>}\CD@nF{-t%
:<-<}{+h:>>}\CD@RC{+t:>>}{\CD@jF{\CD@uC\CD@fF}}\CD@RC{-t:>>}{\CD@jI{\CD@uC
\CD@gF}}\CD@nF{+h:<<}{-t:>>}\CD@nF{-h:<<}{+t:>>}\CD@nF{+t:>->}{+t:>>}\CD@nF{-%
t:>->}{-t:>>}\CD@nF{+h:<-<}{-t:>>}\CD@nF{-h:<-<}{+t:>>}\CD@RC{+f:-}{\CD@EF
\null\else\CD@cF\fi}\CD@nF{-f:-}{+f:-}\def\CD@tC#1#2{\vbox to#1{\vss\hbox to#%
2{\hss.\hss}\vss}}\def\hfdot{\CD@tC{2\axisheight}{.5em}}%
\def\vfdot{\CD@tC{1ex}\z@}
\def\CD@bF{\hbox{\dimen0=.3\CD@zC\dimen1\dimen0 \ifnum\CD@LH>\CD@TC\CD@iC{%
\dimen1}\else\CD@dG{\dimen0}\fi\CD@tC{\dimen0}{\dimen1}}}\newarrowfiller{.}%
\hfdot\hfdot\vfdot\vfdot\def\dfdot{\CD@bF\CD@CK}\CD@RC{+f:.}{\dfdot}\CD@RC{-f%
:.}{\dfdot}\def\CD@@K#1{\hbox\bgroup\def\CD@CH{#1\egroup}\afterassignment
\CD@CH
\count@='}\def\lnchar{\CD@@K\CD@qF}\def\CD@dF#1{\setbox#1=\hbox{\dimen5\dimen
#1 \setbox8=\box#1 \dimen1\wd8 \count@\dimen5 \divide\count@\dimen1 \ifnum
\count@=0 \box8 \ifdim\dimen5<.95\dimen1 \CD@gB{diagonal line too short}\fi
\else\dimen3=\dimen5 \advance\dimen3-\dimen1 \divide\dimen3\count@\dimen4%
\dimen3 \CD@dG{\dimen4}\ifPositiveGradient\multiply\dimen4\m@ne\fi\dimen6%
\dimen1 \advance\dimen6-\dimen3 \loop\raise\count@\dimen4\copy8 \ifnum\count@
>0 \kern-\dimen6 \advance\count@\m@ne\repeat\fi}}\def\CD@CG#1{\CD@EF\CD@xJ{#1%
}\else\CD@dF{#1}\fi}\def\CD@IH#1{}\newdimen\objectheight\objectheight1.8ex
\newdimen\objectwidth\objectwidth1em \def\CD@YD{\dimen6=\CD@aK
\DiagramCellHeight\dimen7=\CD@WK\DiagramCellWidth\CD@KJ\ifnum\CD@LH>0 \ifnum
\CD@TC>0 \CD@aF\else\aftergroup\CD@VC\fi\else\aftergroup\CD@UC\fi}\def\CD@VC{%
\CD@YA{diagonal map is nearly vertical}\CD@NA}\def\CD@UC{\CD@YA{diagonal map
is nearly horizontal}\CD@NA}\CD@rG\CD@NA{Use an orthogonal map instead}\def
\CD@aF{\CD@MJ\dimen3\dimen7\dimen7\dimen6\CD@iC{\dimen7}\advance\dimen3-%
\dimen7 \CD@MF\ifnum\CD@LH>\CD@TC\advance\dimen6-\dimen1\advance\dimen6-%
\dimen5 \CD@iC{\dimen1}\CD@iC{\dimen5}\else\dimen0\dimen1\advance\dimen0%
\dimen5\CD@dG{\dimen0}\advance\dimen6-\dimen0 \fi\dimen2.5\dimen7\advance
\dimen2-\dimen1 \dimen4.5\dimen7\advance\dimen4-\dimen5 \ifPositiveGradient
\dimen0\dimen5 \advance\dimen1-\CD@WK\DiagramCellWidth\advance\dimen1 \CD@ZK
\DiagramCellWidth\setbox6=\llap{\unhbox6\kern.1\ht2}\setbox7=\rlap{\kern.1\ht
2\unhbox7}\else\dimen0\dimen1 \advance\dimen1-\CD@ZK\DiagramCellWidth\setbox7%
=\llap{\unhbox7\kern.1\ht2}\setbox6=\rlap{\kern.1\ht2\unhbox6}\fi\setbox6=%
\vbox{\box6\kern.1\wd2}\setbox7=\vtop{\kern.1\wd2\box7}\CD@dG{\dimen0}%
\advance\dimen0-\axisheight\advance\dimen0-\CD@bK\DiagramCellHeight\dimen5-%
\dimen0 \advance\dimen0\dimen6 \advance\dimen1.5\dimen3 \ifdim\wd3>\z@\ifdim
\ht3>-\dp3\CD@TB\fi\fi\dimen3\dimen2 \dimen7\dimen2\advance\dimen7\dimen4
\ifvoid3 \else\CD@tE\else\dimen8\ht3\advance\dimen8-\axisheight\CD@iC{\dimen8%
}\CD@X{\dimen8}{.5\wd3}\dimen9\dp3\advance\dimen9\axisheight\CD@iC{\dimen9}%
\CD@X{\dimen9}{.5\wd3}\ifPositiveGradient\advance\dimen2-\dimen9\advance
\dimen4-\dimen8 \else\advance\dimen4-\dimen9\advance\dimen2-\dimen8 \fi\fi
\advance\dimen3-.5\wd3 \fi\dimen9=\CD@aK\DiagramCellHeight\advance\dimen9-2%
\DiagramCellHeight\CD@tE\advance\dimen2\dimen4 \CD@CG{2}\dimen2-\dimen0%
\advance\dimen2\dp2 \else\CD@CG{2}\CD@CG{4}\ifPositiveGradient\dimen2-\dimen0%
\advance\dimen2\dp2 \dimen4\dimen5\advance\dimen4-\ht4 \else\dimen4-\dimen0%
\advance\dimen4\dp4 \dimen2\dimen5\advance\dimen2-\ht2 \fi\fi\setbox0=\hbox to%
\z@{\kern\dimen1 \ifvoid1 \else\ifPositiveGradient\advance\dimen0-\dp1 \lower
\dimen0 \else\advance\dimen5-\ht1 \raise\dimen5 \fi\rlap{\unhbox1}\fi\raise
\dimen2\rlap{\unhbox2}\ifvoid3 \else\lower.5\dimen9\rlap{\kern\dimen3\unhbox3%
}\fi\kern.5\dimen7 \lower.5\dimen9\box6 \lower.5\dimen9\box7 \kern.5\dimen7
\CD@tE\else\raise\dimen4\llap{\unhbox4}\fi\ifvoid5 \else\ifPositiveGradient
\advance\dimen5-\ht5 \raise\dimen5 \else\advance\dimen0-\dp5 \lower\dimen0 \fi
\llap{\unhbox5}\fi\hss}\ht0=\axisheight\dp0=-\ht0\box0 }\def\NorthWest{\CD@BI
\CD@rB\DiagonalChoice0 }\def\NorthEast{\CD@CI\CD@rB\DiagonalChoice1 }\def
\SouthWest{\CD@CI\CD@qB\DiagonalChoice3 }\def\SouthEast{\CD@BI\CD@qB
\DiagonalChoice2 }\def\CD@aD{\vadjust{\CD@uA\CD@FA\advance\CD@uA
\ifPositiveGradient\else-\fi\CD@ZK\relax\CD@vA\CD@NB\advance\CD@vA-\CD@bK
\relax\hbox{\advance\CD@uA\ifPositiveGradient-\fi\CD@WK\advance\CD@vA\CD@aK
\hbox{\box6 \kern\CD@DC\kern\CD@eJ\penalty1 \box7 \box\z@}\penalty\CD@uA
\penalty\CD@vA}\penalty\CD@uA\penalty\CD@vA\penalty104}}\def\CD@eH#1{\relax
\vadjust{\hbox@maths{#1}\penalty\CD@FA\penalty\CD@NB\penalty\tw@}}\def\CD@lB{%
\ifcase\CD@GB\or\or\CD@bH{.5\wd0}{\box0}{.5\wd0}\z@\or\unhbox\z@\setbox\z@
\lastbox\CD@bH{.5\wd0}{\box0}{.5\wd0}\z@\unpenalty\unpenalty\setbox\z@
\lastbox\or\CD@TG\else\advance\CD@GB-100 \ifnum\CD@GB<\z@\cd@shouldnt B\fi
\setbox\z@\hbox{\kern\CD@mF\copy\CD@MH\kern\CD@mI\CD@uA\CD@XB\advance\CD@uA-%
\CD@NB\penalty\CD@uA\CD@uA\CD@FA\advance\CD@uA-\CD@lA\penalty\CD@uA\unhbox\z@
\global\CD@yA\lastpenalty\unpenalty\global\CD@zA\lastpenalty\unpenalty}\CD@uA
-\CD@yA\CD@vA\CD@zA\CD@fI\fi}\def\CD@TG{\unhbox\z@\setbox\z@\lastbox\CD@uA
\lastpenalty\unpenalty\advance\CD@uA\CD@mA\CD@vA\CD@XB\advance\CD@vA-%
\lastpenalty\unpenalty\dimen1\lastkern\unkern\setbox3\lastbox\dimen0\lastkern
\unkern\setbox0=\hbox to\z@{\unhbox0\setbox0\lastbox\setbox7\lastbox
\unpenalty\CD@eJ\lastkern\unkern\CD@DC\lastkern\unkern\setbox6\lastbox\dimen7%
\CD@tB\advance\dimen7-\wd\CD@uA\ifdim\dimen7<\z@\CD@CI\multiply\dimen7\m@ne
\let\mv\empty\else\CD@BI\def\mv{\raise\ht1}\kern-\dimen7 \fi\ifnum\CD@vA>%
\CD@NB\dimen6\CD@uB\advance\dimen6-\ht\CD@vA\else\dimen6\z@\fi\CD@jJ\CD@mK
\setbox1\null\ht1\dimen6\wd1\dimen7 \dimen7\dimen2 \dimen6\wd1 \CD@KJ\CD@uA
\CD@LH\CD@vA\CD@TC\dimen6\ht1 \CD@KJ\setbox2\null\divide\dimen2\tw@\advance
\dimen2\CD@eJ\CD@eG{\dimen2}\wd2\dimen2 \dimen0.5\dimen7 \advance\dimen0%
\ifPositiveGradient\else-\fi\CD@eJ\CD@dG{\dimen0}\advance\dimen0-\axisheight
\ht2\dimen0 \dimen0\CD@DC\CD@eG{\dimen0}\advance\dimen0\ht2\ht2\dimen0 \dimen
0\ifPositiveGradient-\fi\CD@DC\CD@dG{\dimen0}\advance\dimen0\wd2\wd2\dimen0
\setbox4\null\dimen0 .6\CD@zC\CD@eG{\dimen0}\ht4\dimen0 \dimen0 .2\CD@zC
\CD@dG{\dimen0}\wd4\dimen0 \dimen0\wd2 \ifvoid6\else\dimen1\ht4 \advance
\dimen1\ht2 \CD@CC6+-\raise\dimen1\rlap{\ifPositiveGradient\advance\dimen0-%
\wd6\advance\dimen0-\wd4 \else\advance\dimen0\wd4 \fi\kern\dimen0\box6}\fi
\dimen0\wd2 \ifvoid7\else\dimen1\ht4 \advance\dimen1-\ht2 \CD@CC7-+\lower
\dimen1\rlap{\ifPositiveGradient\advance\dimen0\wd4 \else\advance\dimen0-\wd7%
\advance\dimen0-\wd4 \fi\kern\dimen0\box7}\fi\mv\box0\hss}\ht0\z@\dp0\z@
\CD@bH{\z@}{\box\z@}{\z@}{\axisheight}}\def\CD@CC#1#2#3{\dimen4.5\wd#1 \ifdim
\dimen4>.25\dimen7\dimen4=.25\dimen7\fi\ifdim\dimen4>\CD@zC\dimen4.4\dimen4
\advance\dimen4.6\CD@zC\fi\CD@eG{\dimen4}\dimen5\axisheight\CD@dG{\dimen5}%
\advance\dimen4-\dimen5 \dimen5\dimen4\CD@eG{\dimen5}\advance\dimen0%
\ifPositiveGradient#2\else#3\fi\dimen5 \CD@dG{\dimen4}\advance\dimen1\dimen4 }
\def\CD@eD#1{\expandafter\CD@IK{#1}}\CD@ZA\CD@EK{output is PostScript
dependent}\def\CD@SC{\CD@IK{/bturn {gsave currentpoint currentpoint translate
4 2 roll neg exch atan rotate neg exch neg exch translate } def /eturn {%
currentpoint grestore moveto} def}}\def\CD@gK{\relax\CD@hK\CD@tK{Q}\else
\CD@IK{eturn}\fi} \def\CD@OJ#1{\count@#1\relax\multiply\count@7\advance
\count@16577\divide\count@33154 }\def\CD@fD#1{\expandafter\special{#1}} \def
\CD@xJ#1{\setbox#1=\hbox{\dimen0\dimen#1\CD@dG{\dimen0}\CD@OJ{\dimen0}\setbox
0=\null\ifPositiveGradient\count@-\count@\ht0\dimen0 \else\dp0\dimen0 \fi\box
0 \CD@uA\count@\CD@OJ\CD@LF\CD@fD{pn \the\count@}\CD@fD{pa 0 0}\CD@OJ{\dimen#%
1}\CD@fD{pa \the\count@\space\the\CD@uA}\CD@fD{fp}\kern\dimen#1}}\def\CD@JI{%
\CD@KJ\begingroup\ifdim\dimen7<\dimen6 \dimen2=\dimen6 \dimen6=\dimen7 \dimen
7=\dimen2 \count@\CD@LH\CD@LH\CD@TC\CD@TC\count@\else\dimen2=\dimen7 \fi
\ifdim\dimen6>.01\p@\CD@KI\global\CD@QA\dimen0 \else\global\CD@QA\dimen7 \fi
\endgroup\dimen2\CD@QA\CD@iK\CD@lK{\ifPositiveGradient\else-\fi\dimen6}\CD@iK
\CD@kK{\ifPositiveGradient-\fi\dimen6}\CD@iK\CD@eK{\dimen7}}\def\CD@KI{\CD@hJ
\ifdim\dimen7>1.73\dimen6 \divide\dimen2 4 \multiply\CD@TC2 \else\dimen2=0.%
353553\dimen2 \advance\CD@LH-\CD@TC\multiply\CD@TC4 \fi\dimen0=4\dimen2 \CD@ZG
4\CD@ZG{-2}\CD@ZG2\CD@ZG{-2.5}}\def\CD@AI{\begingroup\count@\dimen0 \dimen2 45%
pt \divide\count@\dimen2 \ifdim\dimen0<\z@\advance\count@\m@ne\fi\ifodd
\count@\advance\count@1\CD@@A\else\CD@y\fi\advance\dimen0-\count@\dimen2
\CD@gE\multiply\dimen0\m@ne\fi\ifnum\count@<0 \multiply\count@-7 \fi\dimen3%
\dimen1 \dimen6\dimen0 \dimen7 3754936sp \ifdim\dimen0<6\p@\def\CD@OG{4000}%
\fi\CD@KJ\dimen2\dimen3\CD@dG{\dimen2}\CD@hJ\multiply\CD@TC-6 \dimen0\dimen2
\CD@ZG1\CD@ZG{0.3}\dimen1\dimen0 \dimen2\dimen3 \dimen0\dimen3 \CD@ZG3\CD@ZG{%
1.5}\CD@ZG{0.3}\divide\count@2 \CD@gE\multiply\dimen1\m@ne\fi\ifodd\count@
\dimen2\dimen1\dimen1\dimen0\dimen0-\dimen2 \fi\divide\count@2 \ifodd\count@
\multiply\dimen0\m@ne\multiply\dimen1\m@ne\fi\global\CD@QA\dimen0\global
\CD@RA\dimen1\endgroup\dimen6\CD@QA\dimen7\CD@RA}\def\CD@OC{255}\let\CD@OG
\CD@OC\def\CD@KJ{\begingroup\ifdim\dimen7<\dimen6 \dimen9\dimen7\dimen7\dimen
6\dimen6\dimen9\CD@@A\else\CD@y\fi\dimen2\z@\dimen3\CD@XH\dimen4\CD@XH\dimen0%
\z@\dimen8=\CD@OG\CD@XH\CD@lC\global\CD@yA\dimen\CD@gE0\else3\fi\global\CD@zA
\dimen\CD@gE3\else0\fi\endgroup\CD@LH\CD@yA\CD@TC\CD@zA}\def\CD@lC{\count@
\dimen6 \divide\count@\dimen7 \advance\dimen6-\count@\dimen7 \dimen9\dimen4
\advance\dimen9\count@\dimen0 \ifdim\dimen9>\dimen8 \CD@@C\else\CD@AC\ifdim
\dimen6>\z@\dimen9\dimen6 \dimen6\dimen7 \dimen7\dimen9 \expandafter
\expandafter\expandafter\CD@lC\fi\fi}\def\CD@@C{\ifdim\dimen0=\z@\ifdim\dimen
9<2\dimen8 \dimen0\dimen8 \fi\else\advance\dimen8-\dimen4 \divide\dimen8%
\dimen0 \ifdim\count@\CD@XH<2\dimen8 \count@\dimen8 \dimen9\dimen4 \advance
\dimen9\count@\dimen0 \CD@AC\fi\fi}\def\CD@AC{\dimen4\dimen0 \dimen0\dimen9
\advance\dimen2\count@\dimen3 \dimen9\dimen2 \dimen2\dimen3 \dimen3\dimen9 }%
\def\CD@ZG#1{\CD@dG{\dimen2}\advance\dimen0 #1\dimen2 }\def\CD@dG#1{\divide#1%
\CD@TC\multiply#1\CD@LH}\def\CD@eG#1{\divide#1\CD@vA\multiply#1\CD@uA}\def
\CD@iC#1{\divide#1\CD@LH\multiply#1\CD@TC}\def\CD@hJ{\dimen6\CD@LH\CD@XH
\multiply\dimen6\CD@LH\dimen7\CD@TC\CD@XH\multiply\dimen7\CD@TC\CD@KJ}\def
\CD@iK#1#2{\begingroup\dimen@#2\relax\loop\ifdim\dimen2<.4\maxdimen\multiply
\dimen2\tw@\multiply\dimen@\tw@\repeat\divide\dimen2\@cclvi\divide\dimen@
\dimen2\relax\multiply\dimen@\@cclvi\expandafter\CD@jK\the\dimen@\endgroup
\let#1\CD@fK}{\catcode`p=12 \catcode`0=12 \catcode`.=12 \catcode`t=12 \gdef
\CD@jK#1pt{\gdef\CD@fK{#1}}}\ifx\errorcontextlines\CD@qK\CD@tJ\let\CD@GH
\relax\else\def\CD@GH{\errorcontextlines\m@ne}\fi\ifnum\inputlineno<0 \let
\CD@CD\empty\let\CD@W\empty\let\CD@mD\relax\let\CD@uI\relax\let\CD@vI\relax
\let\CD@zF\relax\else\def\CD@W{ at line \number\inputlineno}\def\CD@mD{ - first occurred%
}\def\CD@uI{\edef\CD@h{\the\inputlineno}\global\let\CD@jB\CD@h}\def\CD@h{9999%
}\def\CD@vI{\xdef\CD@jB{\the\inputlineno}}\def\CD@jB{\CD@h}\def\CD@zF{\ifnum
\CD@h<\inputlineno\edef\CD@CD{\space at lines \CD@h--\the\inputlineno}\else
\edef\CD@CD{\CD@W}\fi}\fi\let\CD@CD\empty\def\CD@YA#1#2{\CD@GH\errhelp=#2%
\expandafter\errmessage{\CD@tA: #1}}\def\CD@KB#1{\begingroup\expandafter
\message{! \CD@tA: #1\CD@CD}\ifnum\CD@XB>\CD@NB\ifnum\CD@CA>\CD@NB\else\ifnum
\CD@lA>\CD@FA\else\ifnum\CD@LB>\CD@FA\advance\CD@XB-\CD@NB\advance\CD@FA-%
\CD@lA\advance\CD@FA1\relax\expandafter\message{! (error detected at row \the
\CD@XB, column \the\CD@FA, but probably caused elsewhere)}\fi\fi\fi\fi
\endgroup}\def\CD@gB#1{{\expandafter\message{\CD@tA\space Warning: #1\CD@W}}}%
\def\CD@CB#1#2{\CD@gB{#1 \string#2 is obsolete\CD@mD}}\def\CD@AB#1{\CD@CB{%
Dimension}{#1}\CD@DE#1\CD@BB\CD@BB}\def\CD@BB{\CD@OA=}\def\CD@@B#1{\CD@CB{%
Count}{#1}\CD@DE#1\CD@OH\CD@OH}\def\CD@OH{\count@=}\def\HorizontalMapLength{%
\CD@AB\HorizontalMapLength}\def\VerticalMapHeight{\CD@AB\VerticalMapHeight}%
\def\VerticalMapDepth{\CD@AB\VerticalMapDepth}\def\VerticalMapExtraHeight{%
\CD@AB\VerticalMapExtraHeight}\def\VerticalMapExtraDepth{\CD@AB
\VerticalMapExtraDepth}\def\DiagonalLineSegments{\CD@@B\DiagonalLineSegments}%
\ifx\tenln\nullfont\CD@ZA\CD@KH{\CD@eF\space diagonal line and arrow font not
available}\else\let\CD@KH\relax\fi\def\CD@aG#1#2<#3:#4:#5#6{\begingroup\CD@PA
#3\relax\advance\CD@PA-#2\relax\ifdim.1em<\CD@PA\CD@uA#5\relax\CD@vA#6\relax
\ifnum\CD@uA<\CD@vA\count@\CD@vA\advance\count@-\CD@uA\CD@KB{#4 by \the\CD@PA
}\if#1v\let\CD@CH\CD@JK\edef\tmp{\the\CD@uA--\the\CD@vA,\the\CD@FA}\else
\advance\count@\count@\if#1l\advance\count@-\CD@A\else\if#1r\advance\count@
\CD@A\fi\fi\advance\CD@PA\CD@PA\let\CD@CH\CD@ZE\edef\tmp{\the\CD@NB,\the
\CD@uA--\the\CD@vA}\fi\divide\CD@PA\count@\ifdim\CD@CH<\CD@PA\global\CD@CH
\CD@PA\fi\fi\fi\endgroup}\CD@tG\CD@xE\CD@JD\CD@ID\CD@rG\CD@xI{See the message
above.}\CD@rG\CD@lH{Perhaps you've forgotten to end the diagram before
resuming the text, in\CD@uG which case some garbage may be added to the
diagram, but we should be ok now.\CD@uG Alternatively you've left a blank line
in the middle - TeX will now complain\CD@uG that the remaining \CD@S s are
misplaced - so please use comments for layout.}\CD@rG\CD@hD{You have already
closed too many brace pairs or environments; an \CD@HD\CD@uG command was (%
over)due.}\CD@rG\CD@hH{\CD@dC\space and \CD@HD\space commands must match.}%
\def\CD@jH{\ifnum\inputlineno=0 \else\expandafter\CD@iH\fi}\def\CD@iH{\CD@MD
\CD@GD\crcr\CD@YA{missing \CD@HD\space inserted before \CD@kH- type "h"}%
\CD@lH\enddiagram\CD@AG\CD@kH\par}\def\CD@AG#1{\edef\enddiagram{\noexpand
\CD@rD{#1\CD@W}}}\def\CD@rD#1{\CD@YA{\CD@HD\space(anticipated by #1) ignored}%
\CD@xI\let\enddiagram\CD@SG}\def\CD@SG{\CD@YA{misplaced \CD@HD\space ignored}%
\CD@hH}\def\CD@mC{\CD@YA{missing \CD@HD\space inserted.}\CD@hD\CD@AG{closing
group}}\ifx\ifx
\fi\fi

\catcode`\$=3 
\def\vboxtoz{\vbox to\z@}

\def\scriptaxis#1{\@scriptaxis{$\scriptstyle#1$}}
\def\ssaxis#1{\@ssaxis{$\scriptscriptstyle#1$}}
\def\@scriptaxis#1{\dimen0\axisheight\advance\dimen0-\ss@axisheight\raise
\dimen0\hbox{#1}}\def\@ssaxis#1{\dimen0\axisheight\advance\dimen0-%
\ss@axisheight\raise\dimen0\hbox{#1}}

\ifx\boldmath\CD@qK
\let\boldscriptaxis\scriptaxis
\def\boldscript#1{\hbox{$\scriptstyle#1$}}
\else\def\boldscriptaxis#1{\@scriptaxis{\boldmath$\scriptstyle#1$}}
\def\boldscript#1{\hbox{\boldmath$\scriptstyle#1$}}
\fi

\def\raisehook#1#2#3{\hbox{\setbox3=\hbox{#1$\scriptscriptstyle#3$}%
\dimen0\ss@axisheight
\dimen1\axisheight\advance\dimen1-\dimen0
\dimen2\ht3\advance\dimen2-\dimen0%
\advance\dimen2-0.021em\advance\dimen1 #2\dimen2%
\raise\dimen1\box3}}
\def\shifthook#1#2#3{\setbox1=\hbox{#1$\scriptscriptstyle#3$}\dimen0\wd1%
\divide\dimen0 12\CD@zH{\dimen0}
\dimen1\wd1\advance\dimen1-2\dimen0 \advance\dimen1-2\CD@oI\CD@zH{\dimen1}%
\kern#2\dimen1\box1}

\def\@cmex{\mathchar"03}

\def\make@pbk#1{\setbox\tw@\hbox to\z@{#1}\ht\tw@\z@\dp\tw@\z@\box\tw@}\def
\CD@fH#1{\overprint{\hbox to\z@{#1}}}\def\CD@qH{\kern0.11em}\def\CD@pH{\kern0%
.35em}

\def\dblvert{\def\CD@rH{\kern.5\PileSpacing}}\def\CD@rH{}

\def\SEpbk{\make@pbk{\CD@qH\CD@rH\vrule depth 2.87ex height -2.75ex width 0.%
95em \vrule height -0.66ex depth 2.87ex width 0.05em \hss}}

\def\SWpbk{\make@pbk{\hss\vrule height -0.66ex depth 2.87ex width 0.05em
\vrule depth 2.87ex height -2.75ex width 0.95em \CD@qH\CD@rH}}

\def\NEpbk{\make@pbk{\CD@qH\CD@rH\vrule depth -3.81ex height 4.00ex width 0.%
95em \vrule height 4.00ex depth -1.72ex width 0.05em \hss}}

\def\NWpbk{\make@pbk{\hss\vrule height 4.00ex depth -1.72ex width 0.05em
\vrule depth -3.81ex height 4.00ex width 0.95em \CD@qH\CD@rH}}

\def\puncture{{\setbox0\hbox{A}\vrule height.53\ht0 depth-.47\ht0 width.35\ht
0 \kern.12\ht0 \vrule height\ht0 depth-.65\ht0 width.06\ht0 \kern-.06\ht0
\vrule height.35\ht0 depth0pt width.06\ht0 \kern.12\ht0 \vrule height.53\ht0
depth-.47\ht0 width.35\ht0 }}

\def\NEclck{\overprint{\raise2.5ex\rlap{ \CD@rH$\scriptstyle\searrow$}}}
\def\NEanti{\overprint{\raise2.5ex\rlap{ \CD@rH$\scriptstyle\nwarrow$}}}
\def\NWclck{\overprint{\raise2.5ex\llap{$\scriptstyle\nearrow$ \CD@rH}}}
\def\NWanti{\overprint{\raise2.5ex\llap{$\scriptstyle\swarrow$ \CD@rH}}}
\def\SEclck{\overprint{\lower1ex\rlap{ \CD@rH$\scriptstyle\swarrow$}}}
\def\SEanti{\overprint{\lower1ex\rlap{ \CD@rH$\scriptstyle\nearrow$}}}
\def\SWclck{\overprint{\lower1ex\llap{$\scriptstyle\nwarrow$ \CD@rH}}}
\def\SWanti{\overprint{\lower1ex\llap{$\scriptstyle\searrow$ \CD@rH}}}

\def\rhvee{\mkern-10mu\greaterthan}
\def\lhvee{\lessthan\mkern-10mu}
\def\dhvee{\vboxtoz{\vss\hbox{$\vee$}\kern0pt}}
\def\uhvee{\vboxtoz{\hbox{$\wedge$}\vss}}
\newarrowhead{vee}\rhvee\lhvee\dhvee\uhvee

\def\dhlvee{\vboxtoz{\vss\hbox{$\scriptstyle\vee$}\kern0pt}}
\def\uhlvee{\vboxtoz{\hbox{$\scriptstyle\wedge$}\vss}}
\newarrowhead{littlevee}{\mkern1mu\scriptaxis\rhvee}{\scriptaxis\lhvee}%
\dhlvee\uhlvee\ifx\boldmath\CD@qK
\newarrowhead{boldlittlevee}{\mkern1mu\scriptaxis\rhvee}{\scriptaxis\lhvee}%
\dhlvee\uhlvee\else
\def\dhblvee{\vboxtoz{\vss\boldscript\vee\kern0pt}}
\def\uhblvee{\vboxtoz{\boldscript\wedge\vss}}
\newarrowhead{boldlittlevee}{\mkern1mu\boldscriptaxis\rhvee}{\boldscriptaxis
\lhvee}\dhblvee\uhblvee
\fi

\def\rhcvee{\mkern-10mu\succ}
\def\lhcvee{\prec\mkern-10mu}
\def\dhcvee{\vboxtoz{\vss\hbox{$\curlyvee$}\kern0pt}}
\def\uhcvee{\vboxtoz{\hbox{$\curlywedge$}\vss}}
\newarrowhead{curlyvee}\rhcvee\lhcvee\dhcvee\uhcvee

\def\rhvvee{\mkern-13mu\gg}
\def\lhvvee{\ll\mkern-13mu}
\def\dhvvee{\vboxtoz{\vss\hbox{$\vee$}\kern-.6ex\hbox{$\vee$}\kern0pt}}
\def\uhvvee{\vboxtoz{\hbox{$\wedge$}\kern-.6ex \hbox{$\wedge$}\vss}}
\newarrowhead{doublevee}\rhvvee\lhvvee\dhvvee\uhvvee

\def\rhtriangle{\triangleright\mkern1.2mu}
\def\lhtriangle{\triangleleft\mkern.8mu}
\def\uhtriangle{\vbox{\kern-.2ex \hbox{$\scriptscriptstyle\bigtriangleup$}%
\kern-.25ex}}
\def\dhtriangle{\vbox{\kern-.28ex \hbox{$\scriptscriptstyle\bigtriangledown$}%
\kern-.1ex}}
\def\dhblack{\vbox{\kern-.25ex\nointerlineskip\hbox{$\blacktriangledown$}}}%
\def\uhblack{\vbox{\kern-.25ex\nointerlineskip\hbox{$\blacktriangle$}}}%
\def\dhlblack{\vbox{\kern-.25ex\nointerlineskip\hbox{$\scriptstyle
\blacktriangledown$}}}
\def\uhlblack{\vbox{\kern-.25ex\nointerlineskip\hbox{$\scriptstyle
\blacktriangle$}}}
\newarrowhead{triangle}\rhtriangle\lhtriangle\dhtriangle\uhtriangle
\newarrowhead{blacktriangle}{\mkern-1mu\blacktriangleright\mkern.4mu}{%
\blacktriangleleft}\dhblack\uhblack\newarrowhead{littleblack}{\mkern-1mu%
\scriptaxis\blacktriangleright}{\scriptaxis\blacktriangleleft\mkern-2mu}%
\dhlblack\uhlblack

\def\rhla{\hbox{\setbox0=\lnchar55\dimen0=\wd0\kern-.6\dimen0\ht0\z@\raise
\axisheight\box0\kern.1\dimen0}}
\def\lhla{\hbox{\setbox0=\lnchar33\dimen0=\wd0\kern.05\dimen0\ht0\z@\raise
\axisheight\box0\kern-.5\dimen0}}
\def\dhla{\vboxtoz{\vss\rlap{\lnchar77}}}
\def\uhla{\vboxtoz{\setbox0=\lnchar66 \wd0\z@\kern-.15\ht0\box0\vss}}
\newarrowhead{LaTeX}\rhla\lhla\dhla\uhla

\def\lhlala{\lhla\kern.3em\lhla}
\def\rhlala{\rhla\kern.3em\rhla}
\def\uhlala{\hbox{\uhla\raise-.6ex\uhla}}
\def\dhlala{\hbox{\dhla\lower-.6ex\dhla}}
\newarrowhead{doubleLaTeX}\rhlala\lhlala\dhlala\uhlala

\def\hhO{\scriptaxis\bigcirc\mkern.4mu} \def\hho{{\circ}\mkern1.2mu}%
\newarrowhead{o}\hho\hho\circ\circ
\newarrowhead{O}\hhO\hhO{\scriptstyle\bigcirc}{\scriptstyle\bigcirc}

\def\rhtimes{\mkern-5mu{\times}\mkern-.8mu}\def\lhtimes{\mkern-.8mu{\times}%
\mkern-5mu}\def\uhtimes{\setbox0=\hbox{$\times$}\ht0\axisheight\dp0-\ht0%
\lower\ht0\box0 }\def\dhtimes{\setbox0=\hbox{$\times$}\ht0\axisheight\box0 }%
\newarrowhead{X}\rhtimes\lhtimes\dhtimes\uhtimes\newarrowhead+++++

\newarrowhead{Y}{\mkern-3mu\Yright}{\Yleft\mkern-3mu}\Ydown\Yup

\newarrowhead{->}\rightarrow\leftarrow\downarrow\uparrow

\newarrowhead{=>}\Rightarrow\Leftarrow{\@cmex7F}{\@cmex7E}

\newarrowhead{harpoon}\rightharpoonup\leftharpoonup\downharpoonleft
\upharpoonleft

\def\twoheaddownarrow{\rlap{$\downarrow$}\raise-.5ex\hbox{$\downarrow$}}
\def\twoheaduparrow{\rlap{$\uparrow$}\raise.5ex\hbox{$\uparrow$}}
\newarrowhead{->>}\twoheadrightarrow\twoheadleftarrow\twoheaddownarrow
\twoheaduparrow

\def\ltvee{\mkern-1mu{\lessthan}\mkern.4mu}
\newarrowtail{vee}\greaterthan\ltvee\vee\wedge

\newarrowtail{littlevee}{\scriptaxis\greaterthan}{\mkern-1mu\scriptaxis
\lessthan}{\scriptstyle\vee}{\scriptstyle\wedge}\ifx\boldmath\CD@qK
\newarrowtail{boldlittlevee}{\scriptaxis\greaterthan}{\mkern-1mu\scriptaxis
\lessthan}{\scriptstyle\vee}{\scriptstyle\wedge}\else\newarrowtail{%
boldlittlevee}{\boldscriptaxis\greaterthan}{\mkern-1mu\boldscriptaxis
\lessthan}{\boldscript\vee}{\boldscript\wedge}\fi

\newarrowtail{curlyvee}\succ{\mkern-1mu{\prec}\mkern.4mu}\curlyvee\curlywedge

\def\rttriangle{\mkern1.2mu\triangleright}
\newarrowtail{triangle}\rttriangle\lhtriangle\dhtriangle\uhtriangle
\newarrowtail{blacktriangle}\blacktriangleright{\mkern-1mu\blacktriangleleft
\mkern.4mu}\dhblack\uhblack\newarrowtail{littleblack}{\scriptaxis
\blacktriangleright\mkern-2mu}{\mkern-1mu\scriptaxis\blacktriangleleft}%
\dhlblack\uhlblack

\def\rtla{\hbox{\setbox0=\lnchar55\dimen0=\wd0\kern-.5\dimen0\ht0\z@\raise
\axisheight\box0\kern-.2\dimen0}}
\def\ltla{\hbox{\setbox0=\lnchar33\dimen0=\wd0\kern-.15\dimen0\ht0\z@\raise
\axisheight\box0\kern-.5\dimen0}}
\def\dtla{\vbox{\setbox0=\rlap{\lnchar77}\dimen0=\ht0\kern-.7\dimen0\box0%
\kern-.1\dimen0}}
\def\utla{\vbox{\setbox0=\rlap{\lnchar66}\dimen0=\ht0\kern-.1\dimen0\box0%
\kern-.6\dimen0}}
\newarrowtail{LaTeX}\rtla\ltla\dtla\utla

\def\rtvvee{\gg\mkern-3mu}
\def\ltvvee{\mkern-3mu\ll}
\def\dtvvee{\vbox{\hbox{$\vee$}\kern-.6ex \hbox{$\vee$}\vss}}
\def\utvvee{\vbox{\vss\hbox{$\wedge$}\kern-.6ex \hbox{$\wedge$}\kern\z@}}
\newarrowtail{doublevee}\rtvvee\ltvvee\dtvvee\utvvee

\def\ltlala{\ltla\kern.3em\ltla}
\def\rtlala{\rtla\kern.3em\rtla}
\def\utlala{\hbox{\utla\raise-.6ex\utla}}
\def\dtlala{\hbox{\dtla\lower-.6ex\dtla}}
\newarrowtail{doubleLaTeX}\rtlala\ltlala\dtlala\utlala

\def\utbar{\vrule height 0.093ex depth0pt width 0.4em}
\let\dtbar\utbar
\def\rtbar{\mkern1.5mu\vrule height 1.1ex depth.06ex width .04em\mkern1.5mu}%
\let\ltbar\rtbar
\newarrowtail{mapsto}\rtbar\ltbar\dtbar\utbar
\newarrowtail{|}\rtbar\ltbar\dtbar\utbar

\def\rthooka{\raisehook{}+\subset\mkern-1mu}
\def\lthooka{\mkern-1mu\raisehook{}+\supset}
\def\rthookb{\raisehook{}-\subset\mkern-2mu}
\def\lthookb{\mkern-1mu\raisehook{}-\supset}

\def\dthooka{\shifthook{}+\cap}
\def\dthookb{\shifthook{}-\cap}
\def\uthooka{\shifthook{}+\cup}
\def\uthookb{\shifthook{}-\cup}

\newarrowtail{hooka}\rthooka\lthooka\dthooka\uthooka\newarrowtail{hookb}%
\rthookb\lthookb\dthookb\uthookb

\ifx\boldmath\CD@qK\newarrowtail{boldhooka}\rthooka\lthooka\dthooka\uthooka
\newarrowtail{boldhookb}\rthookb\lthookb\dthookb\uthookb\newarrowtail{%
boldhook}\rthooka\lthooka\dthookb\uthooka\else\def\rtbhooka{\raisehook
\boldmath+\subset\mkern-1mu}
\def\ltbhooka{\mkern-1mu\raisehook\boldmath+\supset}
\def\rtbhookb{\raisehook\boldmath-\subset\mkern-2mu}
\def\ltbhookb{\mkern-1mu\raisehook\boldmath-\supset}
\def\dtbhooka{\shifthook\boldmath+\cap}
\def\dtbhookb{\shifthook\boldmath-\cap}
\def\utbhooka{\shifthook\boldmath+\cup}
\def\utbhookb{\shifthook\boldmath-\cup}
\newarrowtail{boldhooka}\rtbhooka\ltbhooka\dtbhooka\utbhooka\newarrowtail{%
boldhookb}\rtbhookb\ltbhookb\dtbhookb\utbhookb\newarrowtail{boldhook}%
\rtbhooka\ltbhooka\dtbhooka\utbhooka\fi

\def\dtsqhooka{\shifthook{}+\sqcap}
\def\ltsqhooka{\mkern-1mu\raisehook{}+\sqsupset}
\def\rtsqhooka{\raisehook{}+\sqsubset\mkern-1mu}
\def\utsqhooka{\shifthook{}+\sqcup}
\newarrowtail{sqhook}\rtsqhooka\ltsqhooka\dtsqhooka\utsqhooka

\newarrowtail{hook}\rthooka\lthookb\dthooka\uthooka\newarrowtail{C}\rthooka
\lthookb\dthooka\uthooka

\newarrowtail{o}\hho\hho\circ\circ
\newarrowtail{O}\hhO\hhO{\scriptstyle\bigcirc}{\scriptstyle\bigcirc}

\newarrowtail{X}\lhtimes\rhtimes\uhtimes\dhtimes\newarrowtail+++++

\newarrowtail{Y}\Yright\Yleft\Ydown\Yup

\newarrowtail{harpoon}\leftharpoondown\rightharpoondown\upharpoonright
\downharpoonright

\newarrowtail{<=}\Leftarrow\Rightarrow{\@cmex7E}{\@cmex7F}

\newarrowfiller{=}=={\@cmex77}{\@cmex77}
\def\vfthree{\mid\!\!\!\mid\!\!\!\mid}
\newarrowfiller{3}\equiv\equiv\vfthree\vfthree

\def\vfdashstrut{\vrule width0pt height1.3ex depth0.7ex}
\def\vfthedash{\vrule width\CD@LF height0.6ex depth 0pt}
\def\hfthedash{\CD@AJ\vrule\horizhtdp width 0.26em}
\def\hfdash{\mkern5.5mu\hfthedash\mkern5.5mu}
\def\vfdash{\vfdashstrut\vfthedash}
\newarrowfiller{dash}\hfdash\hfdash\vfdash\vfdash

\newarrowmiddle+++++

\iffalse
\newarrow{To}----{vee}
\newarrow{Arr}----{LaTeX}
\newarrow{Dotsto}....{vee}
\newarrow{Dotsarr}....{LaTeX}
\newarrow{Dashto}{}{dash}{}{dash}{vee}
\newarrow{Dasharr}{}{dash}{}{dash}{LaTeX}
\newarrow{Mapsto}{mapsto}---{vee}
\newarrow{Mapsarr}{mapsto}---{LaTeX}
\newarrow{IntoA}{hooka}---{vee}
\newarrow{IntoB}{hookb}---{vee}
\newarrow{Embed}{vee}---{vee}
\newarrow{Emarr}{LaTeX}---{LaTeX}
\newarrow{Onto}----{doublevee}
\newarrow{Dotsonarr}....{doubleLaTeX}
\newarrow{Dotsonto}....{doublevee}
\newarrow{Dotsonarr}....{doubleLaTeX}
\else
\newarrow{To}---->
\newarrow{Arr}---->
\newarrow{Dotsto}....>
\newarrow{Dotsarr}....>
\newarrow{Dashto}{}{dash}{}{dash}>
\newarrow{Dasharr}{}{dash}{}{dash}>
\newarrow{Mapsto}{mapsto}--->
\newarrow{Mapsarr}{mapsto}--->
\newarrow{IntoA}{hooka}--->
\newarrow{IntoB}{hookb}--->
\newarrow{Embed}>--->
\newarrow{Emarr}>--->
\newarrow{Onto}----{>>}
\newarrow{Dotsonarr}....{>>}
\newarrow{Dotsonto}....{>>}
\newarrow{Dotsonarr}....{>>}
\fi

\newarrow{Implies}===={=>}
\newarrow{Project}----{triangle}
\newarrow{Pto}----{harpoon}
\newarrow{Relto}{harpoon}---{harpoon}

\newarrow{Eq}=====
\newarrow{Line}-----
\newarrow{Dots}.....
\newarrow{Dashes}{}{dash}{}{dash}{}

\newarrow{SquareInto}{sqhook}--->

\newarrowhead{cmexbra}{\@cmex7B}{\@cmex7C}{\@cmex3B}{\@cmex38}
\newarrowtail{cmexbra}{\@cmex7A}{\@cmex7D}{\@cmex39}{\@cmex3A}
\newarrowmiddle{cmexbra}{\braceru\bracelu}{\bracerd\braceld}{\vcenter{%
\hbox@maths{\@cmex3D\mkern-2mu}}}
{\vcenter{\hbox@maths{\mkern2mu\@cmex3C}}}
\newarrow{@brace}{cmexbra}-{cmexbra}-{cmexbra}
\newarrow{@parenth}{cmexbra}---{cmexbra}
\def\rightBrace{\d@brace[thick,cmex]}
\def\leftBrace{\u@brace[thick,cmex]}
\def\upperBrace{\r@brace[thick,cmex]}
\def\lowerBrace{\l@brace[thick,cmex]}
\def\rightParenth{\d@parenth[thick,cmex]}
\def\leftParenth{\u@parenth[thick,cmex]}
\def\upperParenth{\r@parenth[thick,cmex]}
\def\lowerParenth{\l@parenth[thick,cmex]}

\let\hEq\rEq
\let\vEq\uEq

\def\labelstyle{
\ifincommdiag
\textstyle
\else
\scriptstyle
\fi}
\let\objectstyle\displaystyle

\newdiagramgrid{pentagon}{0.618034,0.618034,1,1,1,1,0.618034,0.618034}{1.%
17557,1.17557,1.902113,1.902113}

\newdiagramgrid{perspective}{0.75,0.75,1.1,1.1,0.9,0.9,0.95,0.95,0.75,0.75}{0%
.75,0.75,1.1,1.1,0.9,0.9}

\diagramstyle[
dpi=300,
vmiddle,nobalance,
loose,
thin,
pilespacing=10pt,%
shortfall=4pt,
]

\ifx\CD@qK\else\CD@PK\relax\CD@FF\CD@e\fi\fi

\CD@vE\CD@hK\else\fi\else\fi\cdrestoreat
\usepackage{amssymb}
\usepackage{mathrsfs}
\usepackage[all,arc]{xy}
\usepackage{cite}
\usepackage{tikz-cd}
\usepackage{MnSymbol}

\newcommand{\Per}{\operatorname{Per}}

\newcommand{\pa}{\partial}

\newcommand{\WW}{{\mathcal W}}

\newcommand{\ssl}{{\frak sl}}

\newcommand{\ch}{\operatorname{ch}}
\newcommand{\CH}{\operatorname{CH}}
\renewcommand{\mod}{\operatorname{mod}}

\newcommand{\Cone}{\operatorname{Cone}}

\newcommand{\und}{\underline}

\newcommand{\OO}{{\mathcal O}}

\newcommand{\coker}{\operatorname{coker}}
\newcommand{\DD}{{\mathcal D}}
\newcommand{\NN}{{\mathcal N}}

\newcommand{\hra}{\hookrightarrow}
\newcommand{\lan}{\langle}
\newcommand{\ran}{\rangle}
\newcommand{\Coh}{\operatorname{Coh}}
\newcommand{\GG}{{\mathcal G}}

\newcommand{\Spec}{\operatorname{Spec}}
\newcommand{\Sing}{\operatorname{Sing}}
\newcommand{\Td}{\operatorname{Td}}

\newcommand{\Th}{\Theta}
\newcommand{\si}{\sigma}

\newcommand{\Pic}{\operatorname{Pic}}

\newcommand{\de}{\delta}

\renewcommand{\ker}{\operatorname{ker}}

\numberwithin{equation}{subsection}

\newtheorem{thm}{Theorem}[subsection]
\newtheorem{prop}[thm]{Proposition}

\newtheorem{lem}[thm]{Lemma}
\newtheorem{cor}[thm]{Corollary}
{  \theoremstyle{definition}

\newtheorem{exs}[thm]{Examples}
\newtheorem{rem}[thm]{Remark}
\newtheorem{rems}[thm]{Remarks}
}

\newcommand{\Pf}{\noindent {\it Proof}}
\newcommand{\id}{\operatorname{id}}

\newcommand{\LLie}{{\cal L}ie}
\newcommand{\ov}{\overline}
\newcommand{\we}{\wedge}

\renewcommand{\AA}{{\mathcal A}}

\newcommand{\FF}{{\mathcal F}}

\newcommand{\JJ}{{\mathcal J}}

\newcommand{\Om}{\Omega}

\newcommand{\Hom}{\operatorname{Hom}}

\newcommand{\Ext}{\operatorname{Ext}}

\renewcommand{\a}{\alpha}
\renewcommand{\b}{\beta}
\newcommand{\om}{\omega}

\newcommand{\la}{\lambda}
\renewcommand{\th}{\theta}
\newcommand{\C}{{\Bbb C}}

\newcommand{\Z}{{\Bbb Z}}
\newcommand{\Q}{{\Bbb Q}}

\newcommand{\wt}{\widetilde}
\newcommand{\ot}{\otimes}

\newcommand{\sub}{\subset}
\newcommand{\ed}{\qed\vspace{3mm}}

\newcommand{\Qcoh}{\operatorname{Qcoh}}

\newcommand{\NS}{\operatorname{NS}}

\newcommand{\cl}{\operatorname{cl}}

\title{NC-smooth thickenings and Jacobians}
\author{Alexander Polishchuk}
\address{
    Department of Mathematics, 
    University of Oregon, 
    Eugene, OR 97403, USA; and National Research University Higher School of Economics, Moscow, Russia
  }
 \email{apolish@uoregon.edu}

\thanks{Supported in part by NSF grant DMS-2349388, by the Simons Travel grant MPS-TSM-00002745,
and within the framework of the HSE University Basic Research Program (HSE-BR-2025-061).}

\begin{document}

\begin{abstract}
We prove that a bounded complex of coherent sheaves on an abelian variety $A$, whose Fourier-Mukai transform has support of dimension $\le 1$, 
extends to a perfect complex on the standard NC-smooth thickening of $A$ and on the deformation quantization of any constant Poisson bracket on $A$. 
We discuss a related conjectural characterization of Jacobians in
terms of NC-smooth thickenings.
\end{abstract}

\maketitle

\centerline{\sc Introduction}

NC-smooth thickenings of smooth algebraic varieties were introduced by Kapranov \cite{Kapranov} as a way to study noncommutative geometry
in the ``formal neighborhood" of a commutative variety. The model example of an NC-smooth thickening is a completion of a free algebra in $n$ generators
with respect to the commutator filtration (see Sec.\ \ref{NC-thick-sec} below). 
Kapranov's idea is to consider sheaves of algebras $\OO_X^{NC}$ over smooth varieties that look like this model example in a formal neighborhood of each point (such that by taking the quotient of $\OO_X^{NC}$ by the ideal generated by commutators one recovers the structure sheaf $\OO_X$).
Kapranov proved existence and uniqueness of such NC-smooth thickenings for smooth affine varieties, and constructed some non-affine examples.
Non-affine examples were further studied in \cite{P-Tu} and \cite{DP}. 

Every abelian variety $A$ possesses a natural translation-invariant NC-smooth thickening $A^{NC}$ that we will call {\it standard}. 
In holomorphic terms, this can be obtained by the descent of the standard
NC-smooth thickening of the affine space with respect to translations by a lattice. Over any field of characteristic zero, there exists an algebraic construction using
a canonical flat connection on $A$ (see \cite{P-Tu}). In the case when $A=J$, the Jacobian of a smooth projective curve $C$, there is a way to think of the standard NC-smooth thickening of $J$ as representing a natural noncommutative moduli functor (see \cite[Sec.\ 4]{P-Tu}), and this gives rise to an NC-smooth analog of the Fourier-Mukai transform
that produces sheaves on $J^{NC}$ from sheaves on $C$ (but it is far from being an equivalence).

In this paper we show that the theta divisor $\Th$ on $J$ (defined up to a translation) also extends in some sense to the standard NC-smooth thickening $J^{NC}$. Namely, we prove that there exists a coherent sheaf supported on $\Th$, generically of rank $1$ on $\Th$, which extends to an object of the derived category of $J^{NC}$ represented by a bounded complex of vector bundles 
(see Theorem \ref{Jac-thm}). We conjecture that this property characterizes Jacobians among principally polarized abelian varieties with irreducible theta divisors.


\medskip

\noindent
{\bf Conjecture A}.
{\it Let $(A, \Theta)$ be a principally polarized abelian variety with its theta divisor, where
$\Theta$ is irreducible.
Then it is the Jacobian of a smooth curve if and only if there exists a coherent sheaf $\FF$ on $A$, supported on $\Theta$ and of length $1$ at the general point of $\Theta$,
such that $\FF$ extends to an object of $D^b(A^{NC})$ represented by a bounded complex of vector bundles
(where $A^{NC}$ is the standard  NC-smooth thickening of $A$).
}

\medskip

The ``only if" part is Theorem \ref{Jac-thm}. As evidence for the ``if" part, we prove that existence of a coherent sheaf $\FF$ as in Conjecture A implies that
$(A,\Theta)$ belongs to the Andreotti-Mayer locus, i.e., that the dimension of the singular locus of $\Theta$ is $\ge g-4$, where $g=\dim A$ (see Cor.\ \ref{ch-2-cor}).
In fact, to deduce this we only need an extension to the standard 1-smooth thickening of $A$ (see Sec.\ \ref{NC-thick-sec}).

As a possible approach to Conjecture A, we suggest the following conjecture giving a criterion of extendability to the standard NC-smooth thickening in terms of
the Fourier-Mukai transform $F:D^b(A)\to D^b(\hat{A})$ (see \cite{Mukai}). 

\medskip

\noindent
{\bf Conjecture B}.
{\it Let $A$ be an abelian variety, $A^{NC}$ its standard NC-smooth thickening. An object $E\in D^b(A)$ extends to a bounded complex of vector bundles on $A^{NC}$
if and only if the Fourier-Mukai transform of $E$ is supported in dimension $\le 1$.
}

\medskip

We prove the ``if" part in Theorem \ref{NC-FM-thm}. We also show that Conjecture B implies Conjecture A (see Corollary \ref{conjB-conjA-cor}).
Note that existence of a resolution by locally free $\OO_{A^{NC}}$-modules of finite rank is much more restrictive than in the commutative case (see e.g., Cor.\ \ref{codim-cor}).

Instead of asking which sheaves on a variety $X$ extend to an NC-smooth thickening, one can ask a similar question for deformation quantizations of Poisson structures on $X$.
In the case of constant Poisson structures on abelian varieties and the corresponding Moyal deformation quantizations, we show that there is a relation between these questions.
Namely, we construct a homomorphism from the structure sheaf of the standard NC-smooth thickening $\OO_A^{NC}$ to the deformation quantization algebra associated with any constant Poisson bivector (see Proposition \ref{NC-to-quant-prop}). 
Using the result of \cite{BGP} we prove that existence of an extension of an appropriate sheaf on the theta-divisor to the first-order deformation
quantization associated with sufficiently generic constant Poisson bivectors implies that the dimension of the singular locus of $\Theta$ is $\ge \dim A-4$ (see Prop.\ \ref{1st-order-def-qu-prop}).

On the other hand, we show that the ``only if" part of Conjecture B does not hold if we replace the NC-smooth thickening by the corresponding first-order thickening.
Namely, using the results of Toda \cite{Toda} we give an example of a coherent sheaf on an abelian surface which extends to the standard 1-smooth thickening but 
whose Fourier-Mukai transform has full support (see Sec.\ \ref{surface-ex-sec}). 



\bigskip

\noindent
{\it Conventions and notations}. We work over $\C$.
We denote by $\ch^H$ and $c_1^H$ the characteristic classes with values in $H^*(X,\Q)$ (sometimes we also view $c_1^H$ as an element
of $H^1(X,\Om^1_X)$); and by
$\ch$ and $c_1$ the characteristic classes with values in the Chow group. For a smooth variety $X$ we denote by $D^b(X)$ the bounded derived category of coherent
sheaves on $X$.

\section{NC-smooth thickenings and NC-Fourier transform}

\subsection{NC-smooth thickenings}\label{NC-thick-sec}

We refer to \cite{Kapranov} and \cite[Sec.\ 2.1]{P-Tu} for basics on NC-smooth thickenings.

The commutator filtration $F^dR$ on a ring $R$ is defined by letting $F^dR$ be a two-sided ideal generated by all expressions containing $d$ commutators
(possibly nested). The ring is called NC-nilpotent if $F^dR=0$ for some $d$, and NC-complete if the natural map $R\to\varprojlim R/F^dR$ is an isomorphism.
NC-smoothness for NC-complete algebras is defined similarly to formal smoothness, in terms of nilpotent extensions of NC-nilpotent algebras.

These notions can be globalized, so one can consider sheaves of NC-smooth algebras. For a smooth variety $X$, an NC-smooth thickening of $X$ is defined
as a sheaf of NC-smooth algebras $\OO_X^{NC}$ together with an isomorphism $\OO_X^{NC}/F^1\simeq \OO_X$. 

One can also consider $n$-smooth thickenings by restricting to NC-algebras (or sheaves) with $F^{n+1}=0$. For example, a $1$-smooth thickening of a smooth variety
$X$ is a central square zero extension of sheaves of algebras 
$$0\to \NN\to \wt{\OO}_X\to \OO_X\to 0$$
such that the commutator pairing $f,g\mapsto [\wt{f},\wt{g}]$, where $\wt{f}$ and $\wt{g}$ are liftings of $f,g\in \OO_X$ to $\wt{\OO}_X$, induces an isomorphism
$\Om_X^2\rTo{\sim} \NN$. In particular, for every $X$ there is a {\it standard} 1-smooth thickening, 
$$\OO_X^{1-NC}=\Om_X^2\oplus \OO_X,$$
equipped with the product $(\a,f)\cdot (\b,g)=(g\a+f\b+df\we dg,fg)$
(see \cite[(1.3.9)]{Kapranov}).

In \cite{P-Tu} we developed a construction of an NC-smooth thickening associated with a torsion free connection on the tangent bundle $T_X$ (in the case of a flat connection it is recalled in Sec.\ \ref{dg-constr-sec} below).
Applying this construction to the natural flat connection on an abelian variety $A$ we get the {\it standard} NC-smooth thickening $A^{NC}$
(more generally, this works for commutative algebraic groups, see \cite[Def.\ 2.3.19]{P-Tu}). The induced 1-smooth thickening $\OO_A^{NC}/F^2$ is the standard
1-smooth thickening of $A$ discussed above (by \cite[Prop.\ 2.3.18]{P-Tu}).

\subsection{NC-Fourier transform}

Let $A$ be an abelian variety, $\hat{A}$ the dual abelian variety.
It was observed in \cite{P-Tu} that the restriction of Fourier-Mukai transform $F:D^b(\hat{A})\to D^b(A)$ to $D^b(C)$, for a curve $C\sub \hat{A}$,
extends to an NC-Fourier transform $F^{NC}:D^b(C)\to \Per(A^{NC})$, where $A^{NC}$ is the standard NC-smooth thickening of $A$, and $\Per(A^{NC})$
denotes the full subcategory in the derived category of $\OO_{A^{NC}}$-modules formed by bounded complexes of vector bundles (i.e., locally free left modules of finite rank).

In the next theorem we use a slight extension of this construction to prove the ``if" part of Conjecture B.





\begin{thm}\label{NC-FM-thm} Let $A$ be an abelian variety, $A^{NC}$ its standard thickening.
If the Fourier-Mukai transform of $E\in D^b(A)$ has support of dimension $\le 1$ then
$E$ extends to a perfect object in $D^b(A^{NC})$. Furthermore, for any closed subset $Z\sub \hat{A}$ of dimension $\le 1$, 
there exists a functor 
\begin{equation}\label{NC-Fourier-Z-eq}
D^b_Z(\hat{A})\to \Per(A^{NC}),
\end{equation}
where $D^b_Z(\hat{A})\sub D^b(\hat{A})$ is the full subcategory of objects with support on $Z$, 
extending the Fourier-Mukai transform $D^b_Z(\hat{A})\to D^b(A)$.
\end{thm}

\Pf . 
The first assertion follows from the second, by applying the functor \eqref{NC-Fourier-Z-eq}, for appropriate $Z$, to the inverse Fourier-Mukai transform of $E$.

Let $Z\subset \hat{A}$ be a closed subset 
of dimension $1$. We can find two affine open subsets $U_1,U_2\sub\hat{A}$
such that $Z\sub U:=U_1\cup U_2$.
We have $D^b_Z(\hat{A})=D^b_Z(U)$, so it is enough
to construct an integral transform $D^b_Z(U)\to \Per(A^{NC})$ 
extending the Fourier transform $D^b_Z(U)\to D^b(A)$.
As in \cite[Sec.\ 4.3]{P-Tu}, for this we construct an extension of the Poincare line bundle to $U\times A^{NC}$.

Let $A^\natural\to \hat{A}$ denote the $H^0(A,\Om^1)$-torsor parametrizing line bundles with connections on $A$.
Since $U_i$ are affine, we can choose liftings of $U_1$ and $U_2$ to $A^\natural$. 
As is explained in \cite[Sec.\ 4.3]{P-Tu}, using these liftings we get extensions of the Poincar\'e bundle over $U_i\times A$ to $U_i\times A^{NC}$ and
an isomorphism between the induced extensions on $(U_1\cap U_2)\times A^{NC}$.
Hence, we obtain a line bundle on $U\times A^{NC}$ that we can use as a kernel of an integral
transform 
$$F_U:D^b(\Qcoh(U))\to D(\OO^{NC}_A-\mod)$$ 
which lifts the corresponding integral transform from $D^b(\Qcoh(U))$ to $D(\Qcoh A)$.
It remains to prove that $F_U(\FF)$ is a perfect complex for any $\FF$ supported on $Z$.
It is enough to prove this for skyscraper sheaves and for coherent sheaves obtained as push-forwards
from normalizations of irreducible components of $Z$. But this follows from \cite[Thm.\ 4.3.6]{P-Tu}.
\ed



One piece of indirect evidence in favor of Conjecture B is that both conditions imply that $\ch^H_{\ge 2}(E)=0$: 
for the condition that $F(E)$ is supported in dimension $\le 1$ this follows from the formula for the action of the Fourier-Mukai transform on cohomology
(see \cite[Cor.\ 1.18]{Mukai-2}). For the condition that $E$ extends to a perfect complex on $A^{NC}$ this is a consequence of the following result.

\begin{prop}\label{ch-prop}
Let $X$ be a smooth projective variety.
If an object $E\in D^b(X)$ extends to a bounded complex of vector bundles on the standard 1-smooth thickening of $X$ then $\ch^H_{\ge 2}(E)=0$.
\end{prop}

\begin{proof} This follows immediately from \cite[Cor.\ 3.3.5]{P-Tu}.
\end{proof}

\begin{cor}\label{codim-cor} 
Let $\FF$ be a nonzero coherent sheaf on a smooth connected projective variety $X$.

\noindent 
(i) If $\FF$ has support of codimension $\ge 2$ then $\FF$ does not extend to a perfect complex
on the standard 1-smooth thickening of $X$. 

\noindent
(ii) Assume $\FF=i_*\GG$, where $i:D\to X$ is the embedding of a smooth connected divisor, and $\GG$ is a coherent
sheaf of rank $r$ on $D$. If $\FF$ extends to a perfect complex on the standard 1-smooth thickening of $X$ then one has
$$i_*(c^H_1(\GG)-r\cdot c_1^H(\OO(D)|_D)/2)=0$$
in $H^4(X,\Q)$.
\end{cor}

\begin{proof}
(i) Let $n=\dim X$, and let $m$ be the dimension of the support of $\FF$. Then $\ch(\FF)_{n-m}=\sum_i a_i [Z_i]\in \CH^{n-m}(X)_{\Q}$ for some 
irreducible subvarieties $Z_i$ of dimension $m$ and some positive coefficients $a_i$, and $\ch(\FF)_{<n-m}=0$. It follows that if $h$ is an ample divisor class
then $\ch^H(\FF)_{n-m}\cdot h^m\neq 0$, so $\ch^H(\FF)_{n-m}\neq 0$. Thus, if $m\le n-2$ we cannot have $\ch^H_{\ge 2}(\FF)=0$. So, by Proposition
\ref{ch-prop}, $\FF$ cannot extend to the standard 1-smooth thickening of $X$.

\noindent
(ii) This follows immediately by applying the Grothendieck-Riemann-Roch theorem
to rewrite the condition $\ch^H_2(i_*\GG)=0$.
\end{proof}

\begin{rems} 1. Note that both conditions in Conjecture B are preserved by pull-backs under homomorphisms between abelian varieties (as well as by translations) and by tensoring with line bundles in $\Pic^0(A)$. This follows from the standard
compatibility of the Fourier-Mukai transform with these operations and from the fact that they extend to the standard NC-smooth thickenings
(see \cite[Prop.\ 2.3.20]{P-Tu} and  \cite[Thm.\ 3.3.8(ii)]{P-Tu}).


\noindent
2. The results of \cite[Sec.\ 3.2]{P-Tu} suggest that a plausible analog of Conjecture B for $\OO_A^{NC}$-bimodules is that an object $E\in D^b(A)$ extends to a complex of $\OO_A^{NC}$-bimodules, 
locally free of finite rank as left or right modules, if and only if the Fourier-Mukai transform of $E$ has $0$-dimensional support.
\end{rems}

\subsection{Extensions of the theta-divisor to the NC-smooth thickening}


Using Theorem \ref{NC-FM-thm} we get a proof of the ``only if" direction of Conjecture A.

\begin{thm}\label{Jac-thm}
Let $(J,\Th)$ be the Jacobian of a smooth projective curve $C$, with its theta divisor. Then there exists a line bundle $L$ on the smooth part $\Th^{ns}$ of $\Th$,
such that the coherent sheaf $\FF:=j^{ns}_*L$ supported on $\Th$, where $j^{ns}:\Th^{ns}\to J$ is the natural embedding, 
extends to a perfect complex on $J^{NC}$, the standard NC-smooth thickening of $J$.
\end{thm}

\begin{proof}
Let $i:C\to J$ be the embedding associated with some point, and let $L$ be a line bundle of degree $g-1$
on $C$. Then the Fourier-Mukai transform of $i_*L$ is of the form $\FF[-1]$, where $\FF$ is of the described form (see \cite[Prop.\ 2.3]{BP}). It remains
to apply Theorem \ref{NC-FM-thm}.
\end{proof}

\begin{rem}
Note that if a coherent sheaf $\FF$ on $X$ extends to a bounded complex $P^\bullet$ of locally free modules of finite rank over some NC-nilpotent thickening $A\to \OO_X$
then we can assume this complex to be concentrated in non-positive degrees (i.e., $P^{>0}=0$). Indeed, if $P^{>n+1}=0$ for $n\ge 0$, then 
$\coker(d_n:P^n\to P^{n+1})$ becomes zero when tensored with $\OO_X$. Hence, $\coker(d_n)=0$, so we can replace $P^n$ with $\ker(d_n)$ and $P^{n+1}$ with $0$.
Iterating this we will get a complex in non-positive degrees.
\end{rem}

Recall (see \cite{Mukai}) that a coherent sheaf $\FF$ on an abelian variety $A$ is called an IT-sheaf (resp., WIT-sheaf) of index $i$ if $H^j(A,\FF\ot P)=0$ for $j\neq i$ for
every $P\in \Pic^0(A)$ (resp., $\und{H}^j F(\FF)=0$ for $j\neq i$, where $F:D^b(A)\to D^b(\hat{A})$ is the Fourier-Mukai transform).

The following characterization of Jacobians is due to \cite{MMG} (based on Matsusaka-Ran criterion).

\begin{thm}\label{MMG-thm} 
(\cite[Thm.\ 4.3]{MMG}) Let $(A,\Theta)$ be an indecomposable principally polarized abelian variety.
Suppose there exists a WIT-sheaf $\GG$ of index $g$ on $A$ such that $c_i^H(\GG)=(-1)^i\th^i/i!$, where $\th=c_1^H(\OO(\Theta))$. Then $(A,\Theta)$ is a Jacobian of
a smooth curve.
\end{thm}

Here is a version of this result that is better adapted to our context. 

\begin{prop}\label{Jac-prop} 
Let $(A,\Theta)$ be a principally polarized abelian variety, such that $\Theta$ is irreducible.
Assume that there exists a coherent sheaf $\FF$ on $A$, supported sheaf theoretically on $\Theta$ and of length $1$
at the general point of $\Theta$, 
such that
the Fourier transform of $\FF$ is supported in dimension $\le 1$. Then $(A,\Theta)$ is a Jacobian.
\end{prop}

\Pf . We will reduce this to Theorem \ref{MMG-thm}.
By assumption, $\FF=F(E)$, where $E$ is an object of derived category of $\hat{A}\simeq A$ with the
support $Z\sub A$ of dimension $\le 1$. Pick an ample divisor $D$ that intersects each component
of $Z$ transversely, so that all the maps $\und{H}^iE\to \und{H}^iE(D)$ are injective.
Consider the exact triangle
$$E\to E(D)\to E'\to E[1].$$
Applying the Fourier transform and passing to cohomology we get an exact sequence
$$0\to \und{H}^{-1}F(E(D))\to \und{H}^{-1}F(E')\to \FF\to \und{H}^0F(E(D))\to \und{H}^0F(E')\to 0$$
and isomorphisms
$$\und{H}^iF(E(D))\rTo{\sim} \und{H}^iF(E') \ \text{ for } i\neq 0,-1.$$
If $D$ is sufficiently ample then the sheaves $\und{H}^iE(D)$ and $\und{H}^iE'$ are IT-sheaves of index $0$. Hence,
we have 
$$\und{H}^iF(E(D))\simeq F(\und{H}^iE(D)), \ \ \und{H}^iF(E')\simeq F(\und{H}^iE').$$
Thus, $\und{H}^0F(E(D))$ is a vector bundle and the map $\FF\to \und{H}^0F(E(D))$ from the torsion sheaf $\FF$
is zero. Thus, from the above exact sequence we get a short exact sequence
$$0\to F(\und{H}^{-1}E(D))\to F(\und{H}^{-1}E')\to \FF\to 0.$$
On the other hand, since the map $\und{H}^0E\to \und{H}^0E(D)$ is injective, we obtain an exact sequence
$$0\to \und{H}^{-1}E\to \und{H}^{-1}E(D)\to \und{H}^{-1}E'\to 0,$$
so the sheaf $\und{H}^{-1}E'$ has zero-dimensional support.
This implies that $c_1(F(\und{H}^{-1}E'))=0$. Hence, the sheaf 
$$\GG:=F(\und{H}^{-1}E(D))$$ satisfies
$$c_1^H(\GG)=-c_1^H(\FF)=-\th.$$ 
Also, since $\GG$ is the Fourier transform of a sheaf supported in
dimension $1$, we have $\ch_{\ge 2}^H(\GG)=0$ (as a consequence of Mukai's formula \cite[Cor.\ 1.18]{Mukai-2}),
which is equivalent to $c_i^H(\GG)=c_1^H(\GG)^i/i!$.
Finally, since $\und{H}^{-1}E(D)$ is an IT-sheaf of index $0$, $\GG$ is a WIT-sheaf of index $g$. Thus, $\GG$ satisfies the
conditions of Theorem \ref{MMG-thm}. 
\ed


\begin{cor}\label{conjB-conjA-cor}
Conjecture B implies Conjecture A.
\end{cor}

\begin{proof}
We only need to check the ``if" part of Conjecture A. If there exists a coherent sheaf $\FF$ on $A$ with the stated properties
then by Conjecture B its Fourier transform has support of dimension $\le 1$. Hence, by Proposition \ref{Jac-prop}, $(A,\Th)$ is a Jacobian.
\end{proof}


Recall that the Jacobians of curves belong to the Andreotti-Mayer locus of principally polarized abelian varieties 
$(A,\Th)$ for which the singular locus of $\Th$ has dimension $\ge g-4$.

\begin{prop}\label{ch-2-prop} Assume $(A,\Theta)$ is a principally polarized abelian variety of dimension $g\ge 4$, with $\Theta$ irreducible.
Suppose there exists a coherent sheaf $\FF$ on $A$, sheaf theoretically supported on $\Theta$, of length $1$ at the general point of $\Theta$, 
and such that $\ch_2^H(\FF)=0$. 
Then $\dim\Sing(\Th)\ge g-4$.
\end{prop}

\Pf . Set $A':=A\setminus\Sing(\Th)$, $\Th':=\Th\setminus\Sing(\Th)$,
and let $N=\OO_{\Th}(\Th)$ be the normal line bundle of $\Th$ in $A$.
Consider the closed embedding $i:\Th'\to A'$.
By assumption, $\FF|_{A'}=i_*F$, where $F$ is a coherent sheaf on $\Th'$, generically of rank $1$.
By the Grothendieck-Riemann-Roch theorem applied to the embedding $i$, we have
$$\ch_2(\FF|_{A'})=i_*[\ch(F)\Td^{-1}(N|_{\Th'})]_2=i_*(c_1(F)-\frac{c_1(N|_{\Th'})}{2})$$
in $\CH_{g-2}(A')_{\Q}$. Assume that $\dim\Sing(\Th)<g-4$. Then the natural maps
\begin{equation}\label{Pic-restr-maps}
\Pic(A)\to \Pic(\Th)\to \CH_{g-2}(\Th)\to\CH_{g-2}(\Th')
\end{equation}
are isomorphisms (see \cite[Sec.\ 3.1]{BP}), so there exists a class $\ell\in\Pic(A)$ such that $c_1(F)$ is the
restriction of $\ell$ to $\Th'$. Thus, we get an equality in $\CH_{g-2}(A')_{\Q}$,
$$\ch_2(\FF|_{A'})=c_1(\Th)\cdot (\ell-\frac{c_1(\Th)}{2})|_{A'}.$$
Since the restriction map $\CH_{g-2}(A)\to\CH_{g-2}(A')$ is an isomorphism, this implies
that 
$$\ch_2(\FF)=c_1(\Th)\cdot (\ell-\frac{c_1(\Th)}{2})$$
in $\CH_{g-2}(A)_{\Q}$ and hence, we have the same equality in $H^4(A,\C)$. 
Since $\ch^H_2(\FF)=0$ in $H^4(A,\C)$ by assumption, we get
$$c_1^H(\Th)\cup (\cl(\ell)-\frac{c_1^H(\Th)}{2})=0,$$
where $\cl:\CH_{g-1}(A)\to H^2(A,\Z)$ is the cycle class map. But the map
$$H^2(A,\C)\rTo{\cup c_1^H(\Th)} H^4(A,\C)$$
is injective being part of the Lefschetz $\ssl_2$-action (since $2<g$). Hence, 
$$\cl(\ell)-\frac{c_1^H(\Th)}{2}=0$$
in $H^2(A,\C)$ which contradicts the fact that $c_1^H(\Th)$ is a primitive vector in $H^2(A,\Z)$.
\ed

\begin{cor}\label{ch-2-cor}
Assume $(A,\Theta)$ is a principally polarized abelian variety of dimension $g\ge 4$, with $\Theta$ irreducible.
Suppose there exists a coherent sheaf $\FF$ on $A$, sheaf theoretically supported on $\Theta$, of length $1$ at the general point of $\Theta$, 
such that $\FF$, viewed as an object in $D^b(A)$ extends to a complex of vector bundles on the standard $1$-smooth thickening of $A$.
Then $\dim\Sing(\Th)\ge g-4$.
\end{cor}

\begin{proof}
Indeed, by Proposition \ref{ch-prop}, we have $\ch_{\ge 2}^H(\FF)=0$. It remains to apply Proposition \ref{ch-2-prop}.
\end{proof}

\section{Relation to deformation quantization}\label{quant-sec}

\subsection{DG-constructions}\label{dg-constr-sec}

\subsubsection{NC-smooth thickenings}

Here we recall the construction of an NC-smooth thickening associated with a torsion free flat connection $\nabla$ on $X$
(see \cite{P-Tu}) and a parallel construction of deformation quantizations (due to Fedosov \cite{Fedosov}).

Let $\hat{T}_{\OO_X}(\Om^1_X)$ denote the completion of the tensor algebra of $\Om^1_X$.
One can extend the connection $\nabla$ on $T_X$ to a degree $1$ derivation $D_1$ on 
$$\AA_X:=\Om^{\bullet}_X\ot_{\OO_X} \hat{T}_{\OO_X}(\Om^1_X),$$
extending the de Rham differential on $\Om^\bullet_X$ and such that for $\a\in \Om^1_X$ one has
$$D_1(1\ot \a)=\nabla(\a).$$
On the other hand, one has an $\Om^\bullet_X$-linear derivation $D_0$ of $\AA_X$ such that $D_0(1\ot \a)=\a\ot 1$.

In \cite{P-Tu} we studied the dg-algebra $(\AA_X,D)$, where $D=D_0+D_1$, and showed that $\und{H}^0(\AA_X,D)$ gives an NC-smooth thickening.
We also considered a more general version of this construction in \cite{P-Tu},
starting with a  twisted version of $\AA_X$ and a not necessarily flat connection $\nabla$, but we will not use this more general construction here.


\subsubsection{Deformation quantizations}

Recall that an (algebraic) deformation quantization of $X$ is a sheaf of associative
$\C[\![h]\!]$-algebras $\DD$ on $X$, complete with respect to the $h$-adic topology and flat over $\C[\![h]\!]$, equipped with
an isomorphism $\rho:\DD/h\DD\rTo{\sim}\OO_X$.
In particular, the commutator in $\DD$ induces a Poisson bracket on $\OO_X$ such that
$ab-ba=h\{\rho(a),\rho(b)\} \mod h^2\DD$, and one says that $\DD$ quantizes this Poisson
bracket. Recall that the corresponding Poisson bivector $\Pi\in H^0(X,\bigwedge^2 T_X)$ is characterized by the condition
$\lan \Pi, df\wedge dg\ran=\{f,g\}$, where $f$ and $g$ are local functions on $X$.

Now let us start with a Poisson bivector $\Pi$ on $X$.
Assume that we have a torsion free flat connection $\nabla$ on $X$ such that the Poisson bivector 
$\Pi\in H^0(X,\bigwedge^2 T_X)$ is $\nabla$-horizontal. This means
that for any vector field $v$ and any pair of functions $f,g$ we have
$$v\cdot \{f,g\}=\langle \Pi, \nabla_v(df)\wedge dg+df\wedge\nabla_v(dg)\rangle.$$

The sheaf of (possibly degenerate) formal Weyl algebras associated with $\Pi$ is the quotient
$$\WW_{\Pi}=\hat{T}(\Om^1)[\![h]\!]/\JJ_{\Pi},$$
where the two-sided ideal $\JJ_{\Pi}$ is generated by
$$\om_1\ot\om_2-\om_2\ot\om_1- h\lan \Pi, \om_1\wedge \om_2\ran.$$
By the PBW-theorem for the sheaf of Heisenberg Lie algebras associated with $\Pi$, we have a direct sum decomposition
\begin{equation}\label{Poisson-decomposition-eq}
\hat{T}(\Om^1)[\![h]\!]=\JJ_{\Pi}\oplus \hat{S}(\Om^1)[\![h]\!]
\end{equation}
where $\hat{S}(\Om^1)$ consists of symmetric tensors.
Thus, we have a canonical isomorphism of $\OO_X[\![h]\!]$-modules
\begin{equation}\label{WP-Sym-eq}
\hat{S}(\Om_X^1)[\![h]\!]\rTo{\sim}\WW_{\Pi},
\end{equation}
whose reduction modulo $h$ is an isomorphism of algebras.

\begin{prop}\label{Moyal-prop}
(i) The decomposition of $\Om^\bullet\ot_\OO\hat{T}(\Om^1)[\![h]\!]$
induced by \eqref{Poisson-decomposition-eq}
is compatible with the derivations $D_0$ and $D_1$ (extended $h$-linearly).

\noindent
(ii) Let $\ov{D}_0$ and $\ov{D}_1$ denote the induced derivations of $\Om^\bullet\ot\WW_{\Pi}$.
Then 
$$\DD_{\Pi,\nabla}:=\und{H}^0(\Om^\bullet\ot\WW_{\Pi},\ov{D}_0+\ov{D}_1)$$ 
is a deformation quantization of $\Pi$.

\noindent
(iii) If $(x_1,\ldots,x_n)$ is a formal coordinate system at a point $p\in X$, such that $dx_i$ are $\nabla$-horizontal, then 
we have an identification of the completion of $\DD_{\Pi,\nabla}$ at $p$ with $\C[\![x_1,\ldots,x_n,h]\!]$ equipped with the Moyal product
\begin{equation}\label{Moyal-eq}
f*g=\sum_{k=0}^\infty \sum_{i_1,\ldots,i_k;j_1,\ldots,j_k}\frac{h^k}{k!2^k}\Pi_{i_1j_1}\ldots\Pi_{i_kj_k}(\pa_{i_1}\ldots\pa_{i_k}f)(\pa_{j_1}\ldots\pa_{j_k}g),
\end{equation}
where $\Pi=\sum_{i<j}\Pi_{ij}\pa_i\we\pa_j$.
\end{prop}

\begin{proof}
(i) We have $D_0(\om_1\ot\om_2-\om_2\ot\om_1)=0$, which easily implies that $D_0(\JJ_{\Pi})\sub \Om^1\ot \JJ_{\Pi}$. On the other hand, for any $\a\in \Om^1$
one has $D_0(\a^{\ot n})\sub \Om^1\ot \hat{S}(\Om^1)$, hence $D_0(\hat{S}(\Om^1))\sub \Om^1\ot \hat{S}(\Om^1)$. 

To see invariance with respect to $D_1$,
let us choose a $\nabla$-horizontal basis $\a_1,\ldots,\a_n$ of $\Om_X^1$ in the formal neighborhood of some point. Then $\JJ_{\Pi}$ is generated by
$\a_i\ot \a_j-\a_j\ot\a_i-h\Pi_{ij}$ where $\Pi_{ij}$ are constants. Since these elements are killed by $D_1$, we get $D_1(\JJ_{\Pi})\sub \Om^1\ot \JJ_{\Pi}$.
Similarly, $S(\Om^1)$ is spanned by the monomials in $(\a_i)$, which are killed by $D_1$, so $D_1(\hat{S}(\Om^1))\sub \Om^1\ot \hat{S}(\Om^1)$.

\noindent
(ii),(iii) By (i), we have an isomorphism of complexes
$$(\Om^\bullet\ot \hat{S}(\Om^1)[\![h]\!],D_0^s+D_1^s)\rTo{\sim} (\Om^\bullet\ot \WW_{\Pi}, \ov{D}_0+\ov{D}_1),$$
where $D_i^s$ is the restriction of $D_i$, $i=0,1$. Viewing $\hat{S}(\Om^1)$ as a completion of the symmetric algebra, it is easy to check that
$(\Om^\bullet\ot \hat{S}(\Om^1),D_0^s)$ is the completion of the Koszul complex ${\bigwedge}^\bullet (V)\ot S^\bullet(V)$ for $V=\Om^1$. Thus, we can
consider the standard homotopy 
$$H_i:{\bigwedge}^i(V)\ot S^j(V)\to {\bigwedge}^{i-1}(V)\ot S^{j+1}(V), \ i\ge 1,$$ 
of the positive degree terms of the Koszul complex given by the (suitably rescaled) action of the Euler vector field.
By \cite[Lem.\ 2.3.7]{P-Tu}, the operator $\id+HD_1^s$ extends to a well defined automorphism of $\Om^\bullet\ot \hat{S}(\Om^1)$ such that 
$$D_0^s+D_1^s=(\id+HD_1^s)^{-1}D_0^s(\id+HD_1^s),$$
which induces an isomorphism between the $0$th cohomology of $D_0^s+D_1^s$ and that of of $D_0^s$, which is identified with $\OO_X[\![h]\!]$.
This gives a $\C[\![h]\!]$-linear isomorphism
\begin{equation}\label{kappa-isom}
\kappa:\OO_X[\![h]\!]\rTo{\sim} \ker(D_0^s+D_1^s)\sub\hat{S}(\Om^1)[\![h]\!].
\end{equation}
Note that the identification \eqref{WP-Sym-eq} induces a product $*$ on $\hat{S}(\Om^1)[\![h]\!]$, so that $\ker(D_0^s+D_1^s)$ is a $\C[\![h]\!]$-subalgebra.
Namely, for $1$-forms $\om_1,\om_2$ one has 
$$\om_1*\om_2=\frac{1}{2}(\om_1\ot \om_2+\om_2\ot \om_1)+\frac{h}{2}\lan\Pi,\om_1\we\om_2\ran.$$

Explicitly, if $(x_1,\ldots,x_n)$ is a formal coordinate system as in (iii), then \eqref{kappa-isom}
takes $f\in \OO_X$ to
$$\kappa(f):=f+\sum_{m\ge 1, i_1,\ldots,i_m} \frac{1}{m!}\pa_{i_1}\ldots\pa_{i_m}f\cdot dx_{i_1}\ot\ldots\ot dx_{i_m}.$$
It is well known (see e.g., \cite{Fedosov}) that the above identification \eqref{kappa-isom}
satisfies $\kappa(f)*\kappa(g)=\kappa(f*g)$, where $f*g$ is given by \eqref{Moyal-eq}. 
\end{proof}

We will refer to the deformation quantization $\DD_{\Pi,\nabla}$ above as the {\it Moyal deformation quantization} associated with $(\Pi,\nabla)$.

\subsection{Homomorphisms from NC-smooth thickenings to deformation quantizations}

It is natural to ask whether there exist any homomorphisms from NC-smooth thickenings of $X$ to deformation quantizations of $X$.
The answer is always positive in the affine case as the following result shows. Recall that in the affine case an NC-smooth thickening always exists and
is unique up to isomorphism 
(see \cite{Kapranov}).

\begin{prop}\label{NC-to-quant-affine-prop} 
Let $X=\Spec(A)$ be a smooth affine variety, and let $A^{NC}$ be the NC-smooth completion of $A$.
Then for any deformation quantization $D$ of $A$, there exists a homomorphism $A^{NC}\to D$ compatible with the filtrations
$(F^nA^{NC})$ and $(h^nD)$, and inducing 
the identity $A^{NC}/F^1=A\to A=D/hD$. For any such homomorphism, the induced map
$$\Om^2_A\simeq F^1/F^2\to hD/h^2D=hA$$
is given by the Poisson bivector associated with $D$.
\end{prop}

\begin{proof}
It suffices to construct recursively a collection of compatible homomorphisms $\phi_n:A^{NC}/F^nA^{NC}\to D/h^nD$, starting with the identity map
$A^{NC}/F^1\to D/hD$. Assume we already constructed a homomorphism $\phi_{n-1}:A^{NC}/F^{n-1}A^{NC}\to D/h^{n-1}D$ compatible with the filtrations.
Now we use the fact that $A^{NC}/F^nA^{NC}$ is the universal square zero extension of $A^{NC}/F^{n-1}A^{NC}$ by a central bimodule (see
\cite[Prop.\ (1.6.2)]{Kapranov}). This implies that $A^{NC}/F^nA^{NC}$ maps to any other central extension of $A^{NC}/F^{n-1}A^{NC}$ (see
\cite[Prop.\ (1.3.8)]{Kapranov}).
Now $D/h^nD$ is a central extension of $D/h^{n-1}D$ by $D/hD$, and we can pull it back to get a central extension of $A^{NC}/F^{n-1}A^{NC}$. 
Applying the universality we get the required homomorphism $\phi_n$.

The last assertion is obtained by considering the commutator maps 
$$F^0/F^1\times F^0/F^1\to F^1/F^2, \ \ D/hD\times D/hD\to hD/h^2D,$$
and observing that the first one is identified with $(f,g)\mapsto df\we dg$ and the second is the Poisson bracket associated with $D$.
\end{proof}

Now we consider the global setting for the NC-smooth thickening associated with a torsion-free flat connection.

\begin{prop}\label{NC-to-quant-prop} 
Let $\OO_X^{NC}$ be the NC-smooth thickening of $X$ associated with a torsion-free flat connection 
$\nabla$ on $X$. 
Let $\Pi$ be a $\nabla$-horizontal Poisson bivector on $X$,
and let $\DD=\DD_{\Pi,\nabla}$ be the Moyal deformation quantization algebra associated with $(\Pi,\nabla)$ (see Proposition \ref{Moyal-prop}).
Then there exists a 
homomorphism
\begin{equation}\label{NC-to-def-eq}
\OO_X^{NC}\to \DD,
\end{equation}
compatible with the filtrations $(F^n=F^n\OO_X^{NC})$ and $(h^n\DD)$, such that the induced map $\OO_X=F^0/F^1\to \DD/h\DD=\OO_X$ is the identity,
and the induced map $\Om^2_X=F^1/F^2\to h\DD/h^2\DD=h\cdot \OO_X$ is given by $h\Pi$.
\end{prop}

\begin{proof} We use the dg-constructions of Sec.\ \ref{dg-constr-sec} defining $\OO_X^{NC}$ and $\DD_{\Pi,\nabla}$.
The natural projection gives a homomorphism of sheaves of dg-algebras
$$(\Om_X^\bullet\ot_{\OO_X}\hat{T}(\Om^1_X), D_0+D_1)\to (\Om_X^\bullet\ot_{\OO_X}\WW_{\Pi}, \ov{D}_0+\ov{D}_1).$$
Passing to the $0$th cohomology sheaves we get a homomorphism \eqref{NC-to-def-eq}.
The identification of $F^0/F^1$ (resp., $\DD/h\DD$) with $\OO_X$ is induced by the natural projection $\hat{T}(\Om^1)\to \OO_X$ 
(resp., $\WW_{\Pi}\to \WW_{\Pi}/h\WW_{\Pi}\to \OO_X$). This implies the assertion about the induced map on $F^0/F^1$. The assertion about the induced map
on $F^1/F^2$ follows by considering commutators, as in the proof of Proposition \ref{NC-to-quant-affine-prop}.
\end{proof}


\begin{cor} Let $E$ be an object of $D^b(A)$, where $A$ is an abelian variety, such that the Fourier-Mukai transform of $E$ has support of dimension $\le 1$.
Then $E$ extends to a perfect complex over the Moyal deformation quantization of any constant Poisson bivector $\Pi$ on $A$.
\end{cor}

\begin{proof}
By Theorem \ref{NC-FM-thm}, there exists an extension of $E\in D^b(A)$ to a perfect complex $\wt{E}$ of $\OO_A^{NC}$-modules.
Now consider a homomorphism from $\OO_A^{NC}$ to the Moyal quantization $\AA_{\Pi}$ of $\OO_A$ associated
with $\Pi$. Then $E_{\Pi}=\AA_{\Pi}\ot_{\OO_A^{NC}} \wt{E}$ is a perfect complex of $\AA_{\Pi}$-modules extending $E$.
\end{proof}

\begin{exs} 1. Let $C$ be a smooth genus $2$ curve, and let $i:C\to J$ denote a standard embedding into the Jacobian. Since $J$ is $2$-dimensional, there is a unique nonzero constant Poisson bivector $\Pi$ on $J$, up to rescaling. Then for a line bundle $L$ on $C$, the coherent sheaf $i_*L$ extends to a module over the Moyal deformation quantization $\AA_{\Pi}$ of $\Pi$,
flat over $\C[\![h]\!]$, if and only if
$\deg(L)=1$ (by the main result of \cite{BGKP}). It is easy to see that this is equivalent to the condition that the Fourier transform of $i_*L$ has support of dimension $\le 1$.
Thus, by Theorem \ref{NC-FM-thm}, we get that $i_*L$ deforms to a module over $\AA_{\Pi}$ if and only if it extends to a perfect complex over $\OO_J^{NC}$.

\noindent
2. More generally, let $A$ be an abelian surface, and let $i:C\hra A$ be a smooth connected divisor, which is a curve of genus $g\ge 2$, and let $L$ be a line bundle on $C$.
Then still by \cite{BGKP}, the sheaf $i_*L$ extends to a module over the Moyal deformation quantization of a nonzero constant Poisson bivector $\Pi$ on $A$,
flat over $\C[\![h]\!]$, if and only if $\deg(L)=g-1$.
Assuming that $\deg(L)=g-1$, we get that the Fourier transform of $L$ on the Jacobian $J_C$ has the form $\FF[-1]$, where $\FF$ is a coherent sheaf supported sheaf-theoretically
on the theta-divisor $\Th_L\sub J_C$ and of generic rank $1$ on $\Th_L$. The Fourier transform of $i_*(L)$ on $\hat{A}$ is $Lf^*\FF[-1]$, where $f:\hat{A}\to J_C$
is induced by $i$. Thus, if $f(\hat{A})$ is not contained in $\Th_L$, this Fourier transform is supported in dimension $\le 1$, which implies by Theorem \ref{NC-FM-thm} that $i_*L$ extends to the standard NC-smooth thickening $A^{NC}$ of $A$. 
The simplest example when $f(\hat{A})$ is contained in $\Theta_L$ is when $C\in |M^2|$, where $M$ is a principal polarization on $A$, and $L=M|_C$.
Checking whether $i_*L$ extends to $A^{NC}$ in this case would be a good testing case for Conjecture B (see \cite{AC} for other potential similar examples).
\end{exs}

\subsection{First order deformation quantizations of abelian varieties and the theta divisor}

One can also ask whether a  version of the Conjecture A holds in which the NC-smooth thickening
of $A$ is replaced by the deformation quantization with respect to a generic constant Poisson structure on $A$.
We are going to show that the condition of extendability to
the 1st order deformation quantization implies that $(A,\Th)$ is in the Andreotti-Mayer locus.

Let $\Pi\in H^0(X,\bigwedge^2 T_X)$ be a global bivector field on a smooth variety $X$. 
By the {\it standard 1st order deformation of $X$ associated with $\Pi$} we mean the sheaf of algebras
$\OO^{\Pi}_{\le 1}=\OO_X[h]/h^2$ with the multiplication 
$$(f_0+f_1h)\star (g_0+g_1h)=f_0g_0+(\{f_0,g_0\}+f_0g_1+f_1g_0)h,$$
where $\{f_0,g_0\}=\lan \Pi, df_0\wedge dg_0\ran$.

\begin{prop}\label{1st-order-def-qu-prop} 
Let $(A,\Theta)$ be a principally polarized abelian variety of dimension $g\ge 4$, such that $\Theta$ is irreducible. There exists a countable union $Z$ of proper linear subspaces in 
$H^0(A,\bigwedge^2 T_A)$, such that any bivector $\Pi\in H^0(A,\bigwedge^2 T_A)\setminus Z$
has the following property. 
Suppose for some line bundle $L$ on $\Th':=\Th\setminus \Sing(\Th)$ the coherent sheaf $i_*L$ on $A':=A\setminus\Sing(\Th)$, where $i:\Th'\to A'$ is the embedding,
extends to a module over the standard 1st order deformation of $A'$ associated with $\Pi$, flat over $\C[h]/h^2$. Then $\dim\Sing(\Th)\ge g-4$.
\end{prop}

\Pf . Assume $\dim\Sing(\Th)<g-4$ and let us show that this leads to a contradiction.
Since any divisor is coisotropic, the morphism
$\eta=\eta_{\Pi}:\Om^1_A\to T_A$ associated with $\Pi$ induces a morphism
$\eta_{\Th}:\Om^1_{\Th}\to N:=\OO_\Th(\Th)$, so that we have a commutative  diagram
\begin{equation}\label{forms-comm-diag}
\begin{diagram}
\Om^1_A &\rTo{\eta}& T_A\\
\dTo{} &&\dTo{}\\
\Om^1_\Th &\rTo{\eta_\Th}& N
\end{diagram}
\end{equation}
Let us also set $N':=N|_{\Th'}$, $\eta_{\Th'}:=\eta_{\Th}|_{\Th'}$.

We will use the result of Baranovsky-Ginzburg-Pecharich \cite[Thm.\ 7]{BGP} which states
that the existence of the extension in question implies the following equation in $H^1(\Th',N')$:
\begin{equation}\label{BGP-constraint}
\eta_{\Th'} [-c_1^H(N')+2c_1^H(L)]=0,
\end{equation}
where we view first Chern classes as elements in $H^1(\Th',\Om^1_{\Th'})$.
Since $\dim\Sing(\Th)<g-4$, the restriction maps \eqref{Pic-restr-maps} are isomorphisms, so the line bundle $L$ is a restriction of some line bundle $\wt{L}$ on $A$.
Thus, 
$$-c_1^H(N')+2c_1^H(L)=i^*(-c_1^H(\OO_A(\Th)|_{A'})+2c_1^H(\wt{L}|_{A'})),$$
where 
$$i^*:H^1({A'},\Om^1_{A'})\to H^1(\Th',\Om^1_{\Th'})$$
is the natural pull-back map. Using the restriction of the diagram \eqref{forms-comm-diag} to ${A'}$ we 
see that the composition of this pull-back map with the map
$$H^1(\Th',\Om^1_{\Th'})\rTo{\eta_{\Th'}} H^1(\Th',N')$$
is equal to the composition
$$H^1({A'},\Om^1_{A'})\rTo{\eta} H^1({A'},T_{A'})\to H^1(\Th',N').$$
Hence, the equation \eqref{BGP-constraint} means that the composition
\begin{equation}\label{long-composition-eq}
H^1(A,\Om^1_A)\to H^1({A'},\Om^1_{A'})\rTo{\eta|_{A'}} H^1({A'},T_{A'})\to H^1(\Th',N')
\end{equation}
maps the class $-c_1^H(\OO(\Th))+2c_1^H(\wt{L})$ to zero.
Since the class $c_1^H(\Th)\in H^2(A,\Z)$ is primitive, 
to arrive at a contradiction it suffices to prove that for $\eta=\eta_{\Pi}$ associated with a sufficiently generic  $\Pi$, 
the composition \eqref{long-composition-eq}, restricted to the subgroup of Chern classes of line bundles, is injective.
 
We can rewrite \eqref{long-composition-eq} as
$$H^1(A,\Om^1_A)\rTo{\eta} H^1(A,T_A)\to H^1(\Th, N)\to H^1(\Th',N').$$
Next we observe that since $\Th$ is Cohen-Macaulay and $\Sing(\Th)$ has codimension $\ge 3$ in $\Th$, 
the local cohomology $\und{H}_{\Sing(\Th)}^i(N)=0$ for $i\le 2$ vanish. This
implies that the map $H^1(\Th,N)\to H^1(\Th',N')$ is an isomorphism.
Thus, it is enough to prove injectivity of the composition
$$\NS(A)\to H^1(A,\Om^1_A)\rTo{\eta} H^1(A,T_A)\to H^1(\Th,N)$$
for sufficiently generic $\Pi$. From the exact sequence
$$0\to \OO_A\to \OO_A(\Th)\to N\to 0$$
we get a connecting homomorphism $H^1(\Th,N)\to H^2(A,\OO_A)$. It is enough to prove injectivity after post-composing with this map.

By Lemma \ref{divisor-class-lem} below, the composition 
$$H^1(A,T_A)\to H^1(\Th,N)\to H^2(A,\OO_A)$$
is given by the cup product with $c_1^H(\Th)\in H^1(A,\Om^1_A)$. Therefore, we are reduced to checking
injectivity of the composed map
\begin{equation}\label{NS-to-H2-map}
\NS(A)\to H^1(A,\Om^1_A)\rTo{\eta} H^1(A,T_A)\rTo{\cup c_1^H(\Th)} H^2(A,\OO)
\end{equation}
for sufficiently generic $\Pi$. 
Note that the last arrow can be written as the composition
$$H^1(A,T_A)=T_{A,0}\ot H^1(A,\OO)\rTo{\phi_\Th\ot \id} 
T_{\hat{A},0}\ot H^1(A,\OO)=H^1(A,\OO)^{\ot 2}\to 
{\bigwedge}^2 H^1(A,\OO)=H^2(A,\OO),$$
where $\phi_\Th: T_{A,0}\to T_{\hat{A},0}$ is the tangent map to the principal polarization isomorphism
$A\to\hat{A}$ associated with $\Th$. Thus, a class $\la\in\NS(A)$ is in the kernel of \eqref{NS-to-H2-map} 
if and only if 
\begin{equation}\label{eta-lambda-in}
\eta(\la)\in (\phi_\Th\ot\id)^{-1}(S^2H^1(A,\OO))\sub T_{A,0}\ot H^1(A,\OO).
\end{equation}
Since $\NS(A)$ is a finitely generated abelian group, it remains to check that for every nonzero 
$\la\in\NS(A)$ the condition \eqref{eta-lambda-in} defines a proper linear subspace of the space of 
skew-symmetric maps $\eta:T^*_{A,0}\to T_{A,0}$.

Set $V:=T_{A,0}$.
Recall that there is a natural identification $H^1(A,\OO)=\bar{V}^*$ (the dual of the complex conjugate
vector space to $V$), so that with every $\la\in\NS(A)$ the corresponding element
$$c_1(\la)\in H^1(A,\Om^1_A)=V^*\ot H^1(A,\OO)=V^*\ot\bar{V}^*$$
is represented by a Hermitian form $H=H_\la$ associated with $\la$, that is $H(v_1,v_2)$ is $\C$-linear in 
$v_1$, $\C$-antilinear in $v_2$ and satisfies $H(v_2,v_1)=\ov{H(v_1,v_2)}$.
Let $H_0(v_1,v_2)$ be the Hermitian form associated with the class of $\Theta$. Then we can write
any other Hermitian form uniquely as $H(v_1,v_2)=H_0(Qv_1,v_2)$, for a $\C$-linear operator $Q:V\to V$ which is 
self-adjoint with respect to $H_0$. Also, we can use the identification $\bar{V}\to V^*$ given by $H_0$ to 
view the operator $\eta:V^*\to V$ as a ($\C$-linear) operator $\psi:\bar{V}\to V$. Then the
condition \eqref{eta-lambda-in} is equivalent to the condition that the form $H(\psi v_1,v_2)$ is symmetric
(where $H=H_\la$). In terms of the operator $Q$ this condition becomes
$Q\psi=-\psi\ov{Q}$. If we choose a $\C$-basis of $V$, orthonormal with respect to $H_0$ then the matrices
corresponding to $Q$ and $\psi$ satisfy $\ov{Q}=Q^T$ and $\psi^T=-\psi$. Thus, the condition
\eqref{eta-lambda-in} is equivalent to $Q\psi=-\psi Q^T$. It remains to
observe that for every nonzero $n\times n$-matrix $Q$ (where $n>2$), there exists a skew-symmetric matrix $\psi$ 
such that $Q\psi+\psi Q^T\neq 0$ (this can be checked by taking $\psi$ of the form $E_{ij}-E_{ji}$
where $E_{ij}$ are elementary matrices).
\ed

\begin{lem}\label{divisor-class-lem} 
Let $D\sub X$ be an effective divisor in a smooth variety. Then the composition 
$$H^1(X,?_X)\to H^1(X,\OO_D(D))\to H^2(X,\OO),$$
where the second map comes from the exact sequence $0\to \OO_X\to \OO_X(D)\to \OO_D(D)\to 0$,
is equal to the cup product with the class $c_1^H(D)\in H^1(X,\Om^1_X)$.
\end{lem}

\Pf . We have a commutative diagram
\begin{diagram}
0\to &\OO_X &\rTo{}& D_{\le 1}(\OO_X(D))&\rTo{}& T_X&\to 0\\
&\dTo{\id}&&\dTo{}&&\dTo{}\\
0\to &\OO_X &\rTo{}& \OO(D)&\rTo{}& \OO_D(D)&\to 0\\
\end{diagram}
where $D_{\le 1}(\OO_X(D))$ is the sheaf of differential operators $\OO_X(D)\to \OO_X(D)$ of order $\le 1$,
and the middle vertical arrow sends an operator $D$ to $D\cdot 1$, where $1\in\OO_X\sub\OO_X(D)$.
It remains to use the fact that the class of the extension given by the first row of the diagram is
exactly $c_1^H(D)$.
\ed

\begin{rem} There is a homomorphism $\OO^{1-NC}_X\to\OO^{\Pi}_{\le 1}$
from the standard 1-smooth thickening 
to the 1st order deformation associated with any bivector $\Pi$. If we start with a module of finite projective dimension over $\OO^{1-NC}_X$ by extending scalars
we get a module $M$ of finite projective dimension
over $\OO^{\Pi}_{\le 1}$. We claim that any such module is automatically flat
over $\C[h]/h^2$. Indeed, we have to check that the infinite periodic complex 
$$\ldots\to M\rTo{h} M\rTo{h}M\to\ldots$$
is exact, and this follows immediately from the fact that $M$ has a finite resolution by modules for which this 
exactness holds.
Thus, from Proposition \ref{1st-order-def-qu-prop} we get another proof of Corollary \ref{ch-2-cor} for $\FF$ of the form $j^{ns}_*L$ (as in
Theorem \ref{Jac-thm}).
\end{rem}




\subsection{Example with the standard $1$-smooth thickening of an abelian surface}
\label{surface-ex-sec}

For every abelian surface $A$ we will give an example of an object $E\in D^b(A)$ such that
the Fourier transform of $E$ has full support but $E$ extends to the standard 1-smooth thickening of $A$.
This shows that in Conjecture B it is not sufficient to look at the first order NC-thickening.

First, we observe that trivializing $\om_A$, we can view the standard $\om_A$-valued bracket $\{f,g\}=df\we dg$ as a Poisson bracket $\Pi$ on $\OO_A$.
Hence, we can identify the sheaf $\OO_A^{1-NC}$
with the $1$st order deformation of $\OO_A$ associated with $\Pi$.

\begin{lem}\label{H2-lem} 
Let $\a\in H^2(A,\OO_A)$ be a nonzero element. 
Let $F:=\Cone(\a:\OO_A\to \OO_A[2])$. Then the morphism $\a:F\to F[2]$ is zero in $D^b(A)$.
\end{lem}

\begin{proof}
Working with a minimal $A_\infty$-category corresponding to $D^b(A)$, we can represent $F[-1]$ by the twisted complex $[\OO_A\rTo{\a}\OO_A[1]]$
(where the differential has degree $1$). Then the morphism $\a[-1]:F[-1]\to F[1]$ is represented by the obvious degree $0$ chain map 
$$[\OO_A\rTo{\a}\OO_A[1]]\to [\OO_A[2]\rTo{\a}\OO_A[3]]$$
with the components given by $\a$ (there can be no nontrivial component $\OO_A\to \OO_A[3]$). But this chain map is homotopic to zero,
which implies the assertion.
\end{proof}


Our proof will use the extension of Fourier-Mukai transform to the equivalence between derived categories of first order deformations of $\Coh(A)$ and $\Coh(\hat{A})$,
constructed in \cite{Toda} (it has been extended to certain higher order deformations in \cite{BBP}).

\begin{prop} Let $e\in A$ be the neutral element, and let $u\in \Ext^2_A(\OO_e,\OO_e)$ denote a nonzero element.
Then the object $G:=\Cone(u:\OO_e\to \OO_e[2])$ of $D^b(A)$ extends to a perfect complex in $D^b(A^{1-NC})$.
\end{prop}

\begin{proof}
By the result of Toda, \cite[Thm.\ 1.1]{Toda}, it is enough to prove that the Fourier transform $F$ of $G$ extends to a perfect object of the derived category $D^b(\Coh^\a(\hat{A}))$,
where $\Coh^\a(\hat{A})$ is the first order deformation of $\Coh(\hat{A})$ associated with a nonzero class $\a\in H^2(\hat{A},\OO_{\hat{A}})$. Furthermore, 
by \cite[Prop.\ 6.1]{Toda}, it is enough
to check that the morphism $\a:F\to F[2]$ vanishes. But this is proved in Lemma \ref{H2-lem}.
\end{proof}

\end{document}